\documentclass[a4paper,11pt]{article}

\usepackage{color,url}
\usepackage{enumitem}
\usepackage{graphicx}

\usepackage{psfrag}
\usepackage{epic}
\usepackage{amssymb}
\usepackage{epsfig}
\usepackage{pstricks}
\usepackage{stmaryrd,url}
\usepackage{tikz,amsmath}
\usepackage{diagbox}
\usepackage{framed,multirow}
\usepackage{epstopdf}
\usepackage{slashbox}
\usepackage{comment}

\usepackage{amsmath}  
\usepackage{amssymb}   
\usepackage{amsthm} 

\usepackage[normalem]{ulem}

\numberwithin{equation}{section}

\usepackage[top=1.0in,bottom=1.0in,left=1.0in,right=1.0in]{geometry}
\usepackage[font=scriptsize,labelfont=bf,width=0.85\textwidth]{caption}
\usepackage{subcaption}
\usepackage[colorlinks=true,linkcolor=blue]{hyperref}
\usepackage{makecell}
\usepackage{framed,multirow}
\usepackage{diagbox}

\usepackage{algorithmic,algorithm,xpatch}
\xpatchcmd{\algorithmic}{\setcounter}{\algorithmicfont\setcounter}{}{}
\providecommand{\algorithmicfont}{}

\newlength{\bibitemsep}
\newlength{\bibparskip}
\let\oldthebibliography\thebibliography
\renewcommand\thebibliography[1]{%
  \oldthebibliography{#1}%
  \setlength{\parskip}{\bibitemsep}%
  \setlength{\itemsep}{\bibparskip}%
}

\newtheorem{truth}{Theorem}

\newtheorem{remark}{Remark}

\newtheorem{assumption}{Assumption}

\title{Discontinuous Galerkin methods for a first-order semi-linear hyperbolic continuum model of a topological resonator dimer array
}

\begin{document}
\author{
Qiang Du
\thanks{Department of Applied Physics and Applied Mathematics, Columbia University, New York, NY 10027, USA. Email: qd2125@columbia.edu} \and  Huaiyu Li 
\thanks{Department of Applied Physics and Applied Mathematics, Columbia University, New York, NY 10027, USA. Email: hl3002@columbia.edu} \and
Michael Weinstein
\thanks{Department of Applied Physics and Applied Mathematics, Columbia University, New York, NY 10027, USA. Email: miw2103@columbia.edu} \and
Lu Zhang
\thanks{Department of Applied Physics and Applied Mathematics, Columbia University, New York, NY 10027, USA. Email: lz2784@columbia.edu} 
}

\maketitle

\begin{abstract}
{We present discontinuous Galerkin (DG) methods for solving a first-order semi-linear hyperbolic system, which was originally proposed as a continuum model for a one-dimensional dimer lattice of topological resonators. We examine the energy-conserving or energy-dissipating property in relation to the choices of simple, mesh-independent numerical fluxes. We demonstrate that, with certain numerical flux choices, our DG method achieves optimal convergence in the $L^2$ norm. We provide numerical experiments that validate and illustrate the effectiveness of our proposed numerical methods}
\end{abstract}

\textbf{Keywords}: discontinuous Galerkin,
first-order hyperbolic system, stability, error estimates, {topological resonators, dynamics of coherent structure} 

\textbf{AMS subject}: 65M12, 65M60, {35L40, 35C07, 78A40}

\section{Introduction}
There is great interest in the robust transport of energy in naturally occurring and engineered materials. A very active direction is the field of topological materials, which  emerged from the recognition that materials with topologically nontrivial (Floquet) band structure when perturbed, for example, by extended defects (domain walls, terminated along line defects,\dots), support robust edge (interface) states whose properties are stable against large but localized perturbations of the system. 
 The role of band structure topology in wave physics was first recognized in the context of condensed matter physics, e.g., the integer quantum Hall effect \cite{klitz80}, and topological insulators \cite{has10}.

There is very wide interest in technologies based on topologically protected states due to the potential for extraordinarily robust energy and information transfer in communications and computing. This has motivated the exploration of
topological wave phenomena observed in engineered metamaterial systems in photonics, acoustics, electronics, elasticity, and mechanics
\cite{haldane2008possible,elastic,photon,lu2014topological,acoustic,hadad2017solitons,chaunsali2019self,huber2016topological} which, in the regime of linear phenomena, are characterized by a linear band structure. A simple discrete model, which has been very influential, is the dimer model emerging from the work of Su-Schrieffer-Heeger (SSH) \cite{heeger1988solitons,su1979solitons} on the Peierls instability. Such systems can naturally be probed in the nonlinear regime via strong excitation. So it is of great interest to study whether topological properties persist in the regime where nonlinear effects are present and whether perhaps different topological phenomena emerge \cite{Alu17, theo19, Theo21, delplace, kiv20}. 

 With a view toward exploring topologically protected states in the nonlinear systems, Haddad, Vitelli, and Alu (HVA) \cite{hadad2017solitons,smirnova2020nonlinear} 
  introduced a discrete and {\it nonlinear} variant of SSH,
  modeling a spatially periodic nonlinear dimer array of electrical circuits. In each period cell of the array are two nonlinear circuit oscillators, and cells are linearly coupled to their nearest neighbors. Numerical simulations in the propagation of traveling  kink-like and pulse-like states. The formal continuum of the limit of the HVA discrete system is a coupled system of first-order semi-linear hyperbolic partial differential equations. After scaling, the system can be written in the form:
  \begin{equation}\label{original_system}
\begin{aligned}
    \partial_t b_1 &= \partial_x b_2 + N\left(b_1, b_2\right)b_2,\\
    \partial_t b_2 &= \partial_x b_1 - N\left(b_1, b_2\right)b_1.
\end{aligned}
\end{equation}
Here, $b_1(x,t)$ and $b_2(x,t)$ are continuum limits of field amplitudes (which modulate a $\pi-$ quasimomentum Bloch modes of the underlying SSH model) at time $t$ for the two nonlinearly coupled circuits in the dimer at location $x$. After appropriate normalization, the nonlinearity can be taken of the form: $N(b_1,b_2)=N(\sqrt{b_1^2+b_2^2})$ and such that
\begin{equation}\label{N_cond}
N(0)=1,\quad N(1)=0,\ N^\prime(r)<0, \quad N(r)\to N_\infty<0\ \textrm{as}\  r\to+\infty.
\end{equation}

{The continuum model is shown to have a rich family of traveling wave solutions (TWSs): subsonic (speed $|c|<1$) and supersonic (speed $|c|>1$) pulses, and subsonic kinks. Here, $|c|=1 $
 is the maximum group speed of linearized perturbation against spatially-uniform equilibria. The exponential stability of supersonic pulses in an appropriate weighted function space is established in \cite{liweinstein23}; dark soliton-like supersonic pulses always outrun any compactly supported perturbations, however large. Subsonic pulses are exponentially unstable. The linearized analysis of \cite{liweinstein23} shows that kinks with cubic nonlinearity are linearly neutrally stable. Hence its nonlinear dynamics are far more subtle.  To investigate the nonlinear stability properties of kinks, we desire a stable and efficient scheme for the evolution of data initialized near kinks. }

In this work, we devote our effort to the numerical studies of the semi-linear hyperbolic system (\ref{original_system}). For such systems, there is plenty of literature on the numerical approximation, such as finite difference method \cite{sod1978survey,beam1976implicit}, finite volume method \cite{leveque2002finite,dumbser2007quadrature}, finite element method \cite{lesaint1973finite,cockburn2003enhanced}, spectral method \cite{gottlieb2001spectral,tadmor1989convergence}, and particle method \cite{engquist1989particle,raviart1985analysis}. Here, we focus on discontinuous Galerkin (DG) methods. DG method is a class of finite element methods using discontinuous piecewise polynomials for numerical solutions and test functions in spatial variables. The first DG method was introduced in $1973$ by Reed and Hill for the neutron transport equation \cite{reed1973triangular}. In the following year, Lasaint and Raviart \cite{lesaint1974finite} conducted the first systematic analysis of the proposed DG method and showed that the order of convergence in the $L^2$ norm is $q$ when using $q$-th order discontinuous finite element spaces. Later, Johnson, N\"{a}vert, and Pitk\"{a}ranta improved the analysis and proved an optimal error estimate, $q + \frac{1}{2}$, in the $L^2$ norm \cite{johnson1984finite}. Peterson \cite{peterson1991note} further showed that this error estimate is optimal and cannot be improved within the class of quasi-uniform meshes. In the late 90s, Cockburn and Shu proposed a class of total variation bounded DG methods for conservation laws \cite{Cockburnrk1991,cockburn1989tvb,cockburn1989tvb1,cockburn1990runge,cockburn1998runge}. Since the entropy satisfies the conditions on the numerical fluxes, these schemes are generally entropy stable. In \cite{chung2006optimal,chung2009optimal}, Chung and Engquist analyzed an energy-conserving mixed DG method with staggered grids for first-order acoustic wave equations. With different meshes for different variables, they obtained optimal error estimates even with unstructured meshes. We refer to \cite{cockburn2000development,shu2014discontinuous} and reference therein for the development and survey of DG methods. Over the past few decades, DG methods have  proven to be very efficient when solving first-order hyperbolic and nearly hyperbolic systems. Because of their attractive properties, such as arbitrary high-order accuracy, local time evolution, element-wise conservation, geometrical flexibility, hp-adaptivity, etc., they have been widely used to solve problems in many fields of science, engineering, and industry, such as computational fluid dynamics, computational acoustics, and computational magneto-hydrodynamics, see \cite{cockburn2018discontinuous,nguyen2011high,rueda2022entropy} and reference therein. While the system under consideration is new and involves special nonlinear couplings, it is of a hyperbolic/wave nature. The development of DG approximations is a natural choice for its numerical solution. It is well known that high-order methods have less dispersion and dissipation errors than their low-order counterparts, making them more favorable when simulating waves.

In this paper, we consider the first-order hyperbolic system with a nonlinear term satisfying the condition (\ref{N_cond}) and its spatial DG discretization to the system (\ref{system_characteristic}). After first transforming the system into one represented by the characteristic variables, it is easy to identify the direction of wave propagation on the one hand and intuitive to define numerical fluxes by following the idea of DG methods for hyperbolic equations on the other. With proper boundary conditions, the proposed semi-discrete scheme is energy-conserving or energy-dissipating, depending on the choices of simple, mesh-independent numerical fluxes. In addition, the proposed DG scheme admits optimal error estimates in the $L^2$ norm for suitably defined numerical fluxes. Moreover, numerical experiments show that high-order DG methods perform well in capturing the dynamic behaviors of the solutions compared to low-order methods, making the algorithmic development a promising tool to help further investigate the mathematical properties of the system and related physical phenomena.

The rest of this paper is organized as follows: Section \ref{sec:continum_model} presents the governing continuum model and its conservation property. Section \ref{Sec: DG_formulation} introduces the DG semi-discretization along with several interelement fluxes and proves the basic energy estimate. In Section \ref{sec:error_analysis}, we prove an optimal error estimate in the $L^2$ norm with particular numerical fluxes and present several numerical experiments to verify our theoretical findings in Section \ref{sec:simulations}. Last, we summarize our results in Section \ref{sec:conclusion} and point out potential areas for future research. 

\section{The continuum model}\label{sec:continum_model}
In this section, we introduce a first-order hyperbolic system for a pair of re-scaled variables in (\ref{original_system}) with a general nonlinear coupling term $N(\cdot, \cdot)$ satisfying the condition (\ref{N_cond}). A key conservation property is also presented.

\subsection{Model in characteristic variables and conservation property}
To make the system (\ref{original_system}) more amendable to mature techniques for hyperbolic equations when designing DG schemes for the system (\ref{original_system}), we introduce the characteristic variables
\begin{equation}\label{b2w}
\begin{pmatrix}w_1 \\ w_2\end{pmatrix} = A\begin{pmatrix}
b_1\\ 
b_2\end{pmatrix}, \quad A = \frac{1}{\sqrt{2}}
\begin{pmatrix}
1 & 1 \\ 1 & -1
\end{pmatrix}
\end{equation}
to yield the following first-order hyperbolic system 
\begin{equation}\label{system_characteristic}
\begin{aligned}
\partial_t w_1 &= \ \ \partial_x w_1 -N\left(w_1, w_2\right)w_2, \quad (x,t)\in \mathbb{R}\times(0,T), \\
\partial_t w_2 &= -\partial_x w_2 + N\left(w_1, w_2\right)w_1, \quad (x,t)\in \mathbb{R}\times(0,T),
\end{aligned}
\end{equation}
where $N(w_1, w_2)= N\Big(\frac{1}{\sqrt{2}}(w_1+w_2),\frac{1}{\sqrt{2}}(w_1 - w_2)\Big)$ retains its form as before, and the constant $T\in (0,+\infty)$ denotes the terminal time of interests. 

Multiplying the first equation in (\ref{system_characteristic}) by $w_1$, and the second equation in (\ref{system_characteristic}) by $w_2$, then integrating the resulting two equations over $\mathbb{R}$ yields
\begin{align}
    \int_{\mathbb{R}} w_1\partial_t w_1\ dx &=\ \  \int_{\mathbb{R}} \big( w_1\partial_x w_1 
    - N\left(w_1, w_2\right)w_2w_1\big)\ dx,\label{eq1_int}\\
    \int_{\mathbb{R}}w_2\partial_t w_2\ dx &= -\int_{\mathbb{R}} \big(w_2\partial_x w_2 - N\left(w_1, w_2\right)w_1w_2\big)\ dx.\label{eq2_int}
\end{align}
Next, define the energy quantity
\begin{equation}\label{energy}
    E(t) = \int_{\mathbb{R}}(w_1^2 + w_2^2)\ dx,
\end{equation}
and add (\ref{eq1_int})--(\ref{eq2_int}) together, we obtain
\begin{equation}\label{dEdt_original}
\frac{1}{2}\frac{dE}{dt} = \frac{1}{2}\frac{d}{dt}\int_{\mathbb{R}}(w_1^2 + w_2^2)\ dx = \frac{1}{2}\left(w_1^2(+\infty, t) - w_1^2(-\infty, t) - w_2^2(+\infty, t) + w_2^2(-\infty, t)\right).
\end{equation} 
Given the boundary conditions that $b_1(x, t), b_2(x,t)$ decay to some constants at $x = \pm\infty$, it is clear that $w_1(x, t), w_2(x,t)$ also decay to some constants. Moreover, if $b_1, b_2$ satisfy
\[b_1(+\infty, t)b_2(+\infty, t) - b_1(-\infty, t)b_2(-\infty, t) = 0,\]
we then have
\begin{equation}\label{bdry_cond_infty}
w_1^2(+\infty, t) - w_1^2(-\infty, t) - w_2^2(+\infty, t) + w_2^2(-\infty, t) = 0,
\end{equation}
namely, $\frac{d}{dt}\int_{\mathbb{R}}(w_1^2 + w_2^2)\ dx = 0$, which implies $E(t) = E(0)$. 

\section{Finite speed of propagation}

Now, for a given interval $(a,b)$ with $b>a$, for $t<\frac{b-a}{2}$, we define
\begin{equation}\label{energy-ab}
    E_{(a,b)}(t) = \int_{a+t}^{b-t}(w_1^2(x,t) + w_2^2(x,t))\ dx.
\end{equation}
Then using the PDE system, we get
\begin{eqnarray*}\label{dEdt_ab}
\frac{d}{dt}  E_{(a,b)}(t) &=& 
 2 \int_{a+t}^{b-t} 
 (w_1\partial_t w_1 + w_2\partial_t w_2)\ dx  - w_1^2(b-t,t)-w_2^2(b-t,t)\\
 && \qquad  - w_1^2(a+t,t)-w_2^2(a+t,t)\\
 &=&
  2 \int_{a+t}^{b-t} 
 (w_1\partial_x w_1  - N\left(w_1, w_2\right)w_1w_2 - w_2\partial_{x} w_2
 + N\left(w_1, w_2\right)w_2w_1 )\ dx  \\
 && \qquad  - w_1^2(b-t,t)-w_2^2(b-t,t) - w_1^2(a+t,t)-w_2^2(a+t,t)\\
  &=&   w_1^2(b-t,t)- w_1^2(a+t,t) - w_2^2(b-t,t) +w_2^2(a+t,t)\\
  &&\qquad
   - w_1^2(b-t,t)-w_2^2(b-t,t)- w_1^2(a+t,t)-w_2^2(a+t,t)\\
  &=&  -2  w_1^2(a+t,t)-2 w_2^2(b-t,t)  \leq 0.
\end{eqnarray*} 
This implies
$$ E_{(a,b)}(t) \leq E_{(a,b)}(0)= \int_{a}^{b}(w_1^2(x,0) + w_2^2(x,0))\ dx.
$$
Thus, for initial data with compact support in $(-A, A)$, we see that  the solution also has compact support $(-A-t, A+t)$ at any finite time $t$,
which leads to the finite speed of propagation.

One can apply a similar argument to perturbed solutions (with the help of the uniform Liphshitz continuity of the nonlinear term and the Gronwall inequality).

\section{Semi-discrete DG formulation}\label{Sec: DG_formulation}

To derive a DG formulation for the system (\ref{system_characteristic}), we first truncate the real line to a bounded interval $[x_a, x_b] \subseteq \mathbb{R}$, 
then discretize the computational domain $I := [x_a, x_b] \subseteq \mathbb{R}$ by non-overlapping elements $I_j$ such that $[x_a, x_b] = \cup_{j = 1}^N [x_{j-\frac{1}{2}}, x_{j+\frac{1}{2}}]$ with $I_j = (x_{j-\frac{1}{2}}, x_{j+\frac{1}{2}})$. In addition, denote $h_j := x_{j+\frac{1}{2}} - x_{j-\frac{1}{2}}$ and $h := \max_j h_j$, we impose a mesh regularity condition $\frac{h}{\min h_j} < \zeta$ with $\zeta > 0$ being fixed during the mesh refinement. On each element $I_j$, we approximate $(w_1,w_2)$ by $\left(w_1^h,w_2^h\right)$, each belonging to the following space of piecewise polynomial of degree $q$,
\[V_q^h = \{v^h(x,t), v^h(x,t)\in\mathcal{P}^q(I_j),\ q\geq 0,\ \forall j\}.\]
Let us further denote
\[\eta^{\pm}(x) := \lim_{\epsilon\rightarrow 0^{\pm}} \eta(x + \epsilon),\]
and introduce the conventional notations for the weighted averages and jumps at the element edge $x_{j+\frac{1}{2}}, j = 1,\cdots, N-1$ by
\[\{v^h\} := \alpha v^{h,-}(x_{j+\frac{1}{2}}) + (1-\alpha) v^{h,+}(x_{j+\frac{1}{2}}),\quad [v^h] := v^{h,-}(x_{j+\frac{1}{2}}) - v^{h,+}(x_{j+\frac{1}{2}}), \quad \alpha\in\mathbb{R},\]
respectively. Next, consider the discrete energy which is analogous to (\ref{energy}) in the element $I_j$,
\begin{equation}\label{dis_energy}
    E_{j}^h(t) = \frac{1}{2}\int_{I_j} \big( (w_1^h)^2 + (w_2^h)^2 \big)\ dx,
\end{equation}
and its time derivative
\begin{equation}\label{dis_dEdt}
    \frac{dE_{j}^h(t)}{dt} = \int_{I_j} 
    (w_1^hw_{1t}^h + w_2^hw_{2t}^h)\ dx.
\end{equation}

We then seek an approximation of the system (\ref{system_characteristic}) which is compatible with the discrete energy (\ref{dis_energy}) and its time derivative (\ref{dis_dEdt}). To this end, we choose $\phi, \psi \in V_h^q$, and test the first equation in (\ref{system_characteristic}) by $\phi$, and the second equation in (\ref{system_characteristic}) by $\psi$ to make residuals orthogonal to all test functions $\phi, \psi \in V_q^h$, resulting in the following local statement, (see, e.g., \cite{hesthaven2007nodal}), 
\begin{equation}\label{weak_form1}
\left\{
\begin{aligned}
    \int_{I_j}  ( w_{1t}^h \phi - w_{1x}^h \phi)\ dx + \sum_{k} \underline{\omega}_{k,j}N\big(w_1^h(x_{k,j}, t), w_2^h(x_{k,j}, t)\big)w_2^h(x_{k,j}, t)\phi(x_{k,j}) &= 0,\\
    \int_{I_j} (  w_{2t}^h \psi + w_{2x}^h \psi ) \ dx - \sum_{k} \underline{\omega}_{k,j}N\big(w_1^h(x_{k,j}, t), w_2^h(x_{k,j}, t)\big)w_1^h(x_{k,j}, t)\psi(x_{k,j}) &= 0,
\end{aligned} \right.
\end{equation}
where we have used a quadrature rule, satisfying the following Assumption \ref{assump}, with nodes $x_{k, j}$ in $I_j$ and weights $\underline{\omega}_{k,j} > 0$ to approximate the integration of the nonlinear terms involving $N\big(\cdot, \cdot\big)$ (see, e.g., \cite{appelo2020energy}). 
\begin{assumption}\label{assump}
The quadrature rule satisfies, $\forall I_j$,
    \begin{align}
    \sum_k \underline{\omega}_{k, j} \phi^2(x_{k,j}) - \int_{I_j} \phi^2\ d x &= 0,  \nonumber\\
    \sum_j\Big| \sum_k \underline{\omega}_{k, j}\phi(x_{k, j})g(x_{k, j}) - \int_{I_j} \phi f \ d x \Big| &\leq C_0h^{q+1}\|\phi\|_{L^2(I)}|f|_{H^{q+1}(I)},\nonumber
    \end{align}
    $\forall \phi \in \mathcal{P}^q(I_j)$ and $\forall f \in H^{q+1}(I)$. Here, the constant $C_0$ is independent of $h$ and $f$. 
\end{assumption}
The above quadrature evaluates the integrals of $w_{it}^h \phi\pm w_i^h \phi_x$ on $I_j$ exactly.
Further, an integration by parts to the system (\ref{weak_form1}) leads to the following DG scheme
\begin{multline}\label{DG1}
    \int_{I_j} ( w_{1t}^h \phi + w_1^h \phi_x )\ dx \\+ \sum_k \underline{\omega}_{k, j}N\big(w_1^h(x_{k,j}), w^h_2(x_{k,j})\big)w_2^h(x_{k,j})\phi(x_{k,j}) 
    =  \widehat{w_1^h}\phi^- \Big|_{j+\frac{1}{2}} - \widehat{w_1^h}\phi^+\Big|_{j-\frac{1}{2}},
    \end{multline}
    and
\begin{multline}\label{DG2}
    \int_{I_j} ( w_{2t}^h \psi - w_2^h \psi_x ) \ dx\\ - \sum_k \underline{\omega}_{k, j}N\big(w_1^h(x_{k,j}), w_2^h(x_{k,j})\big)w_1^h(x_{k,j})\psi(x_{k,j}) = \widetilde{w_2^h}\psi^+\Big|_{j-\frac{1}{2}}- \widetilde{w_2^h}\psi^- \Big|_{j+\frac{1}{2}},
\end{multline}
where we have omitted $t$ in $w^h_1(x_{k,j}), w^h_2(x_{k,j})$ for simplicity, and $\widehat{w_1^h}$ and $\widetilde{w_2^h}$ are numerical fluxes at the element boundaries which are imposed to guarantee the stability of the proposed numerical scheme. The specific choices of $\widehat{w_1^h}$ and $\widetilde{w_2^h}$ are discussed in the next subsection.

Note that by choosing $(\phi, \psi) = (w_1^h, w_2^h)$, and summing (\ref{DG1}) and (\ref{DG2}) together over all cells $I_j$, we get
\begin{multline}\label{discrete_dEdt}
    \frac{dE^h}{dt} = \sum_{j=1}^{N}  \left(\widehat{w_1^h} - \frac{1}{2}w_1^{h,-}\right)w_1^{h,-} \Big|_{j+\frac{1}{2}} - \left(\widehat{w_1^h} - \frac{1}{2}w_1^{h,+}\right)w_1^{h,+}\Big|_{j-\frac{1}{2}} \\- \left(\frac{1}{2}w_2^{h,+} - \widetilde{w_2^h}\right)w_2^{h,+}\Big|_{j-\frac{1}{2}} +\left(\frac{1}{2}w_2^{h,-}- \widetilde{w_2^h}\right)w_2^{h,-} \Big|_{j+\frac{1}{2}},
\end{multline}
where
\begin{equation}\label{eq:discrete_energy}
E^h = \sum_{j=1}^{N} E_j^h 
\end{equation}
with $E_j^h$ is defined in (\ref{dis_energy}).

\subsection{Fluxes}
To complete the DG formulations (\ref{DG1})--(\ref{DG2}) proposed above and generate a stable scheme, it is desirable to specify the numerical fluxes $\widehat{w_1^h}$ and $\widetilde{w_2^h}$ both at interelement boundaries and physical boundaries such that
\[ \frac{dE^h}{dt} \leq 0.\]
In particular, $\frac{dE^h}{dt} < 0$ leads to a dissipating scheme, and $\frac{dE^h}{dt} = 0$ yields a conservative scheme. 

\subsubsection{Interelement boundaries}
We first consider the interelement boundaries, that is, $x_{\frac{3}{2}}, x_{\frac{5}{2}}, \cdots, x_{N-\frac{1}{2}}$. From the expression in (\ref{discrete_dEdt}), to obtain a stable scheme, it suffices to control the terms
\begin{equation}\label{J}
    \begin{aligned}
J_j:= &\left(2\widehat{w_1^h}-w_1^{h,-}\right)w_1^{h,-} \Big|_{j+\frac{1}{2}} - \left(2\widehat{w_1^h}-w_1^{h,+}\right)w_1^{h,+}\Big|_{j+\frac{1}{2}} \\
 &+\left(w_2^{h,-} - 2\widetilde{w_2^h}\right)w_2^{h,-} \Big|_{j+\frac{1}{2}} - \left(w_2^{h,+}-2\widetilde{w_2^h}\right)w_2^{h,+}\Big|_{j+\frac{1}{2}} \leq 0,\quad j = 1,2,\cdots,N-1.
    \end{aligned}
\end{equation}
Particularly, when $J_j \leq 0$, $\frac{dE^h}{dt} \leq 0$, namely, we obtain a stable scheme. In this work, we introduce the following numerical fluxes at the interelement boundaries (see, e.g., \cite{hesthaven2007nodal}), 
\begin{equation}\label{flux}
\begin{aligned}
    \widehat{w_1^h} &= \{w_1^h\}  -\frac{1-\alpha_1}{2}[w_1^h] + \frac{\beta_1}{2}[w_2^h],\\
    \widetilde{w_2^h} &= \{w_2^h\} + \frac{1-\alpha_2}{2}[w_2^h] + \frac{\beta_2}{2}[w_1^h],
\end{aligned}
\end{equation}
where 
\begin{equation}\label{parameters_flux}
   0\leq \alpha_1, \alpha_2 \leq 1,\quad \beta_1, \beta_2\in\mathbb{R}, \quad \mbox{and} \quad -(1 - \max\{\alpha_1, \alpha_2\}) +\frac{|\beta_1 - \beta_2|}{2} \leq 0. 
\end{equation}
Plugging (\ref{flux}) into (\ref{J}), we have 
\begin{equation*}
\begin{aligned}
J_j = &-(1-\alpha_1)[w_1^h]^2 -(1-\alpha_2)[w_2^h]^2 + (\beta_1-\beta_2)[w_1^h][w_2^h]\\
\leq & -(1 - \max\{\alpha_1, \alpha_2\})\left([w_1^h]^2 + [w_2^h]^2\right) + \frac{|\beta_1 - \beta_2|}{2}\left([w_1^h]^2 + [w_2^h]^2\right)\\
=& \left(-(1 - \max\{\alpha_1, \alpha_2\}) +\frac{|\beta_1 - \beta_2|}{2} \right)\left([w_1^h]^2 + [w_2^h]^2\right).
\end{aligned}
\end{equation*}
when the conditions in (\ref{parameters_flux}) are satisfied, we have $J_j \leq 0$. In Particular, when $\beta_1 = \beta_2 = 0$ and $\alpha_1 = \alpha_2 = 1$, one can recover the commonly used \emph{central flux}, namely,
\begin{equation}\label{central_flux}
    \widehat{w_1^h} = \{w_1^h\}, \quad 
    \widetilde{w_2^h} = \{w_2^h\},
\end{equation}
which gives an energy conserving scheme with $J_j = 0$; when $\beta_1 = \beta_2 = 0, \alpha_1 = \alpha_2 = 0$, we have the so-called \emph{upwind flux}, that is,
\begin{equation}\label{upwind_flux}
    \widehat{w_1^h} = w_1^{h,+}, \quad \widetilde{w_2^h} = w_2^{h,-},
\end{equation}
which yields $J_i < 0$ and gives an energy dissipating scheme.

\subsubsection{Physical boundaries}
We next consider the approximation of the physical boundary conditions. First, note that the DG scheme (\ref{DG1})--(\ref{DG2}) is defined in the truncated domain $[x_a, x_b]$ for $w_1$ and $w_2$; second, the principal part of the hyperbolic system (\ref{system_characteristic}) indicates that the wave associated with $w_1/w_2$ propagates from the right/left to the left/right. This indicates that we need to supply an inflow boundary condition for $w_1/w_2$ at $x_b/x_a$. Here, for simplicity of analysis, we consider Dirichlet boundary conditions of
\begin{equation}\label{physical_bdry}
    w_1(x_b, t) = 0, \quad  w_2(x_a, t) = 0, \quad t \in (0, T]. 
\end{equation}
To approximate the physical boundary conditions (\ref{physical_bdry}), we choose 
\begin{equation}\label{flux_boundary}
    \widehat{w_1^h}(x_a,t) = w_1^h(x_a, t), \quad \widehat{w_1^h}(x_b,t) = 0, \quad \widetilde{w_2^h}(x_a, t) = 0, \quad \widetilde{w_2^h}(x_b, t) = w_2^h(x_b, t).
\end{equation}
This reflects the Dirichlet boundary condition of (\ref{physical_bdry}) and purely internal choices at the outflow. Then plug (\ref{flux_boundary}) into (\ref{discrete_dEdt}), and only consider the physical boundary points, namely, $x_{\frac{1}{2}}=x_a$ and $x_{N+\frac{1}{2}}=x_b$, we have the contribution to the discrete energy from the physical boundaries is given by
\begin{equation*}
\begin{aligned}
 &-\left(\widehat{w_1^h}-\frac{1}{2}w_1^{h,+}\right)w_1^{h,+}\Big|_{\frac{1}{2}}
- \left(\frac{1}{2}w_2^{h,+}-\widetilde{w_2^h}\right)w_2^{h,+}\Big|_{\frac{1}{2}} \\
&+  \left(\widehat{w_1^h}-\frac{1}{2}w_1^{h,-}\right)w_1^{h,-} \Big|_{N+\frac{1}{2}} 
    + \left(\frac{1}{2}w_2^{h,-} - \widetilde{w_2^h}\right)w_2^{h,-} \Big|_{N+\frac{1}{2}}\\
    = &-\frac{1}{2}\left((w_1^{h,+})^2\Big|_{\frac{1}{2}} + (w_2^{h,+})^2\Big|_{\frac{1}{2}}+ (w_1^{h,-})^2\Big|_{N+\frac{1}{2}} + (w_2^{h,-})^2\Big|_{N+\frac{1}{2}}\right) \leq 0.
\end{aligned}
\end{equation*}

We are now ready to establish the stability of the proposed DG scheme (\ref{DG1})--(\ref{DG2}) with the numerical fluxes (\ref{flux}) and (\ref{flux_boundary}).
    
\begin{truth}\label{thm1}
    Given the DG scheme (\ref{DG1})--(\ref{DG2}), the physical boundary conditions (\ref{physical_bdry}) and the numerical fluxes defined in (\ref{flux}) and (\ref{flux_boundary}), then the discrete energy $E^h(t)$ defined in (\ref{eq:discrete_energy}) 
    satisfies
    \[\begin{aligned}
    \frac{dE^h}{dt} =  \frac{1}{2}\sum_{j=1}^{N-1} &-(1-\alpha_1)[w_1^h]^2\Big|_{j+\frac{1}{2}} -(1-\alpha_2)[w_2^h]^2\Big|_{j+\frac{1}{2}} + (\beta_1-\beta_2)[w_1^h][w_2^h]\Big|_{j+\frac{1}{2}}\\
    &-\Big((w_1^{h,+})^2\Big|_{\frac{1}{2}}+ (w_2^{h,+})^2\Big|_{\frac{1}{2}}+ (w_1^{h,-})^2\Big|_{N+\frac{1}{2}} + (w_2^{h,-})^2\Big|_{N+\frac{1}{2}}\Big).
    \end{aligned}\]
    If the parameters $\alpha_1, \alpha_2, \beta_1$ and $\beta_2$ satisfy the conditions in (\ref{parameters_flux}), then
    \[E^h(t) \leq E^h(0).\]
\end{truth}
\begin{remark}
We note that when $\alpha_1 = \alpha_2 = 1$, $\beta_1 = \beta_2$, we do not get any contributions from the interelement boundaries. Moreover, for a periodic boundary condition, the discrete energy $E^h$ is constant as for the continuous system. Here, the conservation of the continuous system is obtained since (\ref{bdry_cond_infty}) is satisfied for a periodic case. In addition, for the numerical experiments conducted in Section \ref{sec:simulations}, we consider three different boundary conditions: the periodic boundary conditions, the Dirichlet boundary conditions (\ref{physical_bdry}), and the boundary conditions where $b_1$ and $b_2$ decay to some constants. For all three different types of boundary conditions, we obtain stable schemes.
\end{remark}

\section{Error analysis}\label{sec:error_analysis}
In this section, We proceed to derive error estimates of the DG scheme (\ref{DG1})-(\ref{DG2}) with the numerical fluxes (\ref{flux}) and (\ref{flux_boundary}) for the system (\ref{system_characteristic}). Denote 
  \begin{equation}\label{g1g2}
  z_1(u,v) := N(u,v)u, \quad\mbox{and}\quad z_2(u, v) := N(u, v)v,
  \end{equation}
  we consider the case where the nonlinear interaction $N(\cdot, \cdot)$ satisfies
  \begin{equation}\label{N_proper}
\left\|\nabla z_1\right\|_{L^\infty(I)} + \left\|\nabla z_2\right\|_{L^\infty(I)}
  \leq C_1,
  \end{equation}
 for a positive constant $C_1$, which is easily satisfied when the nonlinear coupling $N(\cdot, \cdot)$ satisfying (\ref{N_cond}). In Section \ref{sec:projection}, we review some projections and inequalities that are essential for our proof. The error estimates in the $L^2$-norm are given from Section \ref{sec:error_estimate} to Section \ref{sec:improved_convergence}. 
  
  \subsection{Projections}\label{sec:projection}
 We define the Gauss–Radau projections $P_h^{\pm}$ into $V_q^h$ such that for any $u \in H^{q+1}(I)$, $q \geq 1$ and $I_j = (x_{j-\frac{1}{2}}, x_{j+\frac{1}{2}}), j = 1, 2, \cdots, N$,
 \begin{align}
  &\int_{I_j} (P_h^{\pm}u - u)v_h\ dx = 0, \ \forall v_h\in \mathcal{P}^{q-1}(I_j),\nonumber\\
  & P_h^+u(x_{j-\frac{1}{2}}^+) = u(x_{j - \frac{1}{2}}),\ P_h^-u(x_{j+\frac{1}{2}}^-) = u(x_{j + \frac{1}{2}})\label{projection_q0}.
  \end{align}
  When $q = 0$, the Gauss–Radau projections $P_h^{\pm}$ are defined only by (\ref{projection_q0}). For each projection, we have the following inequality holds for any $u\in H^{q+1}(I)$ (see e.g., \cite{ciarlet2002finite}),
  \begin{equation}\label{projection}
  \|u - Q_h u\|_{L^2(I)} + h\|u - Q_hu\|_{L^\infty (I)} + h^{\frac{1}{2}}\|u - Q_hu\|_{L^2(\Gamma_h)} \leq Ch^{q+1},
  \end{equation}
  where $Q_h = P_h^{\pm}$, $\Gamma_h$ contains all the trace data, and $\|\eta\|_{L^2(\Gamma_h)} = \big(\sum_{j = 1}^{N} \eta^2(x_{j-\frac{1}{2}}) + \eta^2(x_{j+\frac{1}{2}})\big)^{\frac{1}{2}}$.
  
  \subsection{Error estimate with general fluxes (\ref{flux})}\label{sec:error_estimate}
  We are now ready to present error estimates for the DG scheme (\ref{DG1})-(\ref{DG2}) with the numerical fluxes defined in (\ref{flux}) and (\ref{flux_boundary}). Denote the errors by
  \begin{equation}\label{error_relations}
 \begin{aligned}
  &e_{w_1} := w_1 - w_1^h = w_1 - P_h^+ w_1 + P_h^+ w_1 - w_1^h =: \eta_{w_1} + \xi_{w_1},\\
  &e_{w_2} := w_2 - w_2^h = w_2 - P_h^- w_2 + P_h^- w_2 - w_2^h =: \eta_{w_2} + \xi_{w_2}.
  \end{aligned}
  \end{equation}
  
  We first consider the projection of the initial data, which is essential for the derivation of optimal convergence. Particularly, the initial data are chosen through
  \begin{equation}\label{initial_projection}
  w_1^h(x, 0) = P_h^+w_1(x, 0), \quad w_2^h(x, 0) = P_h^-w_2(x, 0),
  \end{equation}
  namely, $\xi_{w_1} = 0, \xi_{w_2} = 0$. Next, let us consider the numerical error energy
  \begin{equation}
      \mathcal{E}^h := \sum_j \frac{1}{2}\int_{I_j}  (\xi_{w_1}^2 + \xi_{w_2}^2)\ dx,
  \end{equation}
  and proceed to derive the estimate of $\mathcal{E}^h$. Since both the continuous solution $(w_1, w_2)$ and the DG solution $(w_1^h, w_2^h)$ satisfy the DG scheme (\ref{DG1})--(\ref{DG2}), we then have the following error equations for any $(\phi, \psi) \in V_h^q \times V_h^q$,
  \begin{multline}\label{error_eq1}
      \int_{I_j} e_{w_1t} \phi \ dx +\mathcal{B}_j^1(e_{w_1}, \phi) + \sum_k \underline{\omega}_{k,j} \Big(N\big(w_1(x_{k,j}), w_2(x_{k,j})\big)w_2(x_{k,j}) \\ - N\big(w_1^h(x_{k,j}), w_2^h(x_{k,j})\big)w_2^h(x_{k,j})\Big)\phi(x_{k,j})  = 0,
\end{multline}
\begin{multline}\label{error_eq2}
      \int_{I_j} e_{w_2t} \psi \ dx -\mathcal{B}_j^2(e_{w_2}, \psi) - \sum_k\underline{\omega}_{k,j} \Big(N\big(w_1(x_{k,j}), w_2(x_{k,j})\big)w_1(x_{k,j}) \\
      - N\big(w_1^h(x_{k,j}), w_2^h(x_{k,j})\big)w_1^h(x_{k,j})\Big)\psi(x_{k,j})  = 0,
  \end{multline}
  where 
  \begin{align}
  \mathcal{B}_j^1(e_{w_1}, \phi) &= \int_{I_j} e_{w_1}\phi_x\ dx - \left(\widehat{e}_{w_1}\phi^-\Big|_{j+\frac{1}{2}} - \widehat{e}_{w_1}\phi^+\Big|_{j-\frac{1}{2}}\right),\label{B1}\\
  \mathcal{B}_j^2(e_{w_2}, \psi) &= \int_{I_j} e_{w_2}\psi_x\ dx - \left(\widetilde{e}_{w_2}\psi^-\Big|_{j+\frac{1}{2}} - \widetilde{e}_{w_2}\psi^+\Big|_{j-\frac{1}{2}}\right).\label{B2}
  \end{align}
   Choosing $\phi = \xi_{w_1}$ and $\psi = \xi_{w_2}$ in (\ref{error_eq1})--(\ref{error_eq2}), then adding the resulting two equations, invoking the relation (\ref{error_relations}), and summing over all elements $I_j$ yields
  \begin{equation}\label{dxi_dt}
     \frac{1}{2} \frac{d}{dt} \sum_j \int_{I_j}
     (\xi_{w_1}^2 + \xi_{w_2}^2 )\ dx = \Lambda_1 + \Lambda_2 + \Lambda_3 + \Lambda_4 + \Lambda_5,
  \end{equation}
  where
  \begin{equation}\label{lambda}
  \begin{aligned}
      \Lambda_1 &:= \sum_j\int_{I_j} -  ( \eta_{w_1t}\xi_{w_1} + \eta_{w_2t}\xi_{w_2} )\ dx,\\
      \Lambda_2 &:=\sum_j -\mathcal{B}_j^1(\xi_{w_1}, \xi_{w_1}) + \mathcal{B}_j^2(\xi_{w_2}, \xi_{w_2}),\\
      \Lambda_3 &:= -\sum_{j,k}\underline{\omega}_{k,j} \Big(N\big(w_1(x_{k,j}), w_2(x_{k,j})\big)w_2(x_{k,j}) - N\big(w_1^h(x_{k,j}), w_2^h(x_{k,j})\big)w_2^h(x_{k,j})\Big)\xi_{w_1}(x_{k,j}),\\
      \Lambda_4 &:= \ \ \ \sum_{j,k}\underline{\omega}_{k,j}\Big(N\big(w_1(x_{k,j}), w_2(x_{k,j})\big)w_1(x_{k,j}) - N\big(w_1^h(x_{k,j}), w_2^h(x_{k,j})\big)w_1^h(x_{k,j})\Big)\xi_{w_2}(x_{k,j}),\\
      \Lambda_5 &:= \sum_j\int_{I_j} -\eta_{w_1}\xi_{w_1x} + \eta_{w_2}\xi_{w_2x}\ dx\\
      &\qquad + \left(\widehat{\eta}_{w_1}\xi_{w_1}^-\Big|_{j+\frac{1}{2}} - \widehat{\eta}_{w_1}\xi_{w_1}^+\Big|_{j-\frac{1}{2}}\right) - \left(\widetilde{\eta}_{w_2}\xi_{w_2}^-\Big|_{j+\frac{1}{2}} - \widetilde{\eta}_{w_2}\xi_{w_2}^+\Big|_{j-\frac{1}{2}}\right),
  \end{aligned}
  \end{equation}
  where $\mathcal{B}_j^1$ and $\mathcal{B}_j^2$ are defined in (\ref{B1}) and (\ref{B2}), respectively.
  
  In what follows, we assume that the solution is sufficiently smooth up to some time, $T$. In addition, we denote by $C$ a generic positive constant which is independent of the cell length $h$ for a shape-regular mesh but may vary from line to line. Then we have the following error estimate.

  \begin{truth}\label{truth2}
  Let $(w_1, w_2)$ be a smooth solution of system (\ref{system_characteristic}) with the physical boundary condition (\ref{physical_bdry}), and $(w_1^h, w_2^h) \in V_h^q\times V_h^q, q \geq 0$ be the numerical solution of the DG scheme (\ref{DG1})-(\ref{DG2}) with the numerical fluxes defined in (\ref{flux}) and (\ref{flux_boundary}), and the nonlinear interaction $N(\cdot,\cdot)$ satisfies the condition (\ref{N_proper}). Further, the approximation of the initial data is computed based on (\ref{initial_projection}). Then there exists a positive constant $\mathcal{C}$ depending only on
  the degree $q$, the bound given in (\ref{N_proper}) and the shape regularity of the mesh, but independent of $h$, such that
  \begin{equation}\label{error_estimate}
  \|e_{w_1}(x, T)\|_{L^2(I)} + \|e_{w_2}(x, T)\|_{L^2(I)} \leq \mathcal{C}h^q.
  \end{equation}
  \end{truth}
\begin{proof}
We first estimate the volume integral $\Lambda_1$. Based on the property (\ref{projection}) of the projection operators, we obtain
  \begin{equation}\label{lambda_1}
    |\Lambda_1| \leq Ch^{q+1} \left(\|\xi_{w_1}\|_{L^2(I)} + \|\xi_{w_2}\|_{L^2(I)}\right). 
  \end{equation}
 The same analysis as in the derivation of Theorem \ref{thm1} leads us to
  \begin{equation}
      \Lambda_2 \leq 0.
  \end{equation}
  As for the estimation of $\Lambda_5$, from the definition of the projection operators $P_h^{\pm}$ , the numerical fluxes (\ref{flux}) and (\ref{flux_boundary}), and the trace inequality, we obtain
  \begin{equation}
  \Lambda_5 \leq Ch^q\left(\|\xi_{w_1}\|_{L^2(I)} + \|\xi_{w_2}\|_{L^2(I)}\right).
  \end{equation}
  To estimate $\Lambda_3$, we recall the definition of the functions $z_2(\cdot, \cdot)$ in (\ref{g1g2}) to rewrite $\Lambda_3$ as
  \[\Lambda_3 = -\sum_{j,k}\underline{\omega}_{k,j} \Big(z_2\big(w_1(x_{k,j}), w_2(x_{k,j})\big) - z_2\big(w_1^h(x_{k,j}), w_2^h(x_{k,j})\big)\Big)\xi_{w_1}(x_{k,j}),\]
  then invoke the Taylor expansion and expand $z_2(w_1^h(x_{k,j}), w_2^h(x_{k,j}))$ around $(w_1(x_{k,j}), w_2(x_{k,j}))$, we obtain
  \begin{multline*}
  \Lambda_3 = -\sum_{j,k}\underline{\omega}_{k,j} \bigg(\frac{\partial z_2}{\partial w_1}\Big(w_1(x_{k,j}) + \theta_1 e_{w_1}(x_{k,j}), w_2(x_{k,j})\Big)e_{w_1}(x_{k,j}) \\+ \frac{\partial z_2}{\partial w_2}\Big(w_1(x_{k,j}), w_2(x_{k,j}) + \theta_2 e_{w_2}(x_{k,j})\Big)e_{w_2}(x_{k,j})\bigg) \xi_{w_1}(x_{k,j}),
\end{multline*}
where $-1 < \theta_1, \theta_2 < 0$. Now, combing Young's inequality and the condition (\ref{N_proper}) leads to
\begin{equation}\label{lambda_32}
 |\Lambda_3| \leq C\left(h^{2(q+1)} + \|\xi_{w_1}\|_{L^2(I)}^2 + \|\xi_{w_2}\|_{L^2(I)}^2 \right).
\end{equation}
Similarly, we can rewrite $\Lambda_4$ regarding to $z_1(\cdot, \cdot)$ as
\[\Lambda_4 = \sum_{j,k}\underline{\omega}_{k,j} \Big(z_1\big(w_1(x_{k,j}), w_2(x_{k,j})\big) - z_1\big(w_1^h(x_{k,j}), w_2^h(x_{k,j})\big)\Big)\xi_{w_2}(x_{k,j}),\]
then following the same process as the estimate of $\Lambda_3$, we get
\begin{equation}\label{lambda_4}
    |\Lambda_4| \leq C\left(h^{2(q+1)} + \|\xi_{w_1}\|_{L^2(I)}^2 + \|\xi_{w_2}\|_{L^2(I)}^2\right).
\end{equation}
Next, plugging (\ref{lambda_1}) -- (\ref{lambda_4}) into (\ref{dxi_dt}), invoking Young's equality and integrating the resulting equality from $0$ to $T$ yields
\begin{multline}\label{inequality_at_T}
\|\xi_{w_1}(x,T)\|_{L^2(I)}^2 + \|\xi_{w_2}(x,T)\|_{L^2(I)}^2 \leq \|\xi_{w_1}(x,0)\|_{L^2(I)}^2 + \|\xi_{w_2}(x,0)\|_{L^2(I)}^2\\
+ C\int_0^T  \|\xi_{w_1}\|_{L^2(I)}^2 +\|\xi_{w_2}\|_{L^2(I)}^2 \ dt +  Ch^{2q}.
\end{multline}
Finally, applying Gronwall's inequality to (\ref{inequality_at_T}) and using the fact that the initial projection errors $\xi_{w_1}(x,0) = \xi_{w_2}(x, 0) = 0$ gives rise to
\[\|\xi_{w_1}(x,T)\|_{L^2(I)}^2 + \|\xi_{w_2}(x,T)\|_{L^2(I)}^2 \leq \mathcal{C}h^{2q},\]
and this collects the error estimate (\ref{error_estimate}) thanks to the triangle inequality and the property (\ref{projection}) of Gauss–Radau projection.
  \end{proof}
  \begin{remark}
  While the theorem is stated for the case of physical boundary conditions, we note that a similar conclusion holds for some other types of boundary conditions like periodic boundary conditions. In the latter case, some improvement can also be made, as shown in the next section.
  \end{remark}

\subsection{Improved estimate with the upwind fluxes (\ref{upwind_flux})}\label{sec:improved_convergence}
In this section, we assume periodic boundary conditions are imposed and the numerical fluxes are chosen to be the upwind fluxes defined in (\ref{upwind_flux}). Then we have the following error estimate.

 \begin{truth}\label{truth3}
  Let $(w_1, w_2)$ be a smooth solution of system (\ref{system_characteristic}) with periodic boundary conditions, and $(w_1^h, w_2^h) \in V_h^q\times V_h^q, q \geq 0$ be the numerical solution of the DG scheme (\ref{DG1})-(\ref{DG2}) with the upwind fluxes defined in (\ref{upwind_flux}), and the nonlinear interaction $N(\cdot,\cdot)$ satisfies the condition (\ref{N_proper}). Further, the approximation of the initial data is computed based on (\ref{initial_projection}). Then there exists a positive constant $\mathcal{C}$ depending only on the degree $q$, the bound given in (\ref{N_proper}) and the shape regularity of the mesh, but independent of $h$, such that
  \begin{equation}\label{error_estimate1}
  \|e_{w_1}(x, T)\|_{L^2(I)} + \|e_{w_2}(x, T)\|_{L^2(I)} \leq \mathcal{C}h^{q+1}.
  \end{equation}
  \end{truth}
  \begin{proof}
  Following the same analysis as in the derivation of Theorem \ref{truth2}, we have
   \begin{equation}\label{dxi_dt1}
     \frac{1}{2} \frac{d}{dt} \sum_j \int_{I_j} \xi_{w_1}^2 + \xi_{w_2}^2\ dx = \Lambda_1 + \Lambda_2 + \Lambda_3 + \Lambda_4 + \Lambda_5,
  \end{equation}
  where $\Lambda_i, i = 1,2,\cdots, 5$ are defined in (\ref{lambda}). Similarly, we can obtain
  \begin{equation}\label{estimate_of_lambda}
      |\Lambda_1| + \Lambda_2 + |\Lambda_3| + |\Lambda_4| \leq C\left(h^{2(q+1)} + \|\xi_{w_1}\|^2_{L^2(I)} + \|\xi_{w_2}\|^2_{L^2(I)}\right).
  \end{equation}
  As for the estimate of $\Lambda_5$, based on the property of the projection operator (\ref{projection}) and the definition of the upwind fluxes (\ref{upwind_flux}), we have
  \[\Lambda_5 = 0.\]
  Thus we obtain
  \[\frac{1}{2} \frac{d}{dt} \sum_j \int_{I_j} \xi_{w_1}^2 + \xi_{w_2}^2\ dx \leq C\left(h^{2(q+1)} + \|\xi_{w_1}\|^2_{L^2(I)} + \|\xi_{w_2}\|^2_{L^2(I)}\right),\]
  which leads to 
  \[\|\xi_{w_1}(x,T)\|_{L^2(I)}^2 + \|\xi_{w_2}(x,T)\|_{L^2(I)}^2 \leq \mathcal{C}h^{2(q+1)},\]
  by invoking the initial projection errors $\xi_{w_1}(x,0) = \xi_{w_2}(x, 0) = 0$. Finally, (\ref{error_estimate1}) follows by an application of triangle inequality in $e_{w_1} = \xi_{w_1} + \eta_{w_1}$ and $e_{w_2} = \xi_{w_2} + \eta_{w_2}$.
  \end{proof}
 \begin{remark}
 For the error analysis, we see that the optimal convergence order is attained when the upwind fluxes with $\alpha_1 = \alpha_2 = \beta_1 = \beta_2 = 0$ in (\ref{flux}) and periodic boundary conditions are considered. For other cases, we can only prove a sub-optimal convergence, but numerically we observe optimal convergence rates for the mixed upwind flux (\ref{mu_flux}) with $\alpha_1 = \alpha_2 = 0, \beta_1 = \beta_2 = 1$ in (\ref{flux}) and the mixed central flux (\ref{mc_flux}) with $\alpha_1 = \alpha_2 = \beta_1 = \beta_2 = 1$ in (\ref{flux}), subject to either periodic boundary conditions (see Table \ref{example_one_l2_mc} and Table \ref{example_one_l2_mu}) or homogeneous Dirichlet boundary conditions (see Table \ref{example_two_l2_mc} and Table \ref{example_two_l2_mu}). For the central fluxes with $\alpha_1 = \alpha_2 = 1, \beta_1 = \beta_2 = 0$, we observe a super-convergence order $q + 2$ when $q$ is even; while a sub-optimal convergence order $q$ when $q$ is odd (see Table \ref{example_one_l2_central} and Table \ref{example_two_l2_central}). 
 \end{remark}

\section{Numerical experiments}\label{sec:simulations}
In this section, we present some numerical experiments to illustrate the performance of the proposed DG scheme in Section \ref{Sec: DG_formulation}. 
As mentioned in the introduction, we are interested in the potential of the current numerical scheme to study the nonlinear stability of kink solutions. For this reason, in what follows only such solutions will be simulated. A long-time simulation in the absence of any perturbation will set a stage for future simulation of the perturbed kink-like solutions. Through these studies, we assume that the computational domain is
discretized by a uniform grid $x_{\frac{1}{2}},\cdots,x_{j-\frac{1}{2}},x_{j+\frac{1}{2}},\cdots x_{N+\frac{1}{2}}$ with spacing $h$. Let $x_{j}=(x_{j-\frac{1}{2}}+x_{j+\frac{1}{2}})/2$; then the mapping $r=\frac{2}{h}(x-x_j)$ takes element $I_j$ to the reference element $I_r=(-1,1)$. Moreover, we use standard modal basis formulations,
\[w_1^h(x, t) = \sum_{n = 0}^q \overline{w}_{1n}(t)\phi_n(r), \quad w_2^h(x, t) = \sum_{n = 0}^q \overline{w}_{2n}(t)\psi_n(r)\]
to expand $w_1^h$ and $w_2^h$ in test functions. In addition, for the numerical experiments conducted in this work, we present results by considering the degree of the approximation DG space being $q = (0,1,2,3)$. With such choices, we simply use tensor-product Gauss rules with $17$ nodes in the reference element $I_r$ for the calculation of the nonlinear volume integrals in (\ref{DG1})--(\ref{DG2}) without bothering to pick
the minimal number of nodes required to observe the convergence rates shown in the examples of Section \ref{sec:convergence_study}. Last, we use normalized Legendre polynomials \cite{hesthaven2007nodal} as the test and trial functions.

\subsection{Convergence rates}\label{sec:convergence_study}
We begin to test the order of convergence of the proposed DG scheme (\ref{DG1})--(\ref{DG2}). In particular, we consider four different numerical fluxes:
\begin{align}
    \mbox{upwind flux}: & \quad \widehat{w_1^h} = w_1^{h, +}, \quad \widetilde{w_2^h} = w_2^{h, -}; \label{u_flux}\\
    \mbox{central flux}: & \quad \widehat{w_1^h} = \{w_1^{h}\}, \quad \widetilde{w_2^h} = \{w_2^{h}\}; \label{c_flux}\\
    \mbox{mixed upwind flux}: & \quad \widehat{w_1^h} = w_1^{h, +} + \frac{1}{2}[w_2^h], \quad \widetilde{w_2^h} = w_2^{h, -} + \frac{1}{2}[w_1^h]; \label{mu_flux}\\
    \mbox{mixed central flux}: & \quad \widehat{w_1^h} = \{w_1^{h}\} + \frac{1}{2}[w_2^h], \quad \widetilde{w_2^h} = \{w_2^{h}\} + \frac{1}{2}[w_1^h]. \label{mc_flux}
\end{align}

In addition, to numerically evolve the solution $w_1^h$ and $w_2^h$ in (\ref{DG1})--(\ref{DG2}), we implement the classical $4$-stages explicit Rung-Kutta time integrator for the experiments in Section \ref{sec:convergence_study}. To observe the desired convergence rate for spatial discretization, we use a time step size 
\begin{equation*}
\Delta_t = \mbox{CFL} \times h,\quad \mbox{CFL} = \frac{3.75\times 10^{-2}}{\pi}
\end{equation*}
to guarantee that the spatial error dominates the temporal error. 

\subsubsection{A test problem with periodic boundary conditions}\label{sec:example1}
For the first example, we test a problem with periodic boundary conditions. Consider the manufactured solution
\begin{equation}\label{example1_sol}
    w_1(x,t) = \frac{1}{\sqrt{2}}(\cos(\pi x)+\sin(\pi x)) \cos(t),\quad w_2(x,t) = \frac{1}{\sqrt{2}}(\cos(\pi x) - \sin(\pi x))\cos(t).
\end{equation}
for the following first-order nonlinear hyperbolic system 
\begin{equation}\label{example1}
\begin{aligned}
\partial_t w_1 &= \ \ \partial_x w_1 -N\left(w_1, w_2\right)w_2 + f_1(x, t), \quad (x,t)\in 
(x_a,x_b) \times(0,T),\\
\partial_t w_2 &= -\partial_x w_2 + N\left(w_1, w_2\right)w_1 + f_2(x, t), \quad (x,t)\in 
(x_a,x_b) \times(0,T),
\end{aligned}
\end{equation}
with $(x_a,x_b)=(-2,2)$, $T=1$. For the nonlinear term, $N(w_1, w_2)$, we consider a
nonlinearity proposed in \cite{liweinstein23},
\begin{equation}\label{nonlinear_term_numer}
N(w_1, w_2) = 2\mbox{sech}\left(\mbox{arccosh}(2)
\sqrt{w_1^2 + w_2^2}
\right) - 1,
\end{equation}
which clearly satisfies the condition $(\ref{N_cond})$. In Figure \ref{fig:nonliner_interaction}, we present the picture of $N(w_1, w_2)$ with respect to the modulus $y$, and see that $N(w_1, w_2)\in[-1, 1]$. In particular, $N(w_1, w_2) = 1$ when $\sqrt{w_1^2 + w_2^2}= 0$, and $N(w_1, w_2) \rightarrow -1$ when 
$\sqrt{w_1^2 + w_2^2}\rightarrow \infty$. Moreover, direct calculation with the help of MATLAB shows that
\[\max_{w_1,w_2} \left|\frac{\partial}{\partial w_j}\Big(w_i N(w_1,w_2)\Big)\right| \leq 2\]
for $i,j=1,2$ so that the assumption (\ref{N_proper}) is satisfied. Moreover, the initial data, external forcing $f_1(x,t), f_2(x,t)$ are determined based on (\ref{example1_sol})--(\ref{example1}).
\begin{figure}[htb!]
	\centering
	\includegraphics[width=0.5\linewidth,trim={.1cm 0.1cm .5cm .5cm},clip]{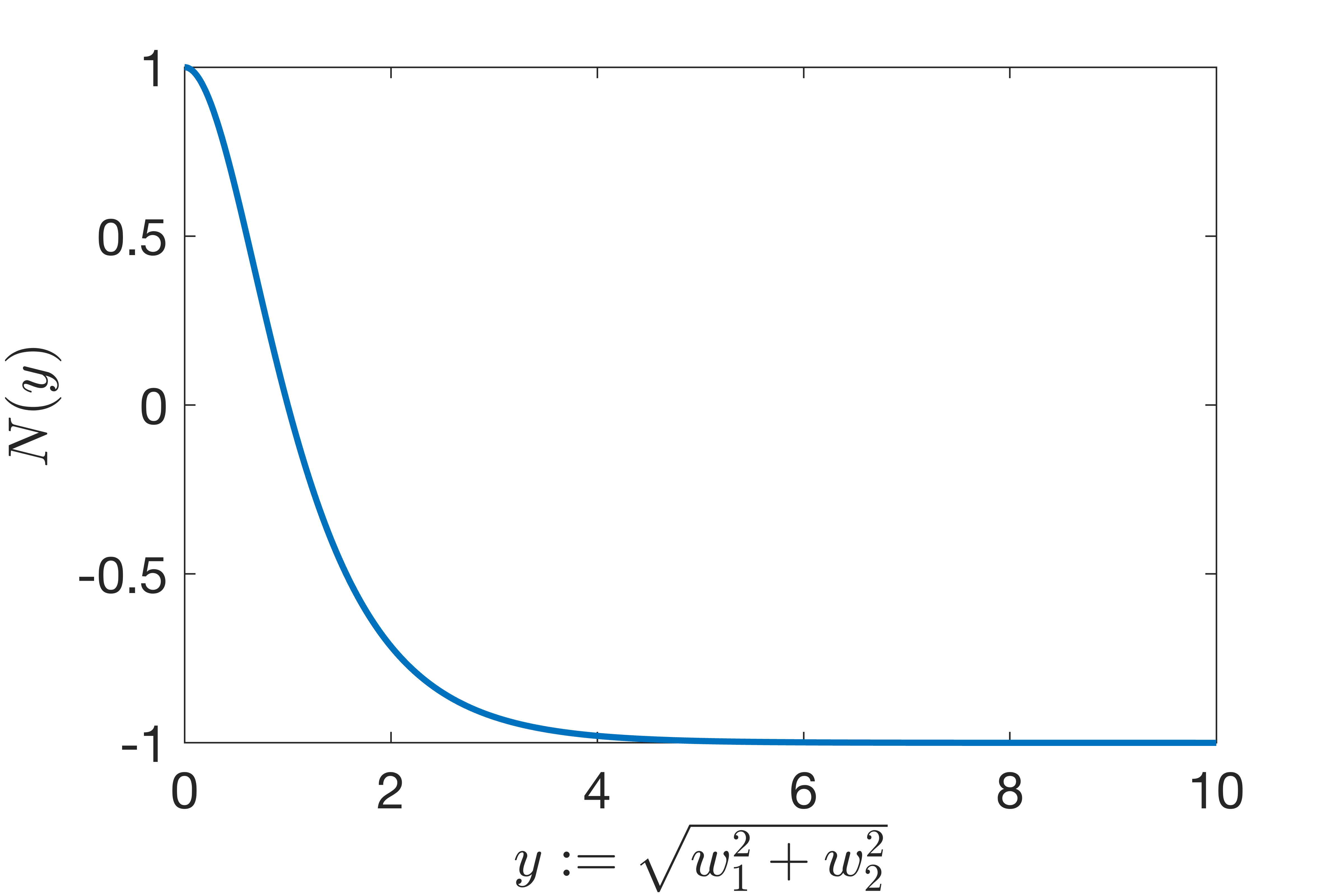}
\caption{We present the picture of $N(w_1, w_2)$ with respect to the modulus $y = \sqrt{w_1^2+w_2^2}$.}\label{fig:nonliner_interaction}
\end{figure}



From Table \ref{example_one_l2_up} to Table \ref{example_one_l2_mu}, we list $L^2$ errors in $w_1, w_2, b_1, b_2$, and the corresponding numerical orders of accuracy subject to the variation of $q$ and $N$ at final time $T = 1$ with the numerical fluxes given from (\ref{u_flux}) to (\ref{mc_flux}). Note that the DG scheme (\ref{DG1})--(\ref{DG2}) is developed for the characteristic variables $w_1, w_2$; for completeness, we also include the errors in $b_1$ and $b_2$ by performing the linear transformation (\ref{b2w}) to obtain the DG approximations $b_1^h$ and $b_2^h$. In particular, Table \ref{example_one_l2_central} and Table \ref{example_one_l2_mc} present the results with conservative schemes: the central flux (\ref{c_flux}) and the mixed central flux (\ref{mc_flux}), while Table \ref{example_one_l2_up} and Table \ref{example_one_l2_mu} display the results with dissipating schemes: the upwind scheme (\ref{u_flux}) and the mixed upwind scheme (\ref{mu_flux}). 

For the energy dissipating scheme where the upwind flux or the mixed upwind scheme is implemented, from Table \ref{example_one_l2_up} or Table \ref{example_one_l2_mu}, we observe that the proposed scheme consistently gives the optimal $(q + 1)$-th order of accuracy across the choices of the number of elements $N$ for the solutions $w_1, w_2, b_1$ and $b_2$. 

For the energy conserving scheme where the central flux (\ref{c_flux}) is used, from Table \ref{example_one_l2_central}, we note a super-convergence order $q+2$ when $q$ is even, while a sub-optimal convergence rate $q$ when $q$ is odd. Last, for the energy conserving scheme where the mixed central flux (\ref{mc_flux}) is considered, from Table \ref{example_one_l2_mc}, we also observe an optimal convergence order $q+1$ though there are fluctuations, which is common for a conservative scheme since the initial error cannot be quickly damped.

\begin{table}
 	\footnotesize
 	\begin{center}
 		\scalebox{1.0}{
 			\begin{tabular}{c c c c c c c c c c c c c c c}
 				\hline
 				~ & ~ &  $w_1$ &~ & $w_2$& ~ & $b_1$ & ~ & $b_2$ & ~\\
 				\cline{3-10}
 				$q$ & $N$ & $L^2$ error & order   & $L^2$ error & order  & $L^2$ error & order & $L^2$ error & order \\
 				\hline
 				0 & 40 & 2.0618e-01 & -- & 2.2555e-01 & -- & 2.7714e-01 & -- & 3.3160e-01 & --  \\
~ & 80 & 1.2388e-01 & 0.7349 & 1.2250e-01 & 0.8807 & 1.5022e-01 & 0.8835 & 1.9529e-01 & 0.7638  \\
~ & 160 & 6.8560e-02 & 0.8536 & 6.3772e-02 & 0.9417 & 7.9408e-02 & 0.9197 & 1.0597e-01 & 0.8820 \\
~ & 320 & 3.6141e-02 & 0.9237 & 3.2518e-02 & 0.9717 & 4.1012e-02 & 0.9533 & 5.5183e-02 & 0.9413 \\
~ & 640 & 1.8563e-02 & 0.9612 & 1.6416e-02 & 0.9861 & 2.0866e-02 & 0.9749 & 2.8156e-02 & 0.9708 \\
 				~ & ~ & ~ & ~ & ~ & ~ & ~ & ~ & ~ & ~\\
 				1 & 40 & 2.7426e-03 & -- & 2.7172e-03 & -- & 4.0346e-03 & -- & 3.6786e-03 & --  \\
~ & 80 & 6.5702e-04 & 2.0616 & 6.5526e-04 & 2.0520 & 9.5004e-04 & 2.0864 & 9.0527e-04 & 2.0228 \\
~ & 160 & 1.6186e-04 & 2.0212 & 1.6174e-04 & 2.0184 & 2.3167e-04 & 2.0359 & 2.2594e-04 & 2.0024 \\
~ & 320 & 4.0247e-05 & 2.0078 & 4.0240e-05 & 2.0070 & 5.7276e-05 & 2.0160 & 5.6548e-05 & 1.9984 \\
~ & 640 & 1.0040e-05 & 2.0031 & 1.0040e-05 & 2.0029 & 1.4244e-05 & 2.0075 & 1.4153e-05 & 1.9984 \\
 			    ~ & ~ & ~ & ~ & ~ & ~ & ~ & ~ & ~ & ~\\
 				2 & 40 & 6.3391e-05 & -- & 6.3341e-05 & -- & 9.3036e-05 & -- & 8.6053e-05 & --  \\
~ & 80 & 7.8644e-06 & 3.0109 & 7.8629e-06 & 3.0100 & 1.1336e-05 & 3.0368 & 1.0901e-05 & 2.9807 \\
~ & 160 & 9.7931e-07 & 3.0055 & 9.7926e-07 & 3.0053 & 1.3984e-06 & 3.0191 & 1.3713e-06 & 2.9909 \\
~ & 320 & 1.2218e-07 & 3.0028 & 1.2218e-07 & 3.0027 & 1.7363e-07 & 3.0097 & 1.7194e-07 & 2.9956 \\
~ & 640 & 1.5257e-08 & 3.0014 & 1.5257e-08 & 3.0014 & 2.1630e-08 & 3.0049 & 2.1524e-08 & 2.9978 \\
 				~ & ~ & ~ & ~ & ~ & ~ & ~ & ~ & ~ & ~\\
 				3 & 40 & 1.1951e-06 & -- & 1.1960e-06 & -- & 1.7589e-06 & -- & 1.6197e-06 & --  \\
~ & 80 & 7.4395e-08 & 4.0058 & 7.4413e-08 & 4.0065 & 1.0684e-07 & 4.0411 & 1.0358e-07 & 3.9669 \\
~ & 160 & 4.6378e-09 & 4.0037 & 4.6381e-09 & 4.0040 & 6.6103e-09 & 4.0146 & 6.5074e-09 & 3.9925 \\
~ & 320 & 2.8948e-10 & 4.0019 & 2.8948e-10 & 4.0020 & 4.1099e-10 & 4.0075 & 4.0777e-10 & 3.9962 \\
~ & 640 & 1.8080e-11 & 4.0009 & 1.8080e-11 & 4.0010 & 2.5619e-11 & 4.0038 & 2.5520e-11 & 3.9981 \\
 				\hline
 			\end{tabular}
 		}
 	\end{center}
 	\caption{\scriptsize{$L^2$ errors and the corresponding convergence rates for $w_1, w_2, b_1, b_2$ of problem (\ref{example1}) using $\mathcal{P}^q$ polynomials and the upwind flux (\ref{u_flux}).  The interval is divided into $N$ uniform cells, and the terminal computational time is $T = 1$.}}\label{example_one_l2_up}
 \end{table}
 
 \begin{table}
 	\footnotesize
 	\begin{center}
 		\scalebox{1.0}{
 			\begin{tabular}{c c c c c c c c c c c c c c c}
 				\hline
 				~ & ~ &  $w_1$ &~ & $w_2$& ~ & $b_1$ & ~ & $b_2$ & ~\\
 				\cline{3-10}
 				$q$ & $N$ & $L^2$ error & order   & $L^2$ error & order  & $L^2$ error & order & $L^2$ error & order \\
 				\hline
 				0 & 40 & 1.7427e-02 & -- & 2.7306e-02 & -- & 3.8033e-02 & -- & 2.5536e-02 & --  \\
~ & 80 & 4.3580e-03 & 1.9996 & 6.7582e-03 & 2.0145 & 9.5417e-03 & 1.9949 & 6.1877e-03 & 2.0450 \\
~ & 160 & 1.0896e-03 & 1.9999 & 1.6855e-03 & 2.0035 & 2.3878e-03 & 1.9985 & 1.5344e-03 & 2.0117 \\
~ & 320 & 2.7240e-04 & 2.0000 & 4.2112e-04 & 2.0009 & 5.9712e-04 & 1.9996 & 3.8281e-04 & 2.0030 \\
~ & 640 & 6.8101e-05 & 2.0000 & 1.0526e-04 & 2.0002 & 1.4929e-04 & 1.9999 & 9.5654e-05 & 2.0007 \\
 				~ & ~ & ~ & ~ & ~ & ~ & ~ & ~ & ~ & ~\\
 				1 & 40 & 4.6923e-02 & -- & 4.9988e-02 & -- & 6.8179e-02 & -- & 6.8940e-02 & --  \\
~ & 80 & 2.3545e-02 & 0.9948 & 2.5185e-02 & 0.9890 & 3.2583e-02 & 1.0652 & 3.6272e-02 & 0.9265 \\
~ & 160 & 1.1781e-02 & 0.9989 & 1.2614e-02 & 0.9975 & 1.6092e-02 & 1.0178 & 1.8354e-02 & 0.9827 \\
~ & 320 & 5.8917e-03 & 0.9997 & 6.3097e-03 & 0.9994 & 8.0206e-03 & 1.0046 & 9.2044e-03 & 0.9957 \\
~ & 640 & 2.9460e-03 & 0.9999 & 3.1552e-03 & 0.9998 & 4.0071e-03 & 1.0011 & 4.6056e-03 & 0.9989 \\
 			    ~ & ~ & ~ & ~ & ~ & ~ & ~ & ~ & ~ & ~\\
 				2 & 40 & 1.7123e-05 & -- & 1.6677e-05 & -- & 2.6737e-05 & -- & 2.0682e-05 & --  \\
~ & 80 & 6.0194e-07 & 4.8301 & 5.7921e-07 & 4.8476 & 8.7872e-07 & 4.9273 & 7.8961e-07 & 4.7111 \\
~ & 160 & 5.4792e-08 & 3.4576 & 5.3447e-08 & 3.4379 & 7.7657e-08 & 3.5002 & 7.5410e-08 & 3.3883 \\
~ & 320 & 2.3318e-09 & 4.5545 & 2.2482e-09 & 4.5713 & 3.2600e-09 & 4.5742 & 3.2180e-09 & 4.5505 \\
~ & 640 & 2.4289e-10 & 3.2631 & 2.3752e-10 & 3.2426 & 3.4132e-10 & 3.2556 & 3.3810e-10 & 3.2506 \\
 				~ & ~ & ~ & ~ & ~ & ~ & ~ & ~ & ~ & ~\\
 				3 & 40 & 2.1038e-05 & -- & 2.0260e-05 & -- & 3.7819e-05 & -- & 1.6610e-05 & --  \\
~ & 80 & 2.7425e-06 & 2.9394 & 2.7826e-06 & 2.8641 & 4.1628e-06 & 3.1835 & 3.6332e-06 & 2.1927 \\
~ & 160 & 3.4322e-07 & 2.9983 & 3.5234e-07 & 2.9814 & 4.8917e-07 & 3.0891 & 4.9456e-07 & 2.8770 \\
~ & 320 & 4.2905e-08 & 2.9999 & 4.4170e-08 & 2.9958 & 6.0111e-08 & 3.0246 & 6.3010e-08 & 2.9725 \\
~ & 640 & 5.3632e-09 & 3.0000 & 5.5252e-09 & 2.9990 & 7.4811e-09 & 3.0063 & 7.9130e-09 & 2.9933 \\
 				\hline
 			\end{tabular}
 		}
 	\end{center}
 	\caption{\scriptsize{$L^2$ errors and the corresponding convergence rates for $w_1, w_2, b_1, b_2$ of problem (\ref{example1}) using $\mathcal{P}^q$ polynomials and the central flux (\ref{c_flux}).  The interval is divided into $N$ uniform cells, and the terminal computational time is $T = 1$.}}\label{example_one_l2_central}
 \end{table}
 
 \begin{table}
 	\footnotesize
 	\begin{center}
 		\scalebox{1.0}{
 			\begin{tabular}{c c c c c c c c c c c c c c c}
 				\hline
 				~ & ~ &  $w_1$ &~ & $w_2$& ~ & $b_1$ & ~ & $b_2$ & ~\\
 				\cline{3-10}
 				$q$ & $N$ & $L^2$ error & order   & $L^2$ error & order  & $L^2$ error & order & $L^2$ error & order \\
 				\hline
 				0 & 40 & 3.1939e-01 & -- & 2.4017e-01 & -- & 3.5673e-01 & -- & 4.3832e-01 & --  \\
~ & 80 & 1.5264e-01 & 1.0652 & 1.2336e-01 & 0.9612 & 1.6686e-01 & 1.0961 & 2.2179e-01 & 0.9828 \\
~ & 160 & 7.4617e-02 & 1.0325 & 6.2322e-02 & 0.9850 & 8.0635e-02 & 1.0492 & 1.1136e-01 & 0.9939 \\
~ & 320 & 3.6892e-02 & 1.0162 & 3.1302e-02 & 0.9935 & 3.9637e-02 & 1.0245 & 5.5772e-02 & 0.9976 \\
~ & 640 & 1.8343e-02 & 1.0081 & 1.5684e-02 & 0.9970 & 1.9652e-02 & 1.0122 & 2.7906e-02 & 0.9990 \\
 				~ & ~ & ~ & ~ & ~ & ~ & ~ & ~ & ~ & ~\\
 				1 & 40 & 7.7379e-03 & -- & 6.6755e-03 & -- & 1.0531e-02 & -- & 9.8986e-03 & --  \\
~ & 80 & 9.3795e-04 & 3.0444 & 8.0919e-04 & 3.0443 & 1.2227e-03 & 3.1064 & 1.2546e-03 & 2.9800 \\
~ & 160 & 3.3209e-04 & 1.4979 & 3.1869e-04 & 1.3443 & 4.5389e-04 & 1.4297 & 4.6655e-04 & 1.4271 \\
~ & 320 & 1.1293e-04 & 1.5561 & 1.1180e-04 & 1.5112 & 1.5910e-04 & 1.5125 & 1.5873e-04 & 1.5555 \\
~ & 640 & 1.5939e-05 & 2.8248 & 1.5778e-05 & 2.8250 & 2.2371e-05 & 2.8302 & 2.2483e-05 & 2.8196 \\
 			    ~ & ~ & ~ & ~ & ~ & ~ & ~ & ~ & ~ & ~\\
 				2 & 40 & 1.5944e-04 & -- & 1.2786e-04 & -- & 1.9650e-04 & -- & 2.1196e-04 & --  \\
~ & 80 & 1.0080e-05 & 3.9835 & 8.4746e-06 & 3.9152 & 1.2982e-05 & 3.9199 & 1.3352e-05 & 3.9886 \\
~ & 160 & 1.6531e-06 & 2.6081 & 1.5312e-06 & 2.4685 & 2.2696e-06 & 2.5160 & 2.2369e-06 & 2.5776 \\
~ & 320 & 2.9000e-07 & 2.5111 & 2.8078e-07 & 2.4471 & 4.0386e-07 & 2.4905 & 4.0345e-07 & 2.4710 \\
~ & 640 & 2.0669e-08 & 3.8105 & 2.0258e-08 & 3.7929 & 2.9057e-08 & 3.7969 & 2.8826e-08 & 3.8070 \\
 				~ & ~ & ~ & ~ & ~ & ~ & ~ & ~ & ~ & ~\\
 				3 & 40 & 2.7237e-06 & -- & 2.5996e-06 & -- & 3.5133e-06 & -- & 4.0012e-06 & --  \\
~ & 80 & 1.8335e-07 & 3.8929 & 1.6457e-07 & 3.9815 & 2.5275e-07 & 3.7970 & 2.3983e-07 & 4.0603 \\
~ & 160 & 1.2695e-08 & 3.8522 & 1.2362e-08 & 3.7347 & 1.7513e-08 & 3.8512 & 1.7924e-08 & 3.7421 \\
~ & 320 & 5.4586e-10 & 4.5396 & 5.2741e-10 & 4.5508 & 7.6433e-10 & 4.5181 & 7.5370e-10 & 4.5718 \\
~ & 640 & 4.9108e-11 & 3.4745 & 4.8950e-11 & 3.4296 & 6.9477e-11 & 3.4596 & 6.9199e-11 & 3.4452 \\
 				\hline
 			\end{tabular}
 		}
 	\end{center}
 	\caption{\scriptsize{$L^2$ errors and the corresponding convergence rates for $w_1, w_2, b_1, b_2$ of problem (\ref{example1}) using $\mathcal{P}^q$ polynomials and the mixed central flux (\ref{mc_flux}).  The interval is divided into $N$ uniform cells, and the terminal computational time is $T = 1$.}}\label{example_one_l2_mc}
 \end{table}

 \begin{table}
 	\footnotesize
 	\begin{center}
 		\scalebox{1.0}{
 			\begin{tabular}{c c c c c c c c c c c c c c c}
 				\hline
 				~ & ~ &  $w_1$ &~ & $w_2$& ~ & $b_1$ & ~ & $b_2$ & ~\\
 				\cline{3-10}
 				$q$ & $N$ & $L^2$ error & order   & $L^2$ error & order  & $L^2$ error & order & $L^2$ error & order \\
 				\hline
 				0 & 40 & 2.8770e-01 & -- & 2.4850e-01 & -- & 3.7720e-01 & -- & 3.8311e-01 & --  \\
~ & 80 & 1.6003e-01 & 0.8463 & 1.3862e-01 & 0.8421 & 1.9287e-01 & 0.9677 & 2.2902e-01 & 0.7423 \\
~ & 160 & 8.5314e-02 & 0.9074 & 7.2910e-02 & 0.9270 & 9.7628e-02 & 0.9823 & 1.2513e-01 & 0.8720 \\
~ & 320 & 4.4175e-02 & 0.9496 & 3.7342e-02 & 0.9653 & 4.9162e-02 & 0.9897 & 6.5382e-02 & 0.9365 \\
~ & 640 & 2.2493e-02 & 0.9738 & 1.8891e-02 & 0.9831 & 2.4678e-02 & 0.9943 & 3.3416e-02 & 0.9684 \\
 				~ & ~ & ~ & ~ & ~ & ~ & ~ & ~ & ~ & ~\\
 				1 & 40 & 1.8697e-03 & -- & 1.9552e-03 & -- & 2.6172e-03 & -- & 2.7906e-03 & --  \\
~ & 80 & 4.5582e-04 & 2.0363 & 4.6687e-04 & 2.0662 & 6.4279e-04 & 2.0256 & 6.6205e-04 & 2.0756 \\
~ & 160 & 1.1337e-04 & 2.0074 & 1.1480e-04 & 2.0239 & 1.6021e-04 & 2.0044 & 1.6247e-04 & 2.0268 \\
~ & 320 & 2.8326e-05 & 2.0009 & 2.8507e-05 & 2.0097 & 4.0051e-05 & 2.0001 & 4.0323e-05 & 2.0105 \\
~ & 640 & 7.0829e-06 & 1.9997 & 7.1058e-06 & 2.0043 & 1.0016e-05 & 1.9995 & 1.0050e-05 & 2.0045 \\
 			    ~ & ~ & ~ & ~ & ~ & ~ & ~ & ~ & ~ & ~\\
 				2 & 40 & 9.6659e-05 & -- & 7.9271e-05 & -- & 1.3628e-04 & -- & 1.1261e-04 & --  \\
~ & 80 & 1.1624e-05 & 3.0558 & 1.0528e-05 & 2.9126 & 1.6425e-05 & 3.0526 & 1.4904e-05 & 2.9176 \\
~ & 160 & 1.4181e-06 & 3.0351 & 1.3498e-06 & 2.9634 & 2.0050e-06 & 3.0342 & 1.9094e-06 & 2.9645 \\
~ & 320 & 1.7491e-07 & 3.0193 & 1.7065e-07 & 2.9836 & 2.4734e-07 & 3.0190 & 2.4135e-07 & 2.9839 \\
~ & 640 & 2.1711e-08 & 3.0101 & 2.1446e-08 & 2.9923 & 3.0704e-08 & 3.0100 & 3.0330e-08 & 2.9923 \\
 				~ & ~ & ~ & ~ & ~ & ~ & ~ & ~ & ~ & ~\\
 				3 & 40 & 8.3178e-07 & -- & 9.0957e-07 & -- & 1.1909e-06 & -- & 1.2728e-06 & --  \\
~ & 80 & 5.1621e-08 & 4.0102 & 5.4133e-08 & 4.0706 & 7.3144e-08 & 4.0252 & 7.6421e-08 & 4.0579 \\
~ & 160 & 3.2398e-09 & 3.9940 & 3.3172e-09 & 4.0284 & 4.5815e-09 & 3.9968 & 4.6916e-09 & 4.0258 \\
~ & 320 & 2.0341e-10 & 3.9935 & 2.0582e-10 & 4.0105 & 2.8766e-10 & 3.9934 & 2.9107e-10 & 4.0106 \\
~ & 640 & 1.2744e-11 & 3.9964 & 1.2819e-11 & 4.0050 & 1.8022e-11 & 3.9965 & 1.8130e-11 & 4.0049 \\
 				\hline
 			\end{tabular}
 		}
 	\end{center}
 	\caption{\scriptsize{$L^2$ errors and the corresponding convergence rates for $w_1, w_2, b_1, b_2$ of problem (\ref{example1}) using $\mathcal{P}^q$ polynomials and the mixed upwind flux (\ref{mu_flux}). The interval is divided into $N$ uniform cells, and the terminal computational time is $T = 1$.}}\label{example_one_l2_mu}
 \end{table}
 
 In Figure \ref{fig:errors_eg1}, we report the errors in the solution $w_1$ (left) and $w_2$ (right) with respect to the spatial locations at the final time $T = 1$ with the approximation order $q = 3$, the number of elements $N = 20,40,80$; from top to bottom are the results by using the upwind fluxes (\ref{u_flux}), the central fluxes (\ref{c_flux}), the mixed central fluxes (\ref{mc_flux}), and the mixed upwind fluxes (\ref{mu_flux}), respectively. We note that there is no severe error localization for both solutions $w_1$ and $w_2$ with four different numerical fluxes.
 
 \begin{figure}[htb!]
	\centering
	\includegraphics[width=0.45\linewidth,trim={.0cm .0cm .0cm .0cm},clip]{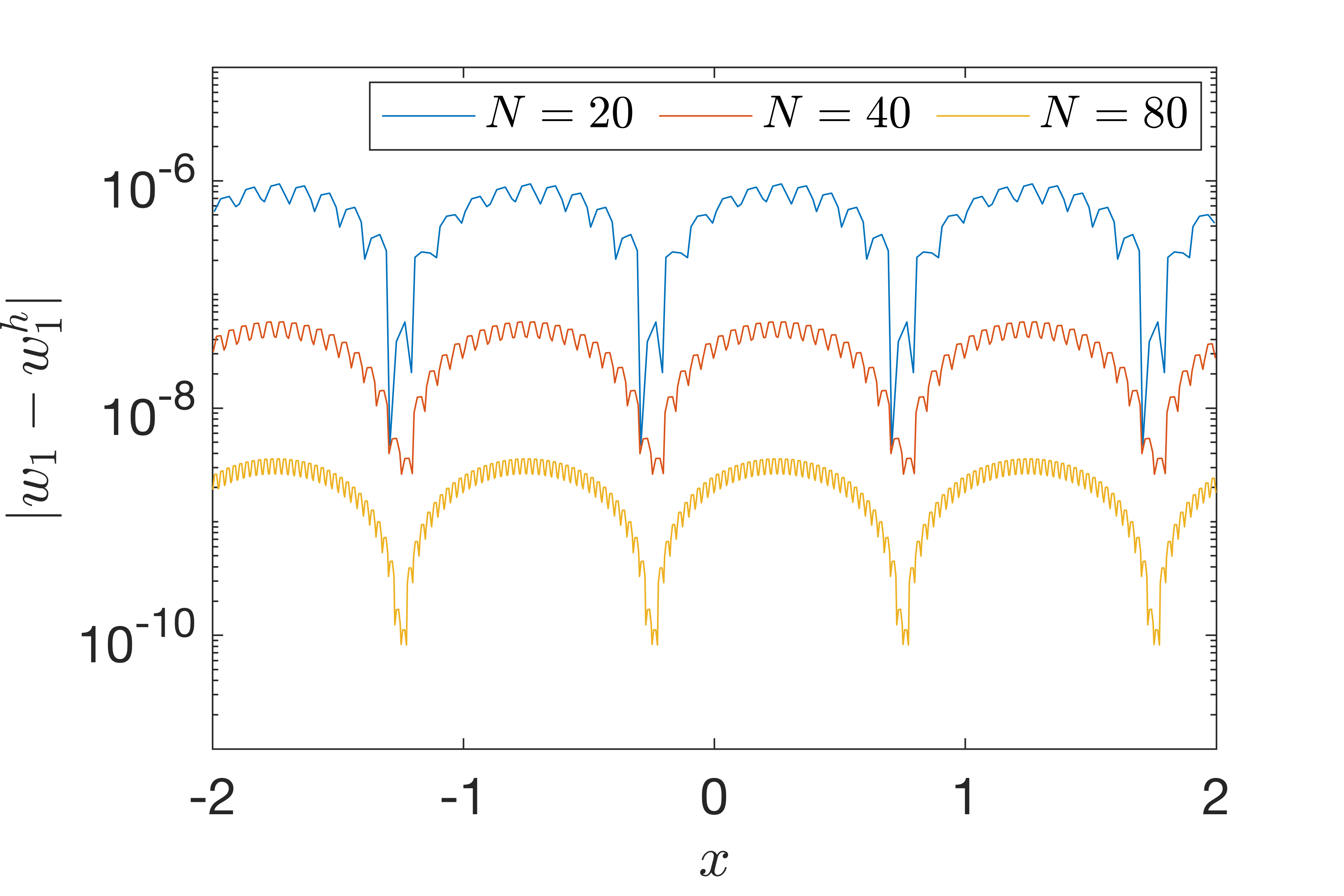}
	\includegraphics[width=0.45\linewidth,trim={.0cm .0cm .0cm .0cm},clip]{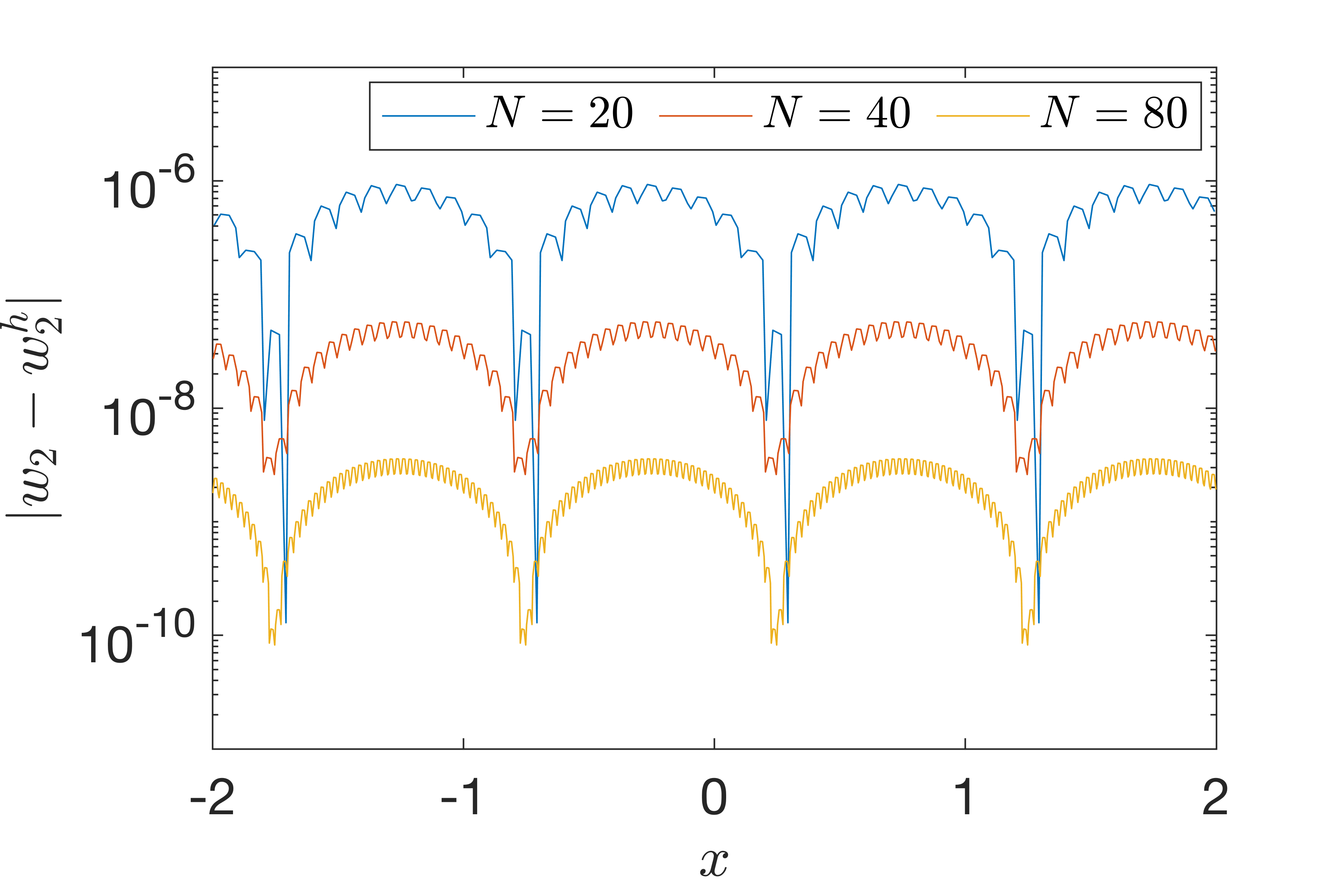}\\
	\includegraphics[width=0.45\linewidth,trim={.0cm .0cm .0cm .0cm},clip]{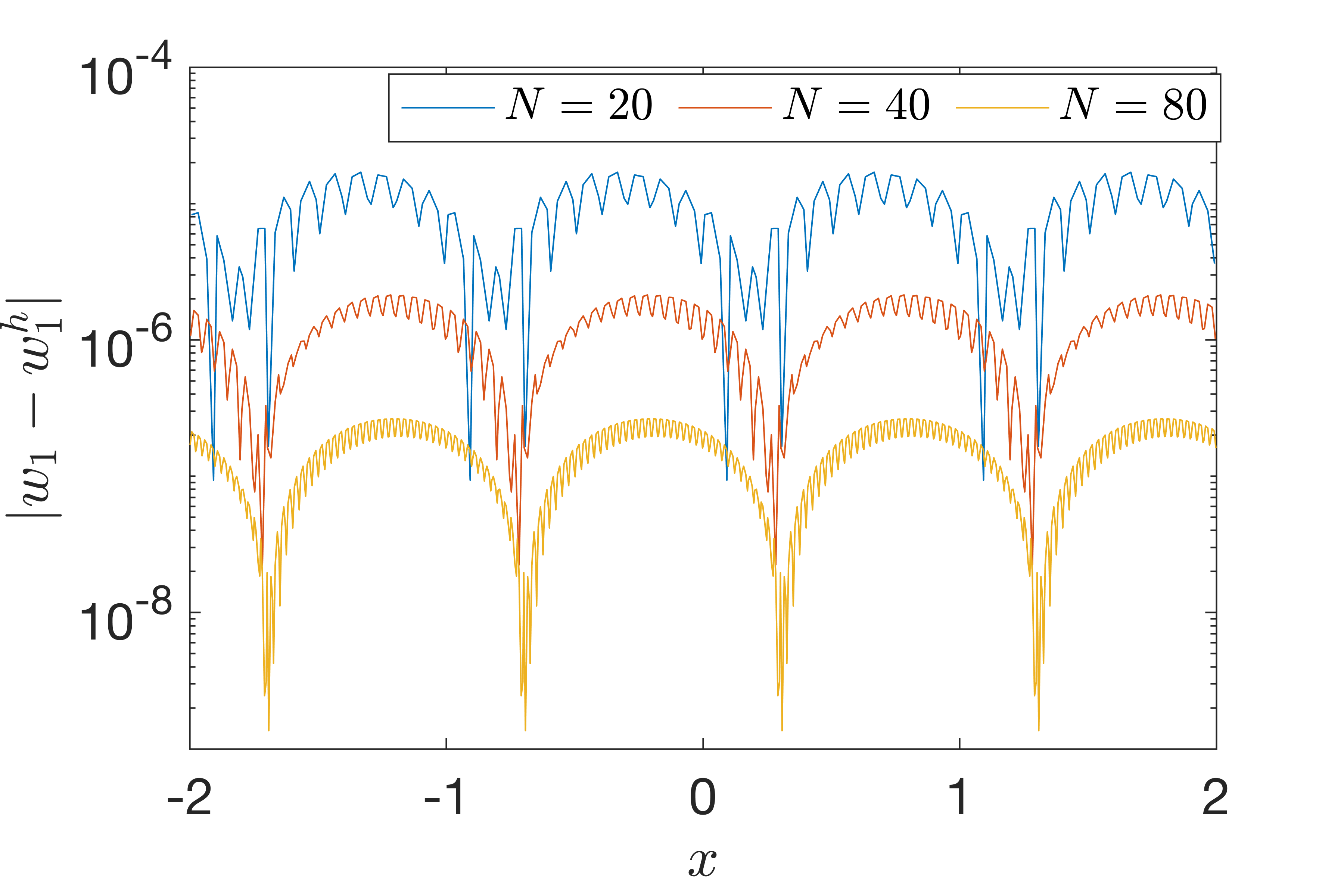}
	\includegraphics[width=0.45\linewidth,trim={.0cm .0cm .0cm .0cm},clip]{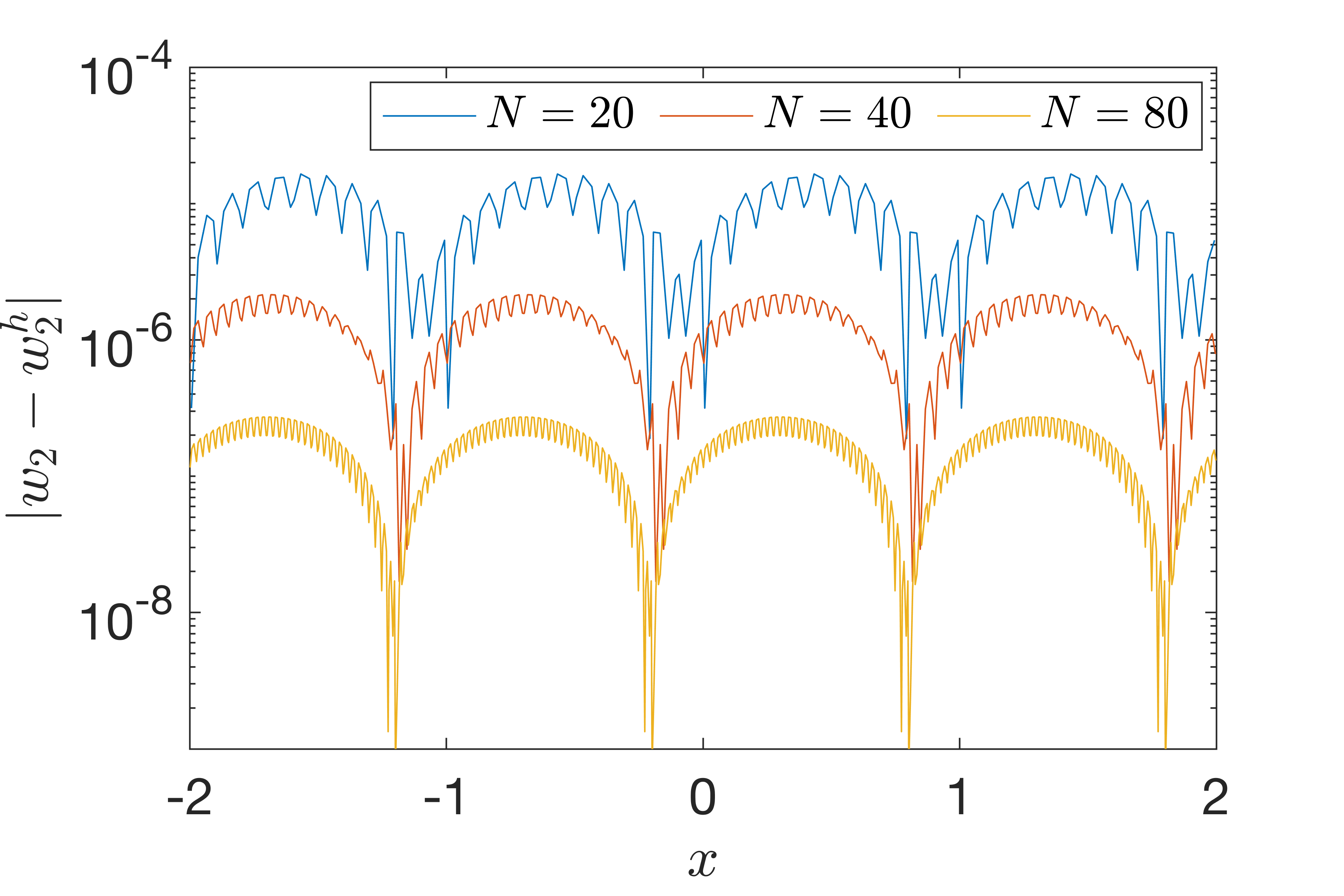}\\
	\includegraphics[width=0.45\linewidth,trim={.0cm .0cm .0cm .0cm},clip]{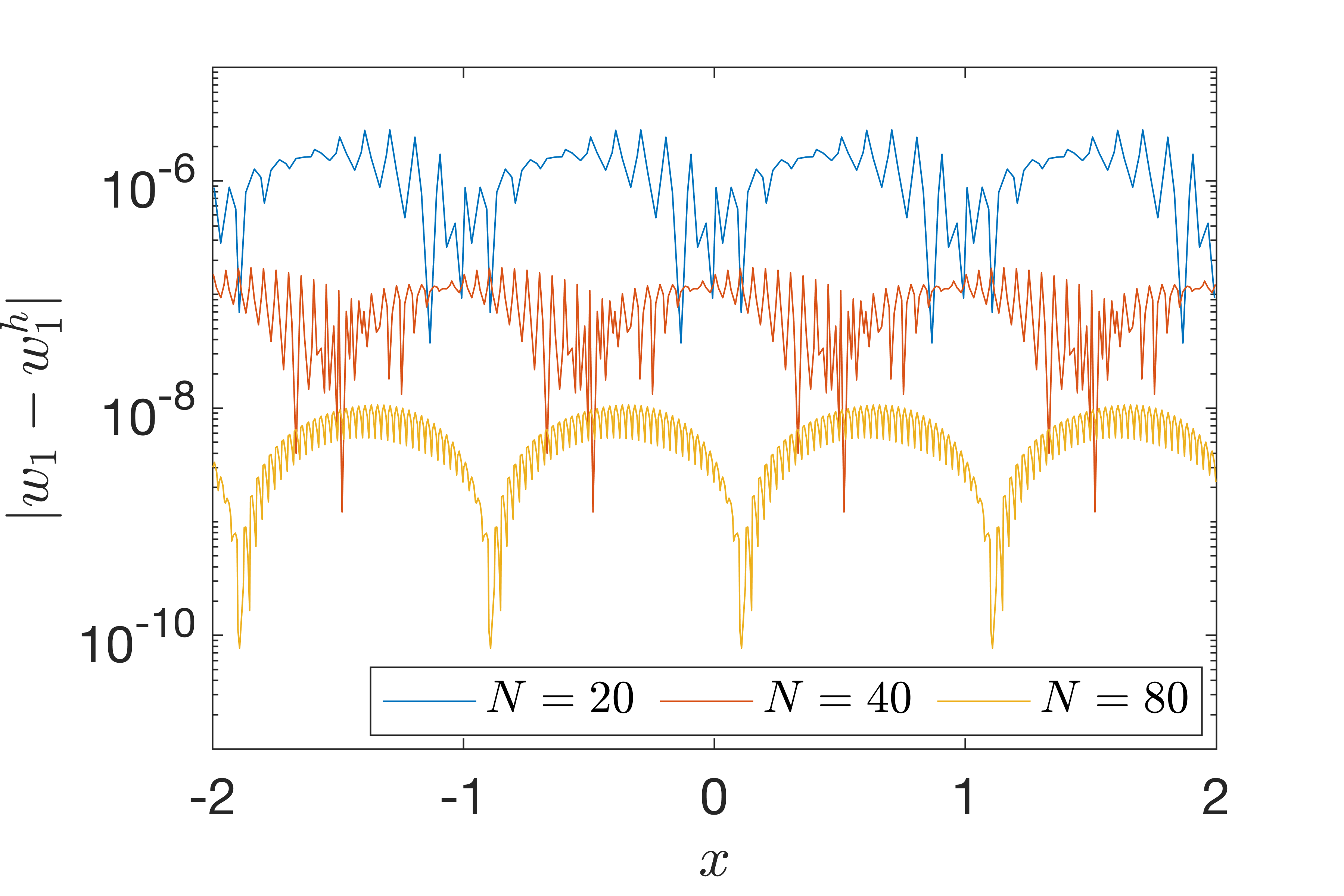}
	\includegraphics[width=0.45\linewidth,trim={.0cm .0cm .0cm .0cm},clip]{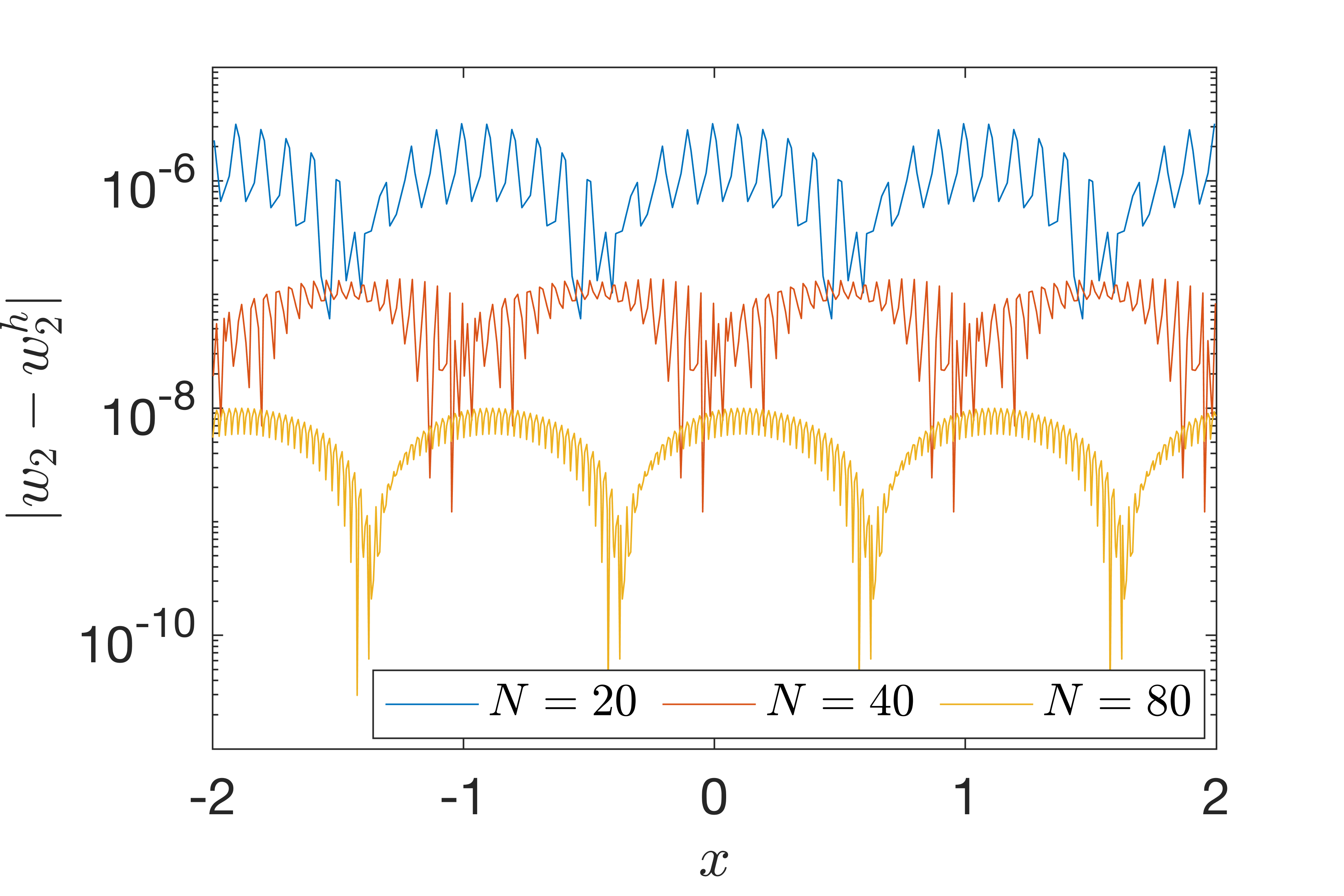}\\
	\includegraphics[width=0.45\linewidth,trim={.0cm .0cm .0cm .0cm},clip]{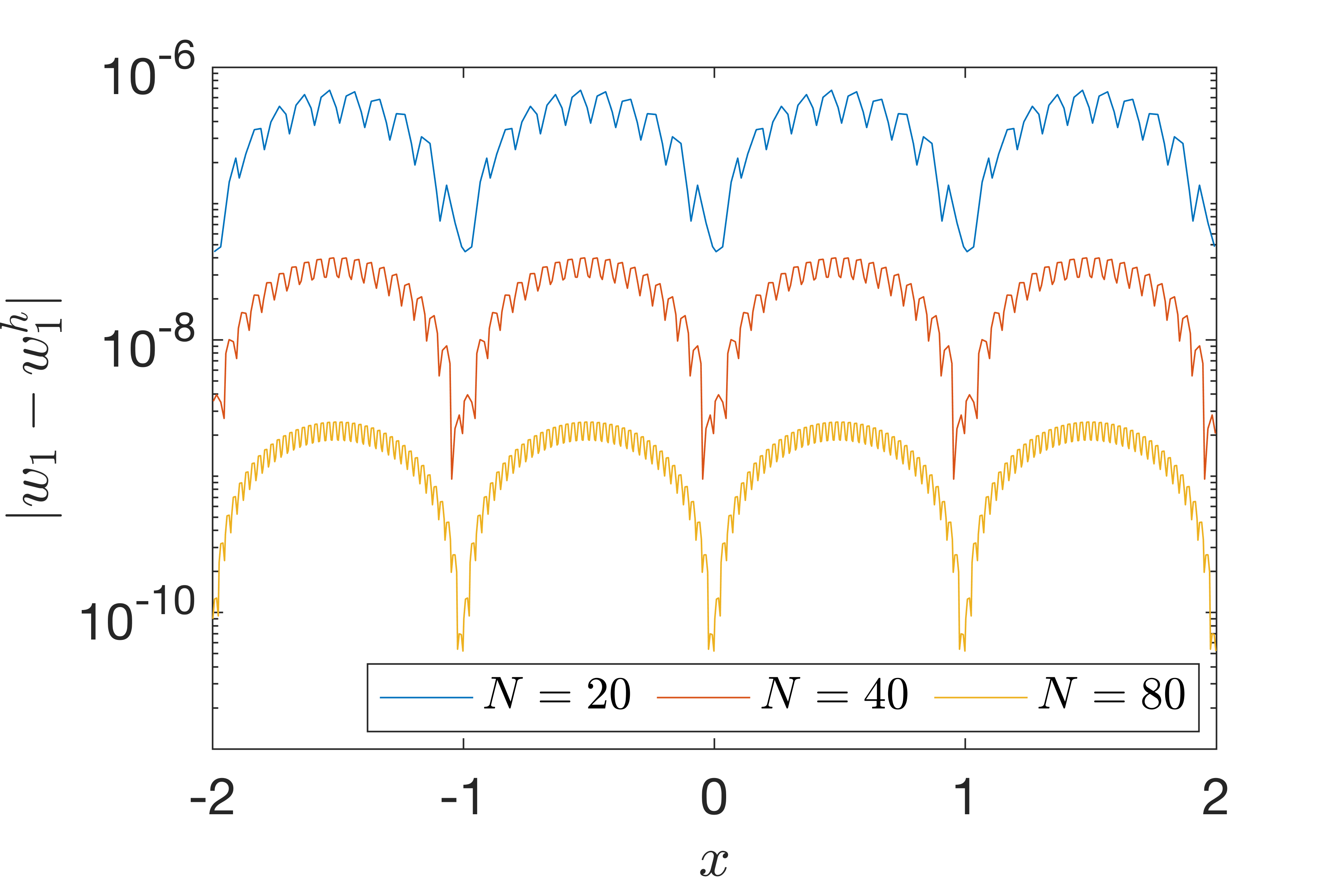}
	\includegraphics[width=0.45\linewidth,trim={.0cm .0cm .0cm .0cm},clip]{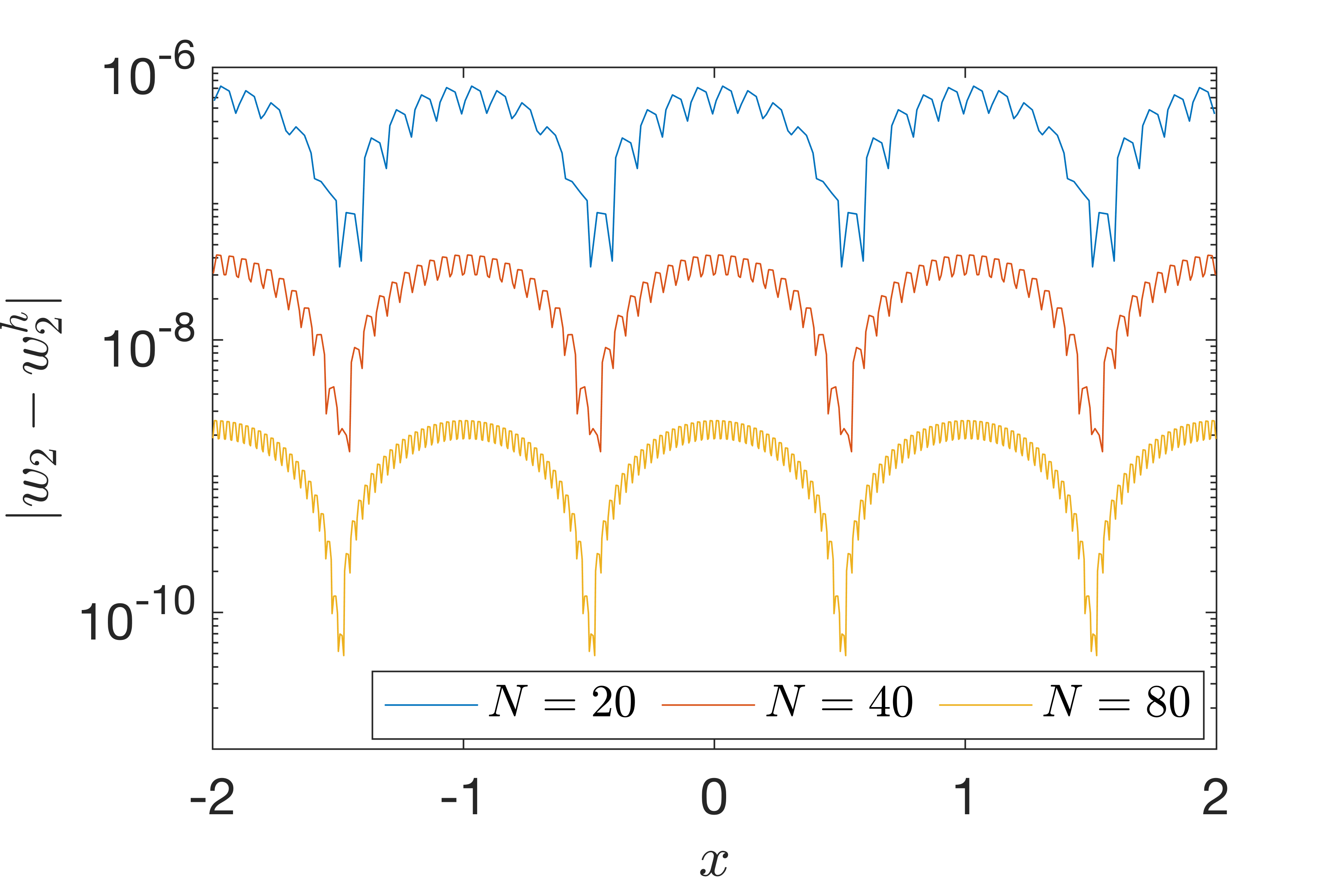}
\caption{From the left to the right, we plot the errors $|w_1 - w_1^h|$ and $|w_2 - w_2^h|$ over the spatial location $x$ for the problem (\ref{example1_sol}) using $\mathcal{P}^3$ polynomial on a uniform mesh of $N = 20, 40, 80$; from the top to the bottom are the results by using the upwind fluxes (\ref{u_flux}), the central fluxes (\ref{c_flux}), the mixed central fluxes (\ref{mc_flux}), and the mixed upwind fluxes (\ref{mu_flux}), respectively. }\label{fig:errors_eg1}
\end{figure}
 
 \subsubsection{Homogeneous Dirichlet boundary conditions}
For the second example, we test a problem where the solution data decay to $0$ in the region near physical boundaries, i.e., near $x_a=-2$ and $x_b=2$. For simplicity, we still consider the same problem (\ref{example1})--(\ref{nonlinear_term_numer}), but with the following manufactured solution
\begin{equation}\label{example2_sol}
w_1 = \sqrt{2}\Big(\cos(2\pi t)e^{-\frac{x^2}{0.01}} + 2\cos(4\pi t)e^{-\frac{x^2}{0.025}}\Big),\quad w_2 =  \sqrt{2}\Big(\cos(2\pi t)e^{-\frac{x^2}{0.01}} - 2\cos(4\pi t)e^{-\frac{x^2}{0.025}}\Big).
\end{equation}

With the same spatial discretization as conducted in Section \ref{sec:example1}, we also present the $L^2$ errors in solutions $w_1, w_2, b_1, b_2$ from Table \ref{example_two_l2_up} to Table \ref{example_two_l2_mu} with four different numerical fluxes defined from (\ref{u_flux}) to (\ref{mc_flux}). We observe the same results with the first example (\ref{example1_sol}): optimal convergence order $q+1$ for the dissipating schemes, the upwind flux (\ref{u_flux}) in Table \ref{example_two_l2_up} and the mixed upwind scheme (\ref{mu_flux}) in Table \ref{example_one_l2_mu}; super-convergence rate $q+2$ for the central flux (\ref{c_flux}) when $q$ is even, while a sub-optimal convergence rate $q$ when $q$ is odd as presented in Table \ref{example_two_l2_central}; when the mixed central flux (\ref{mc_flux}) is implemented displayed in Table \ref{example_two_l2_mc}, an optimal convergence order $q+1$ is obtained, but there are some fluctuations for this energy conservative scheme.

\begin{table}
 	\footnotesize
 	\begin{center}
 		\scalebox{1.0}{
 			\begin{tabular}{c c c c c c c c c c c c c c c}
 				\hline
 				~ & ~ &  $w_1$ &~ & $w_2$& ~ & $b_1$ & ~ & $b_2$ & ~\\
 				\cline{3-10}
 				$q$ & $N$ & $L^2$ error & order   & $L^2$ error & order  & $L^2$ error & order & $L^2$ error & order \\
 				\hline
 				0 & 40 & 3.4682e-01 & -- & 1.6750e-01 & -- & 2.7693e-01 & -- & 4.6903e-01 & --  \\
~ & 80 & 2.2832e-01 & 0.6032 & 1.8108e-01 & -0.1125 & 2.0982e-01 & 0.4003 & 3.5470e-01 & 0.4031 \\
~ & 160 & 1.4636e-01 & 0.6415 & 1.5203e-01 & 0.2523 & 1.7072e-01 & 0.2975 & 2.4479e-01 & 0.5351 \\
~ & 320 & 9.2281e-02 & 0.6654 & 1.0447e-01 & 0.5413 & 1.2168e-01 & 0.4886 & 1.5509e-01 & 0.6585 \\
~ & 640 & 5.5131e-02 & 0.7432 & 6.2669e-02 & 0.7372 & 7.5532e-02 & 0.6879 & 9.0712e-02 & 0.7737 \\
 				~ & ~ & ~ & ~ & ~ & ~ & ~ & ~ & ~ & ~\\
 				1 & 40 & 6.5776e-02 & -- & 4.1806e-02 & -- & 7.8510e-02 & -- & 7.7361e-02 & --  \\
~ & 80 & 1.6044e-02 & 2.0355 & 7.7474e-03 & 2.4319 & 1.8100e-02 & 2.1169 & 1.7530e-02 & 2.1418 \\
~ & 160 & 3.7870e-03 & 2.0829 & 1.4300e-03 & 2.4377 & 4.1830e-03 & 2.1134 & 3.9084e-03 & 2.1651 \\
~ & 320 & 9.2584e-04 & 2.0322 & 3.1078e-04 & 2.2021 & 1.0144e-03 & 2.0439 & 9.3726e-04 & 2.0601 \\
~ & 640 & 2.3002e-04 & 2.0090 & 7.4325e-05 & 2.0640 & 2.5150e-04 & 2.0120 & 2.3155e-04 & 2.0172 \\
 			    ~ & ~ & ~ & ~ & ~ & ~ & ~ & ~ & ~ & ~\\
 				2 & 40 & 9.0106e-03 & -- & 4.0634e-03 & -- & 7.4050e-03 & -- & 1.1856e-02 & --  \\
~ & 80 & 1.0237e-03 & 3.1378 & 4.1478e-04 & 3.2923 & 9.3769e-04 & 2.9813 & 1.2493e-03 & 3.2465 \\
~ & 160 & 1.2889e-04 & 2.9896 & 5.2273e-05 & 2.9882 & 1.1600e-04 & 3.0150 & 1.5885e-04 & 2.9754 \\
~ & 320 & 1.6149e-05 & 2.9966 & 6.5466e-06 & 2.9972 & 1.4479e-05 & 3.0021 & 1.9941e-05 & 2.9939 \\
~ & 640 & 2.0200e-06 & 2.9990 & 8.1840e-07 & 2.9999 & 1.8099e-06 & 2.9999 & 2.4949e-06 & 2.9987 \\
 				~ & ~ & ~ & ~ & ~ & ~ & ~ & ~ & ~ & ~\\
 				3 & 40 & 6.9888e-04 & -- & 2.9768e-04 & -- & 9.5778e-04 & -- & 4.8660e-04 & --  \\
~ & 80 & 6.9326e-05 & 3.3336 & 3.7505e-05 & 2.9886 & 5.0811e-05 & 4.2365 & 9.9215e-05 & 2.2941 \\
~ & 160 & 4.3839e-06 & 3.9831 & 2.3764e-06 & 3.9802 & 3.1020e-06 & 4.0339 & 6.3332e-06 & 3.9696 \\
~ & 320 & 2.7492e-07 & 3.9951 & 1.4898e-07 & 3.9955 & 1.9294e-07 & 4.0070 & 3.9790e-07 & 3.9924 \\
~ & 640 & 1.7199e-08 & 3.9986 & 9.3152e-09 & 3.9994 & 1.2048e-08 & 4.0012 & 2.4899e-08 & 3.9982 \\
 				\hline
 			\end{tabular}
 		}
 	\end{center}
 	\caption{\scriptsize{$L^2$ errors and the corresponding convergence rates for $w_1, w_2, b_1, b_2$ of problem (\ref{example2_sol}) using $\mathcal{P}^q$ polynomials and the upwind flux (\ref{u_flux}).  The interval is divided into $N$ uniform cells, and the terminal computational time is $T = 1$.}}\label{example_two_l2_up}
 \end{table}
 
 \begin{table}
 	\footnotesize
 	\begin{center}
 		\scalebox{1.0}{
 			\begin{tabular}{c c c c c c c c c c c c c c c}
 				\hline
 				~ & ~ &  $w_1$ &~ & $w_2$& ~ & $b_1$ & ~ & $b_2$ & ~\\
 				\cline{3-10}
 				$q$ & $N$ & $L^2$ error & order   & $L^2$ error & order  & $L^2$ error & order & $L^2$ error & order \\
 				\hline
 				0 & 40 & 4.3279e-01 & -- & 8.6203e-01 & -- & 8.5573e-01 & -- & 1.0623e+00 & --  \\
~ & 80 & 1.5428e-01 & 1.4882 & 1.3085e-01 & 2.7199 & 2.1157e-01 & 2.0160 & 1.9256e-01 & 2.4638 \\
~ & 160 & 3.4830e-02 & 2.1471 & 2.9357e-02 & 2.1561 & 4.7421e-02 & 2.1576 & 4.3602e-02 & 2.1428 \\
~ & 320 & 8.4962e-03 & 2.0355 & 7.1608e-03 & 2.0355 & 1.1543e-02 & 2.0385 & 1.0662e-02 & 2.0319 \\
~ & 640 & 2.1117e-03 & 2.0084 & 1.7798e-03 & 2.0084 & 2.8679e-03 & 2.0090 & 2.6513e-03 & 2.0077 \\
 				~ & ~ & ~ & ~ & ~ & ~ & ~ & ~ & ~ & ~\\
 				1 & 40 & 1.5483e-01 & -- & 1.0178e-01 & -- & 1.6979e-01 & -- & 1.9959e-01 & --  \\
~ & 80 & 6.7997e-02 & 1.1872 & 4.4655e-02 & 1.1886 & 7.1201e-02 & 1.2538 & 9.0364e-02 & 1.1432  \\
~ & 160 & 3.2712e-02 & 1.0557 & 2.1182e-02 & 1.0759 & 3.3579e-02 & 1.0843 & 4.3703e-02 & 1.0480 \\
~ & 320 & 1.6196e-02 & 1.0142 & 1.0453e-02 & 1.0189 & 1.6534e-02 & 1.0221 & 2.1675e-02 & 1.0117 \\
~ & 640 & 8.0776e-03 & 1.0036 & 5.2092e-03 & 1.0048 & 8.2345e-03 & 1.0057 & 1.0815e-02 & 1.0030 \\
 			    ~ & ~ & ~ & ~ & ~ & ~ & ~ & ~ & ~ & ~\\
 				2 & 40 & 5.4484e-02 & -- & 5.5823e-02 & -- & 7.8592e-02 & -- & 7.7412e-02 & --  \\
~ & 80 & 7.2114e-04 & 6.2394 & 4.2114e-04 & 7.0504 & 9.7526e-04 & 6.3325 & 6.6609e-04 & 6.8607 \\
~ & 160 & 4.2372e-05 & 4.0891 & 2.4013e-05 & 4.1324 & 5.7785e-05 & 4.0770 & 3.7483e-05 & 4.1514 \\
~ & 320 & 2.3909e-06 & 4.1475 & 1.4440e-06 & 4.0556 & 3.2248e-06 & 4.1634 & 2.2813e-06 & 4.0383 \\
~ & 640 & 1.4802e-07 & 4.0137 & 8.4411e-08 & 4.0965 & 1.9563e-07 & 4.0430 & 1.4071e-07 & 4.0191 \\
 				~ & ~ & ~ & ~ & ~ & ~ & ~ & ~ & ~ & ~\\
 				3 & 40 & 9.6599e-04 & -- & 3.8383e-04 & -- & 9.9817e-04 & -- & 1.0791e-03 & --  \\
~ & 80 & 1.4126e-04 & 2.7736 & 5.7700e-05 & 2.7338 & 1.6239e-04 & 2.6199 & 1.4212e-04 & 2.9247 \\
~ & 160 & 1.6458e-05 & 3.1015 & 7.1899e-06 & 3.0045 & 2.0128e-05 & 3.0121 & 1.5490e-05 & 3.1977 \\
~ & 320 & 1.9666e-06 & 3.0650 & 7.9819e-07 & 3.1712 & 2.3930e-06 & 3.0723 & 1.8120e-06 & 3.0957 \\
~ & 640 & 2.4608e-07 & 2.9985 & 9.9923e-08 & 2.9978 & 3.0224e-07 & 2.9850 & 2.2301e-07 & 3.0224 \\
 				\hline
 			\end{tabular}
 		}
 	\end{center}
 	\caption{\scriptsize{$L^2$ errors and the corresponding convergence rates for $w_1, w_2, b_1, b_2$ of problem (\ref{example2_sol}) using $\mathcal{P}^q$ polynomials and the central flux (\ref{c_flux}).  The interval is divided into $N$ uniform cells, and the terminal computational time is $T = 1$.}}\label{example_two_l2_central}
 \end{table}
 
 \begin{table}
 	\footnotesize
 	\begin{center}
 		\scalebox{1.0}{
 			\begin{tabular}{c c c c c c c c c c c c c c c}
 				\hline
 				~ & ~ &  $w_1$ &~ & $w_2$& ~ & $b_1$ & ~ & $b_2$ & ~\\
 				\cline{3-10}
 				$q$ & $N$ & $L^2$ error & order   & $L^2$ error & order  & $L^2$ error & order & $L^2$ error & order \\
 				\hline
 				0 & 40 & 1.1420e+00 & -- & 1.0684e+00 & -- & 1.6597e+00 & -- & 1.4617e+00 & --  \\
~ & 80 & 5.6375e-01 & 1.0184 & 5.7341e-01 & 0.8978 & 8.5070e-01 & 0.9642 & 7.5467e-01 & 0.9538 \\
~ & 160 & 2.7904e-01 & 1.0146 & 2.9644e-01 & 0.9518 & 4.3075e-01 & 0.9818 & 3.8201e-01 & 0.9822 \\
~ & 320 & 1.3942e-01 & 1.0011 & 1.5073e-01 & 0.9757 & 2.1747e-01 & 0.9860 & 1.9241e-01 & 0.9894 \\
~ & 640 & 6.9812e-02 & 0.9978 & 7.5997e-02 & 0.9880 & 1.0940e-01 & 0.9912 & 9.6589e-02 & 0.9942 \\
 				~ & ~ & ~ & ~ & ~ & ~ & ~ & ~ & ~ & ~\\
 				1 & 40 & 6.0056e-02 & -- & 8.9200e-02 & -- & 1.0663e-01 & -- & 1.0843e-01 & --  \\
~ & 80 & 1.2054e-02 & 2.3168 & 2.0543e-02 & 2.1184 & 2.6396e-02 & 2.0143 & 2.0926e-02 & 2.3733 \\
~ & 160 & 3.1799e-03 & 1.9224 & 2.5745e-03 & 2.9963 & 3.9940e-03 & 2.7244 & 4.1866e-03 & 2.3214 \\
~ & 320 & 7.9069e-04 & 2.0078 & 1.4214e-03 & 0.8570 & 8.3806e-04 & 2.2527 & 2.1421e-03 & 0.9668 \\
~ & 640 & 2.2029e-04 & 1.8437 & 1.5592e-04 & 3.1884 & 3.3500e-04 & 1.3229 & 1.8288e-04 & 3.5501 \\
 			    ~ & ~ & ~ & ~ & ~ & ~ & ~ & ~ & ~ & ~\\
 				2 & 40 & 5.9545e-03 & -- & 1.4531e-02 & -- & 1.8142e-02 & -- & 1.2809e-02 & --  \\
~ & 80 & 8.5996e-04 & 2.7916 & 1.3745e-03 & 3.4022 & 1.6622e-03 & 3.4482 & 1.5794e-03 & 3.0197 \\
~ & 160 & 9.9033e-05 & 3.1183 & 1.8951e-04 & 2.8585 & 2.1876e-04 & 2.9257 & 2.0878e-04 & 2.9193 \\
~ & 320 & 1.1776e-05 & 3.0721 & 2.3327e-05 & 3.0222 & 2.3848e-05 & 3.1974 & 2.8230e-05 & 2.8867 \\
~ & 640 & 1.3547e-06 & 3.1197 & 7.7926e-07 & 4.9037 & 1.5916e-06 & 3.9053 & 1.5336e-06 & 4.2022 \\
 				~ & ~ & ~ & ~ & ~ & ~ & ~ & ~ & ~ & ~\\
 				3 & 40 & 5.5938e-04 & -- & 9.2006e-04 & -- & 1.2434e-03 & -- & 8.7905e-04 & --  \\
~ & 80 & 6.0453e-05 & 3.2100 & 9.4019e-05 & 3.2907 & 1.1651e-04 & 3.4158 & 1.0684e-04 & 3.0405 \\
~ & 160 & 3.9276e-06 & 3.9441 & 4.9625e-06 & 4.2438 & 6.0736e-06 & 4.2617 & 6.5739e-06 & 4.0225 \\
~ & 320 & 3.4942e-07 & 3.4906 & 3.6036e-07 & 3.7836 & 6.6399e-07 & 3.1933 & 2.5104e-07 & 4.7108 \\
~ & 640 & 1.6524e-08 & 4.4023 & 2.9671e-08 & 3.6023 & 4.2447e-08 & 3.9674 & 2.2473e-08 & 3.4817 \\
 				\hline
 			\end{tabular}
 		}
 	\end{center}
 	\caption{\scriptsize{$L^2$ errors and the corresponding convergence rates for $w_1, w_2, b_1, b_2$ of problem (\ref{example2_sol}) using $\mathcal{P}^q$ polynomials and the mixed central flux (\ref{mc_flux}).  The interval is divided into $N$ uniform cells, and the terminal computational time is $T = 1$.}}\label{example_two_l2_mc}
 \end{table}
 
 \begin{table}
 	\footnotesize
 	\begin{center}
 		\scalebox{1.0}{
 			\begin{tabular}{c c c c c c c c c c c c c c c}
 				\hline
 				~ & ~ &  $w_1$ &~ & $w_2$& ~ & $b_1$ & ~ & $b_2$ & ~\\
 				\cline{3-10}
 				$q$ & $N$ & $L^2$ error & order   & $L^2$ error & order  & $L^2$ error & order & $L^2$ error & order \\
 				\hline
 				0 & 40 & 4.7347e-01 & -- & 2.4353e-01 & -- & 6.0749e-01 & -- & 4.4487e-01 & --  \\
~ & 80 & 3.5425e-01 & 0.4185 & 1.8822e-01 & 0.3716 & 4.3451e-01 & 0.4835 & 3.6475e-01 & 0.2865 \\
~ & 160 & 2.5777e-01 & 0.4587 & 1.1979e-01 & 0.6519 & 2.9064e-01 & 0.5801 & 2.7771e-01 & 0.3933 \\
~ & 320 & 1.7288e-01 & 0.5764 & 6.9034e-02 & 0.7951 & 1.8549e-01 & 0.6479 & 1.8681e-01 & 0.5720 \\
~ & 640 & 1.0479e-01 & 0.7223 & 3.8201e-02 & 0.8537 & 1.1043e-01 & 0.7482 & 1.1262e-01 & 0.7301 \\
 				~ & ~ & ~ & ~ & ~ & ~ & ~ & ~ & ~ & ~\\
 				1 & 40 & 6.5922e-02 & -- & 3.3550e-02 & -- & 5.8441e-02 & -- & 8.6760e-02 & --  \\
~ & 80 & 1.0901e-02 & 2.5964 & 8.6228e-03 & 1.9601 & 8.8293e-03 & 2.7266 & 1.7561e-02 & 2.3046 \\
~ & 160 & 2.2018e-03 & 2.3076 & 1.9331e-03 & 2.1572 & 1.5023e-03 & 2.5552 & 3.8618e-03 & 2.1850 \\
~ & 320 & 5.1385e-04 & 2.0993 & 4.6666e-04 & 2.0505 & 3.1504e-04 & 2.2535 & 9.2972e-04 & 2.0544 \\
~ & 640 & 1.2616e-04 & 2.0261 & 1.1560e-04 & 2.0133 & 7.4564e-05 & 2.0790 & 2.3021e-04 & 2.0138 \\
 			    ~ & ~ & ~ & ~ & ~ & ~ & ~ & ~ & ~ & ~\\
 				2 & 40 & 5.9664e-03 & -- & 6.5599e-03 & -- & 5.5459e-03 & -- & 1.1247e-02 & --  \\
~ & 80 & 9.1259e-04 & 2.7088 & 9.5240e-04 & 2.7840 & 8.4156e-04 & 2.7203 & 1.6648e-03 & 2.7562 \\
~ & 160 & 1.1692e-04 & 2.9645 & 1.4525e-04 & 2.7131 & 1.0732e-04 & 2.9712 & 2.4087e-04 & 2.7890 \\
~ & 320 & 1.4532e-05 & 3.0082 & 1.9448e-05 & 2.9008 & 1.3222e-05 & 3.0209 & 3.1686e-05 & 2.9263 \\
~ & 640 & 1.8113e-06 & 3.0041 & 2.4783e-06 & 2.9722 & 1.6409e-06 & 3.0104 & 4.0191e-06 & 2.9789 \\
 				~ & ~ & ~ & ~ & ~ & ~ & ~ & ~ & ~ & ~\\
 				3 & 40 & 6.2444e-04 & -- & 4.5116e-04 & -- & 5.8689e-04 & -- & 9.1789e-04 & --  \\
~ & 80 & 2.9422e-05 & 4.4076 & 6.5006e-05 & 2.7950 & 5.3307e-05 & 3.4607 & 8.5681e-05 & 3.4213 \\
~ & 160 & 1.6370e-06 & 4.1678 & 3.4026e-06 & 4.2559 & 2.6600e-06 & 4.3248 & 4.6302e-06 & 4.2098 \\
~ & 320 & 9.7929e-08 & 4.0631 & 2.0255e-07 & 4.0703 & 1.5352e-07 & 4.1149 & 2.7869e-07 & 4.0543 \\
~ & 640 & 6.0479e-09 & 4.0172 & 1.2504e-08 & 4.0178 & 9.3843e-09 & 4.0321 & 1.7257e-08 & 4.0134 \\
 				\hline
 			\end{tabular}
 		}
 	\end{center}
 	\caption{\scriptsize{$L^2$ errors and the corresponding convergence rates for $w_1, w_2, b_1, b_2$ of problem (\ref{example2_sol}) using $\mathcal{P}^q$ polynomials and the mixed upwind flux (\ref{mu_flux}).  The interval is divided into $N$ uniform cells, and the terminal computational time is $T = 1$.}}\label{example_two_l2_mu}
 \end{table}
 
 \begin{figure}[htb!]
	\centering
	\includegraphics[width=0.45\linewidth,trim={.0cm .0cm .0cm .0cm},clip]{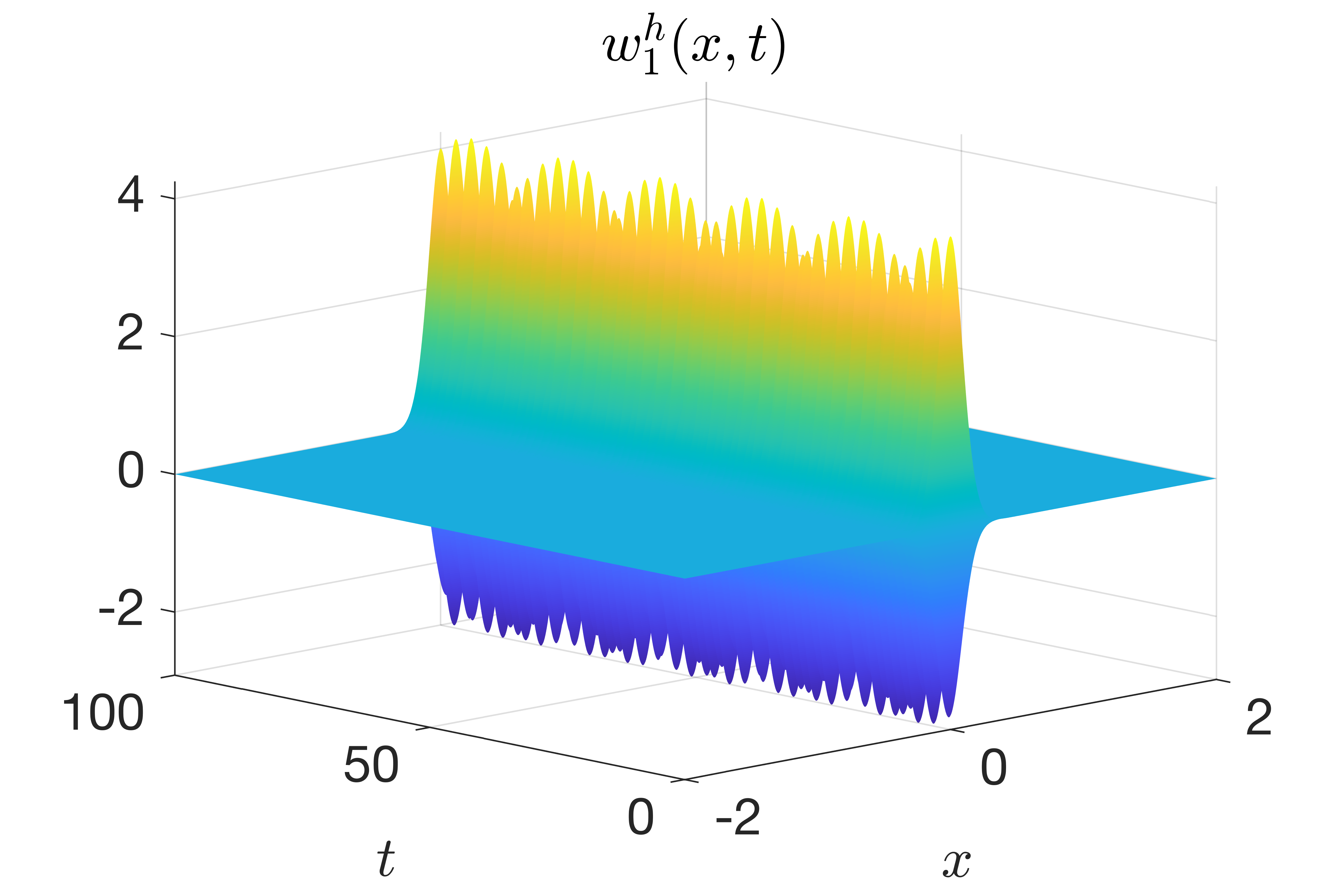}
	\includegraphics[width=0.45\linewidth,trim={.0cm .0cm .0cm .0cm},clip]{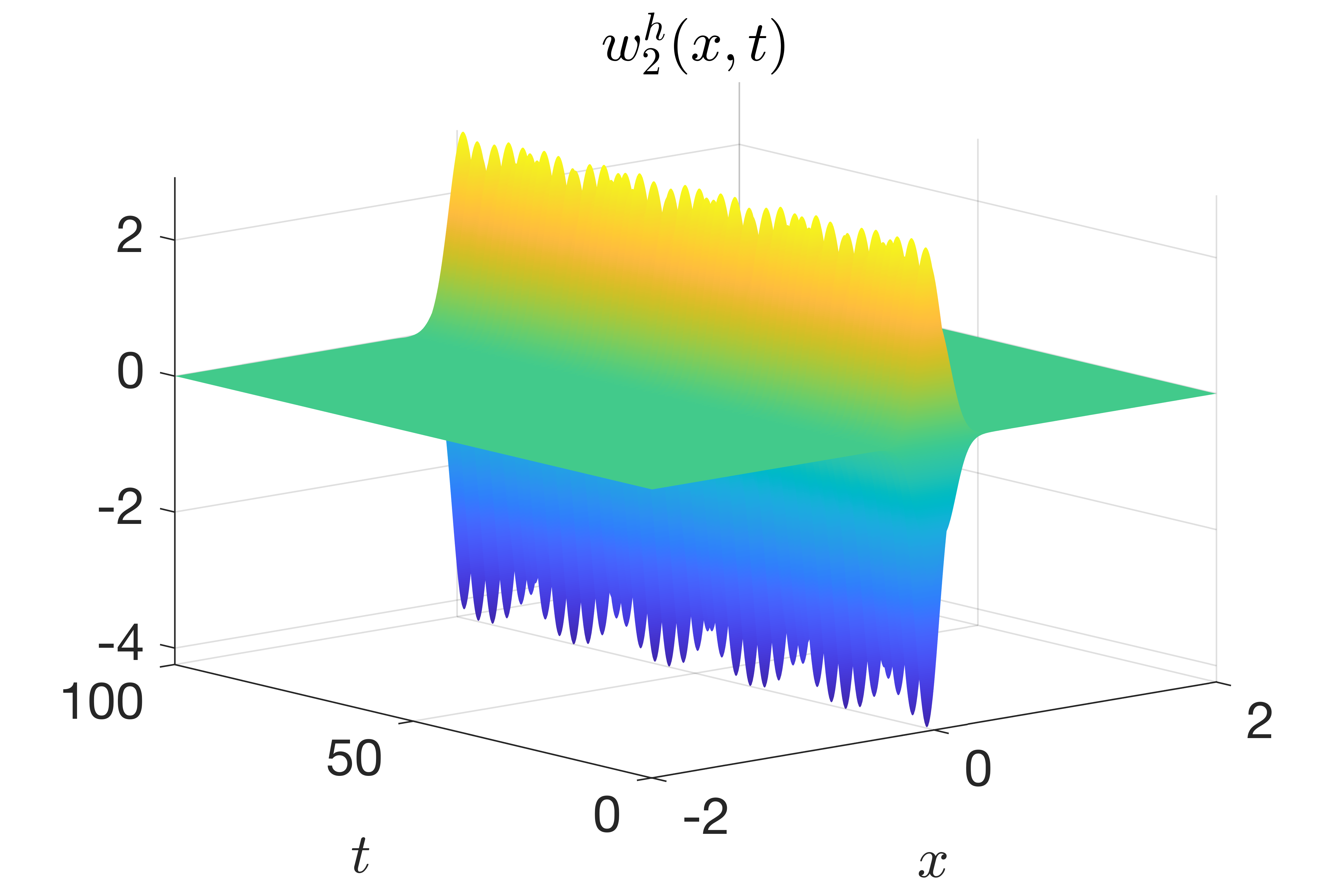}
\caption{Space-time plots of the solutions of problem (\ref{example2_sol}) with the degree of approximation space $q=3$. The upwind fluxes (\ref{u_flux}) is used in the simulation.}\label{fig:w1_w2_eg2}
\end{figure}

 We also present the solutions of the problem (\ref{example2_sol}) in Figure \ref{fig:solutions_w1_w2_eg2} and Figure \ref{fig:w1_w2_eg2}. In the simulation, the approximation degrees for $w_1^h$ and $w_2^h$ are $q = 3$, and the upwind flux (\ref{u_flux}) is used. Figure \ref{fig:solutions_w1_w2_eg2} displays both the numerical solutions $w_1^h, w_2^h$ and the exact solutions $w_1, w_2$ at $t = 0, 30, 75, 100$; while Figure \ref{fig:w1_w2_eg2} presents the space-time plot of the solutions (\ref{example2_sol}) from $t = 0$ to $t = 100$. We find that the numerical results match well with the analytic solution for a very long time.
\begin{figure}[htb!]
	\centering
	\includegraphics[width=0.45\linewidth,trim={.0cm .0cm .0cm .0cm},clip]{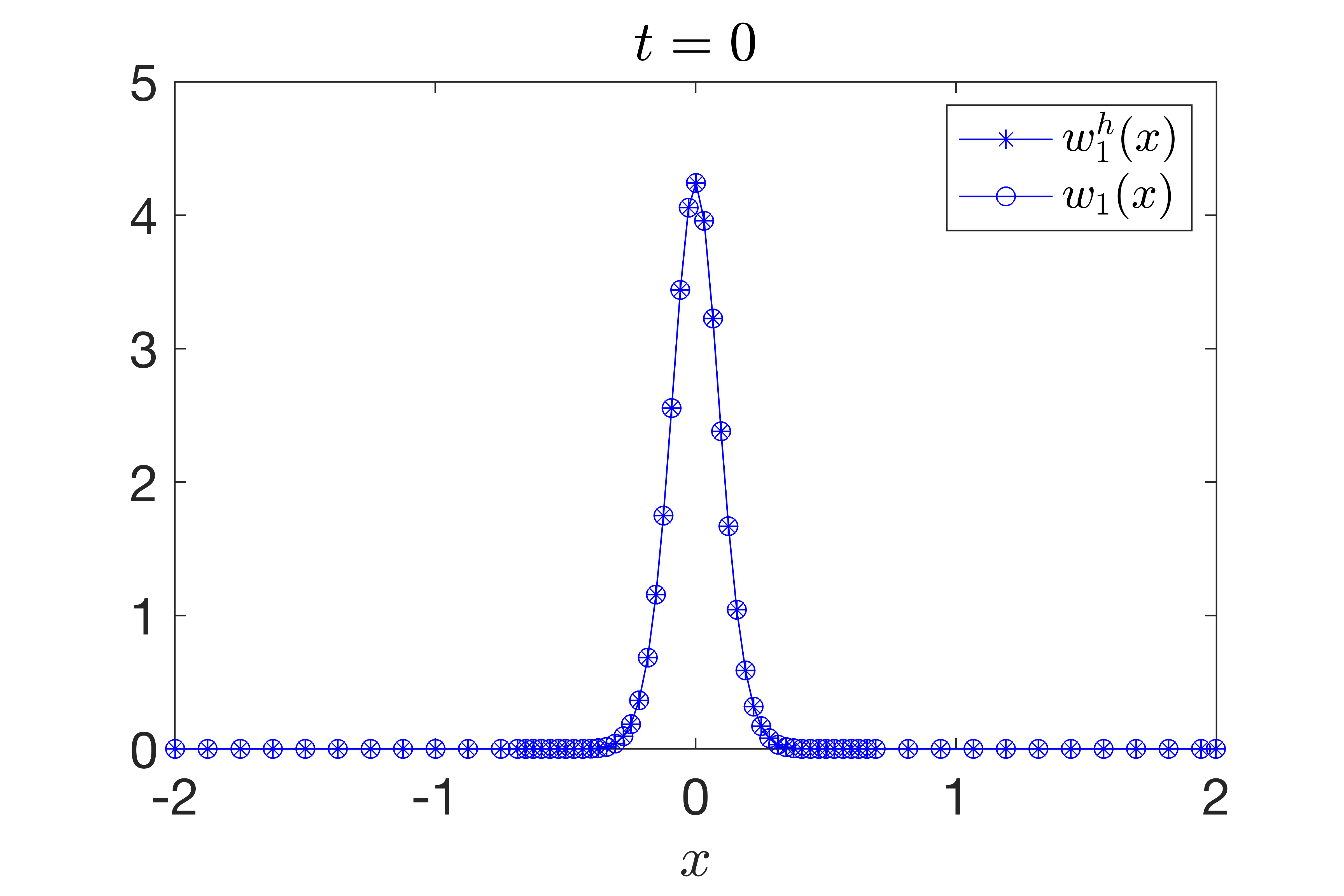}
	\includegraphics[width=0.45\linewidth,trim={.0cm .0cm .0cm .0cm},clip]{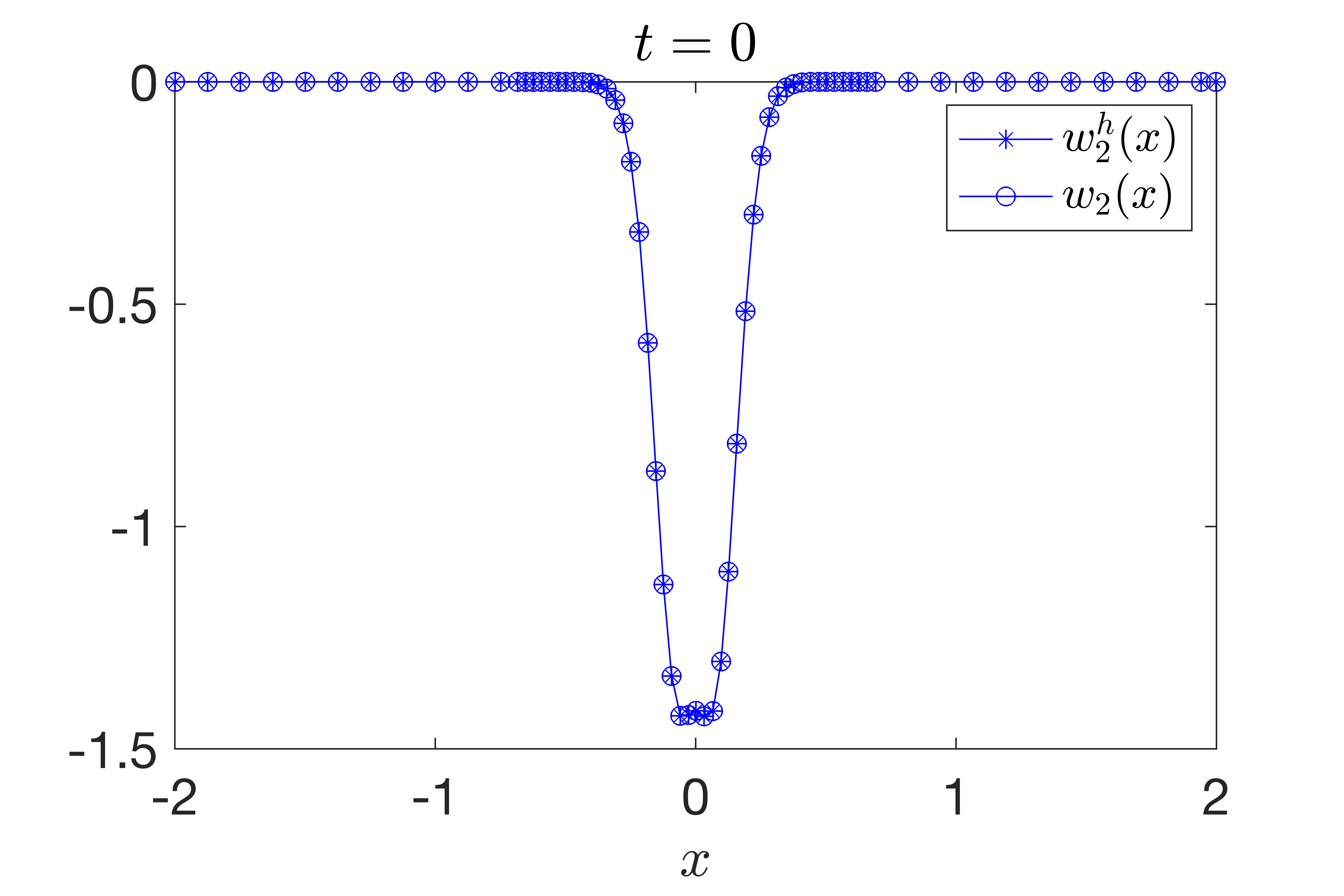}\\
	\includegraphics[width=0.45\linewidth,trim={.0cm .0cm .0cm .0cm},clip]{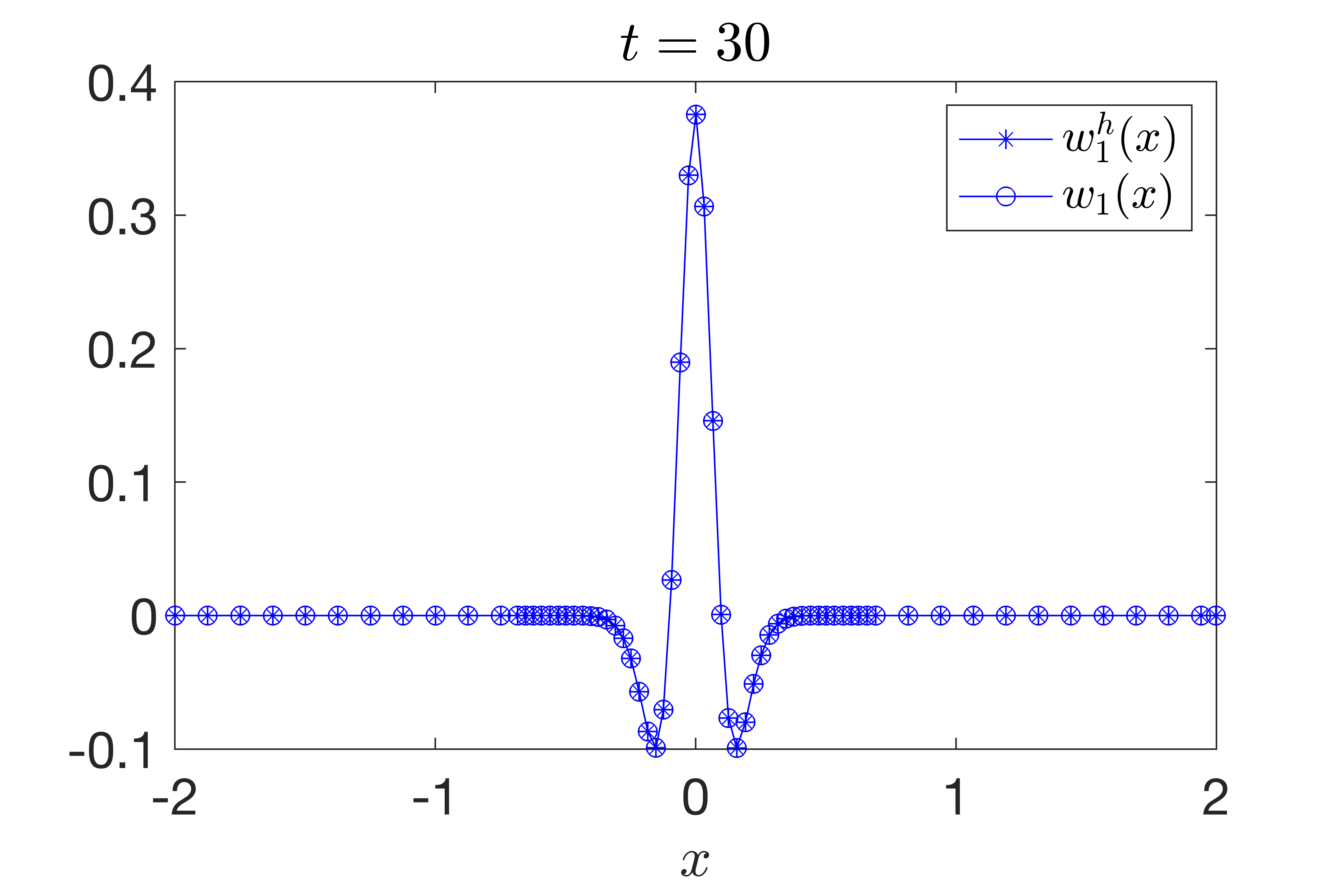}
	\includegraphics[width=0.45\linewidth,trim={.0cm .0cm .0cm .0cm},clip]{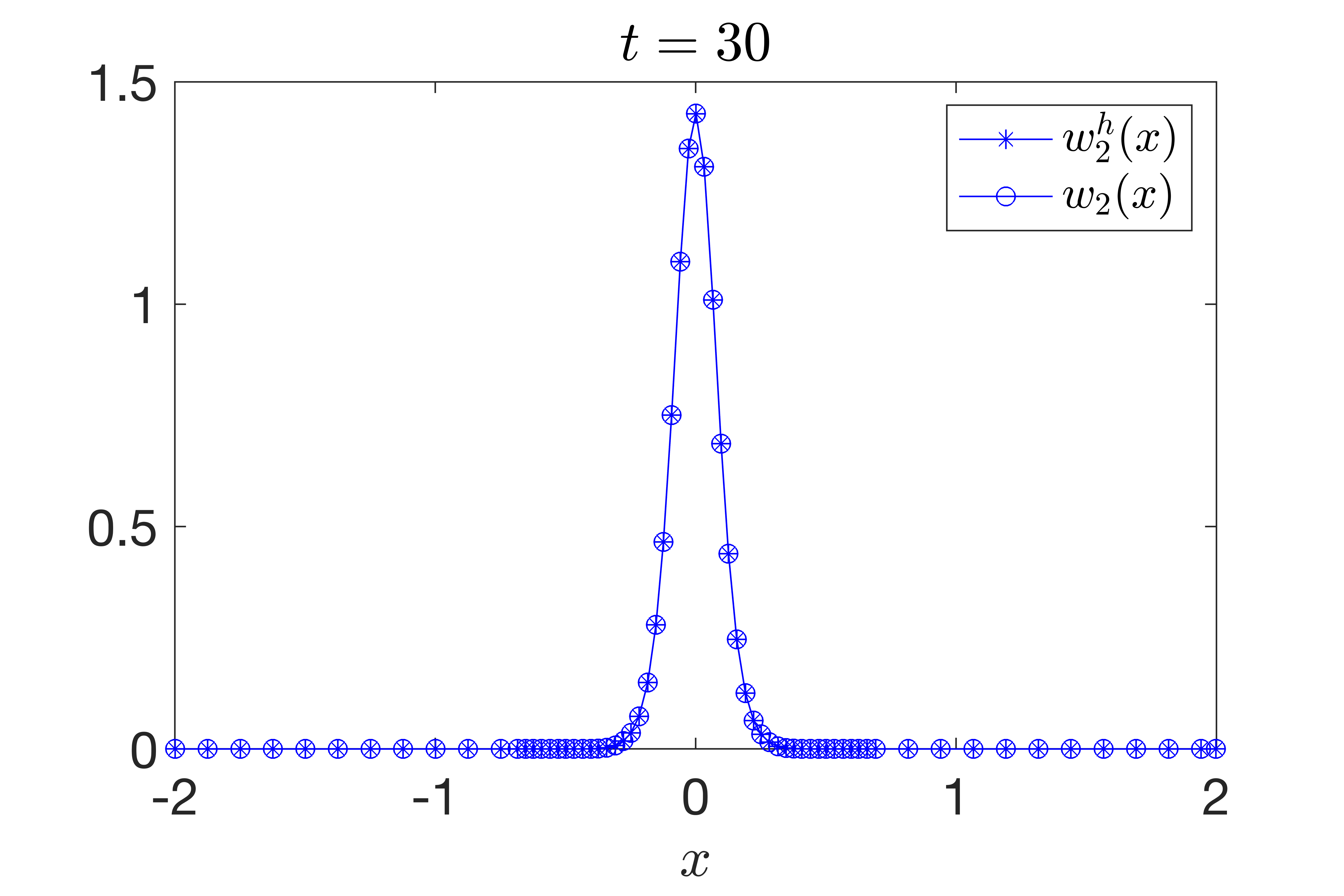}\\
	\includegraphics[width=0.45\linewidth,trim={.0cm .0cm .0cm .0cm},clip]{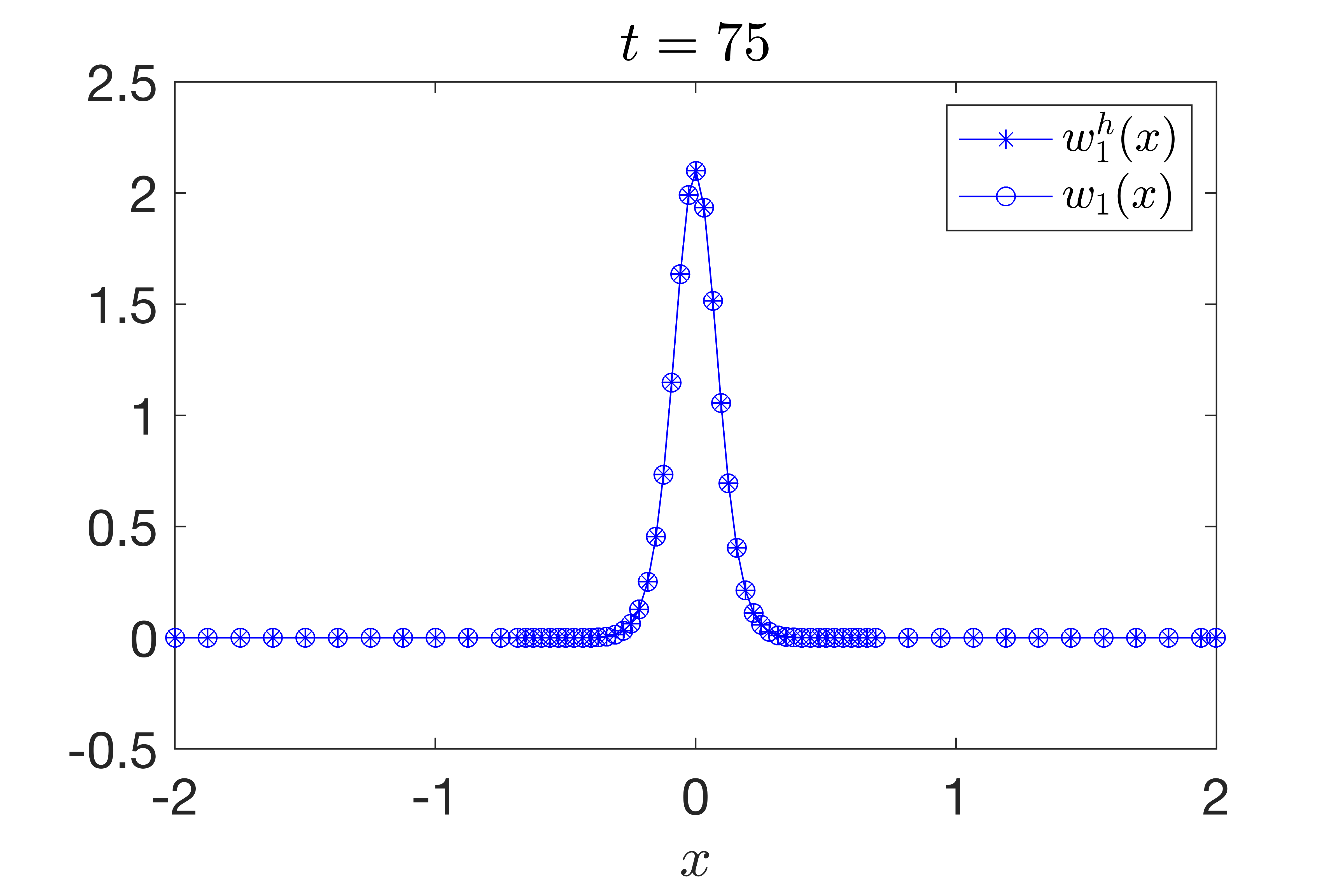}
	\includegraphics[width=0.45\linewidth,trim={.0cm .0cm .0cm .0cm},clip]{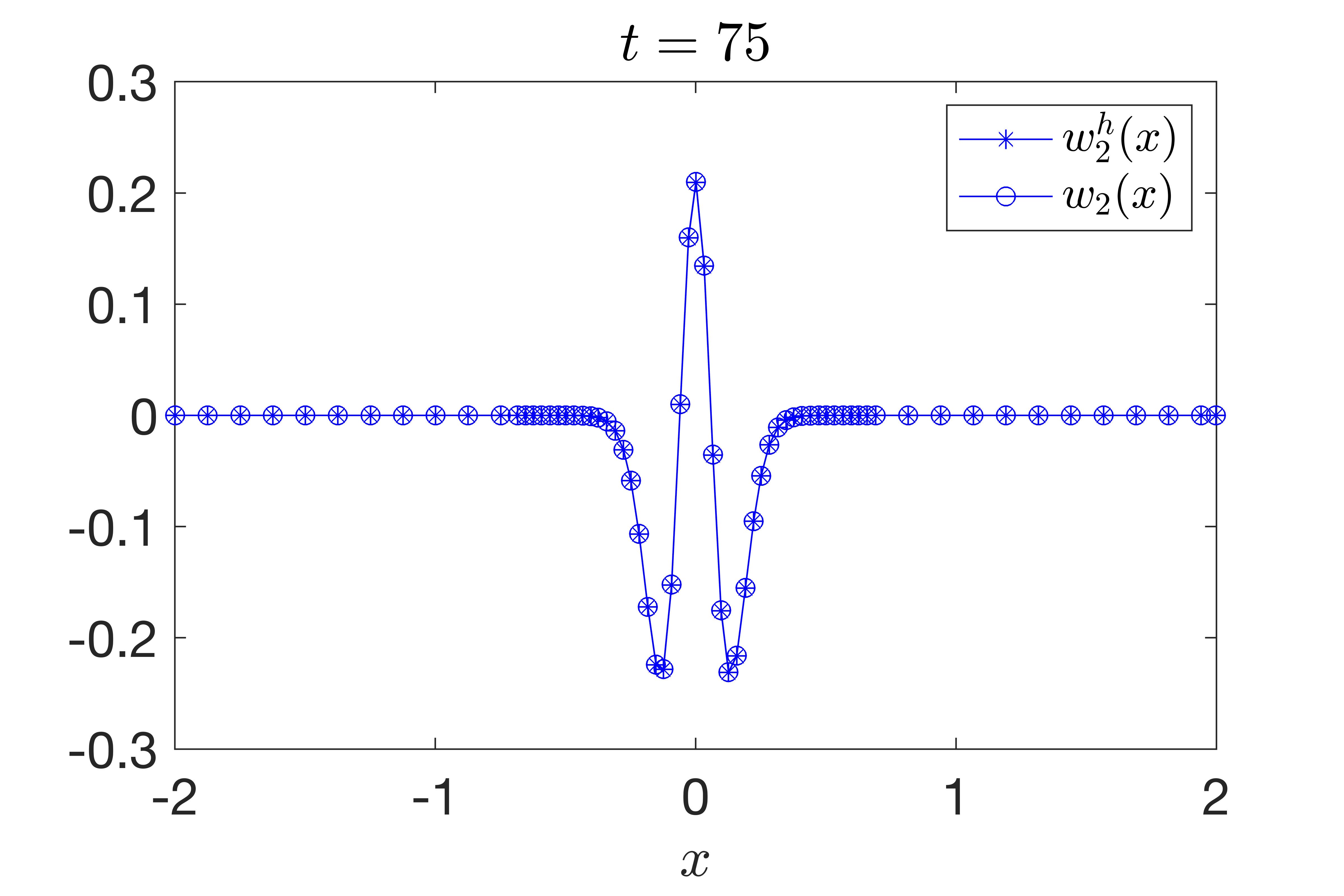}\\
	\includegraphics[width=0.45\linewidth,trim={.0cm .0cm .0cm .0cm},clip]{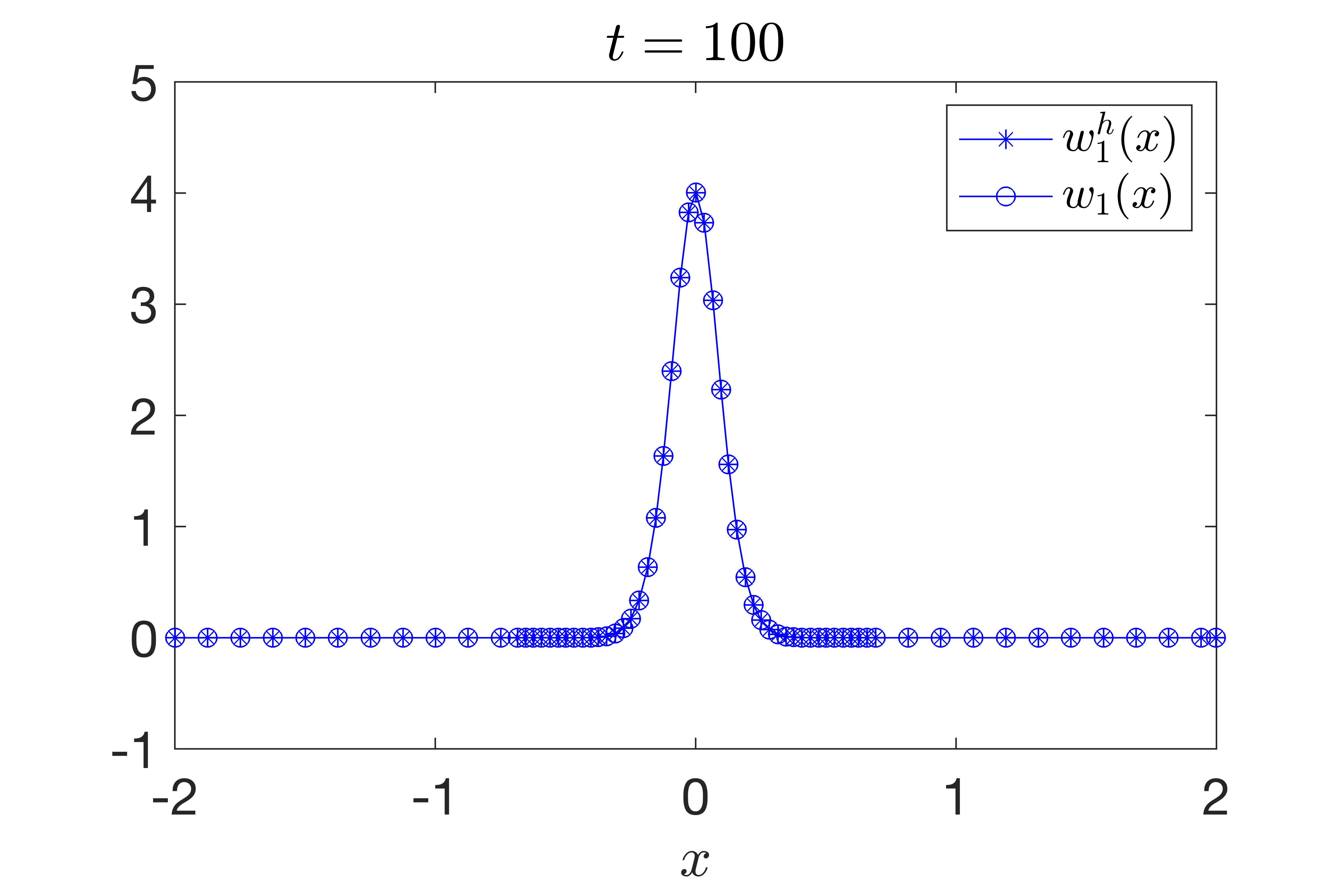}
	\includegraphics[width=0.45\linewidth,trim={.0cm .0cm .0cm .0cm},clip]{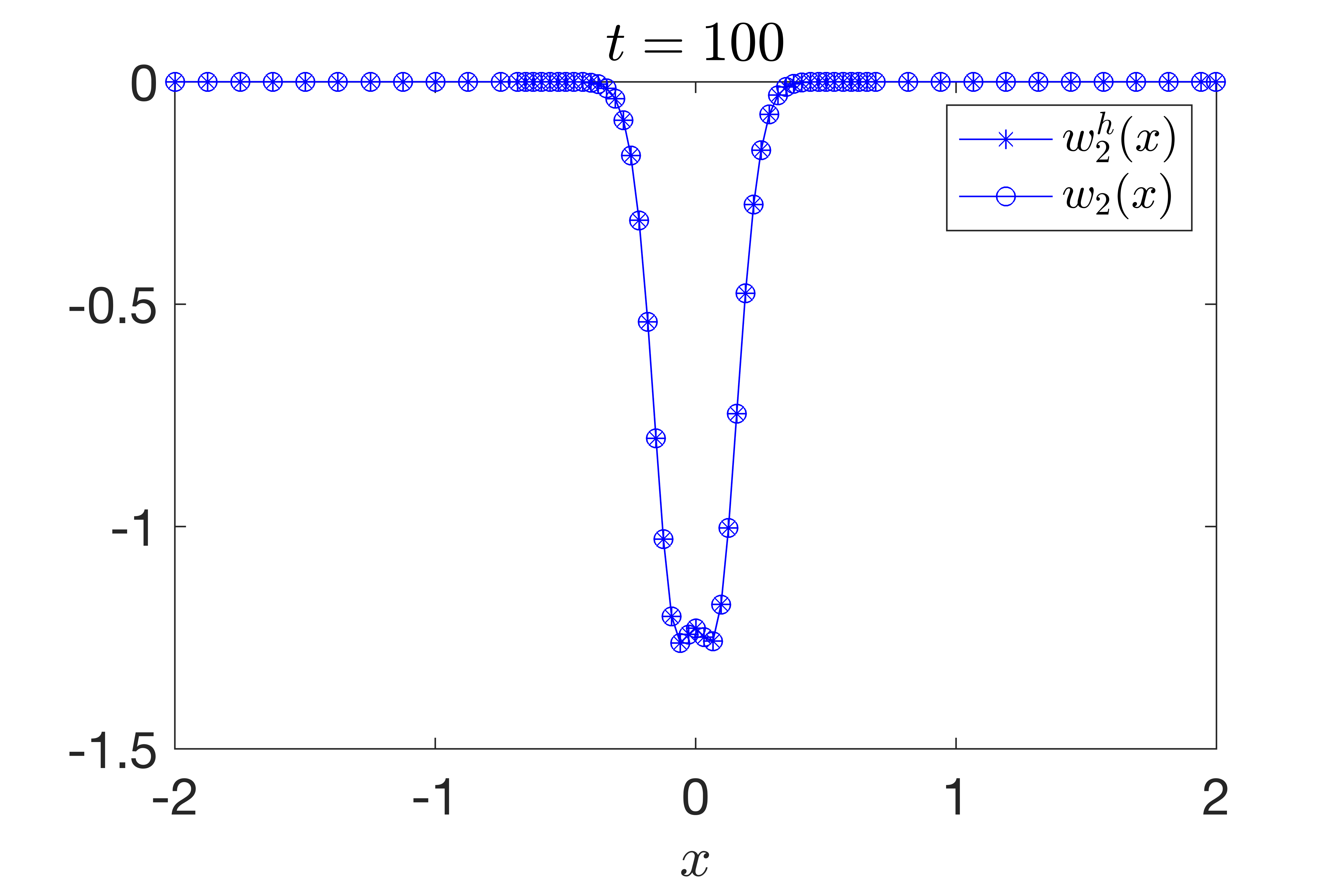}
\caption{Plots of both the numerical solutions $w_1^h, w_2^h$, and the exact solutions $w_1, w_2$ of problem (\ref{example2_sol}) with the degree of approximation space $q=3$. The upwind fluxes (\ref{u_flux}) is used in the simulation.
The numerical and exact solutions at $t=0,30,75,100$ are plotted from top to bottom. On the left, we present the results for $w_1$, while on the right, we display the results for $w_2$.  }\label{fig:solutions_w1_w2_eg2}
\end{figure}

\subsection{Kink solitons}
Here, we consider the traveling wave-like solutions of the system (\ref{system_characteristic}). Let
\[w_1(x, t) = w_1(x - ct), \quad w_2(x, t) = w_2(x - ct),\]
then system (\ref{system_characteristic}) can be rewritten to the following ODE system,
\begin{equation}\label{ODE_system}
    \begin{aligned}
    (c+1)\frac{dw_1}{dz} &= N(w_1, w_2)w_2,\\
    (1-c)\frac{dw_2}{dz} &= N(w_1, w_2)w_1,
    \end{aligned}
\end{equation}
where $z:= x - ct$, and the nonlinear interaction, $N(w_1, w_2)$, is specified the same as the one in example one, see (\ref{nonlinear_term_numer}). When $|c|\neq 1$, we have
\[(w_1, w_2) = (0, 0)\ \ \mbox{and}\ \ (w_1, w_2) = (\cos\theta, \sin\theta), \ \theta \in (0,2\pi]\]
are fixed points of the ODE system (\ref{ODE_system}). As shown in \cite{liweinstein23}, we have
\[\frac{dQ\big(w_1(z), w_2(z)\big)}{dz} = 0,\]
where
\begin{equation}\label{integral_curve}
Q\big(w_1(z), w_2(z)\big) = c\left(\frac{(w_1+w_2)^2}{2} + \frac{(w_1-w_2)^2}{2}\right) + w_1^2-w_2^2.
\end{equation}
This indicates that for a given constant $Q$, (\ref{integral_curve}) can only be unbounded, periodic, a fixed point, or approaches fixed points of (\ref{ODE_system}). We have the so-called kink solitons when $|c| < 1$ and the solutions with nontrivial asymptotic values at either $+\infty$ or $-\infty$. 
{We refer to \cite{liweinstein23} for more extensive discussions.}

In this section, we 
numerically study the dynamics of these kink solitons of the system (\ref{system_characteristic}) by the proposed DG scheme in Section \ref{Sec: DG_formulation} with the upwind fluxes (\ref{u_flux}).

To this end, we consider problem (\ref{system_characteristic}) in a computational domain defined from $x_a = -40$ to $x_b = 200$, the wave speed is chosen to be $c = 0.4 < 1$, and the degree of the approximation space is set as $q = (0,1,2,3)$. In addition, we discretize the computational domain by a uniform mesh with DG element vertices $x_i = x_a + ih, i = 0, 1, \cdots, 600$, namely, $h = 0.4$. The problem is evolved with the classical $4$-stage Runge Kutta time integrator until the final time $T = 100$ with time step size $\Delta_t = 4\times 10^{-5}$.

An accurate initial data profile, given by a kink solution, is generated as follows. We perturb the trivial fixed point, $(w_1, w_2) = (0, 0)$, of the ODE system (\ref{ODE_system}) to make
\[w_1(z = 0) = -7\sqrt{2}\times 10^{-51}, \quad w_2(z = 0) = -3\sqrt{2}\times 10^{-51}\]
be the initial data, then implement MATLAB inline function $ode45$ to solve (\ref{ODE_system}). In particular, when calling $ode45$, we set the termination time to be coincidental with spatial Gauss nodes, then followed by setting $t = 0$ in $z = x - ct$ of the resulting ODE system solution, we obtain the initial profile for the original PDE system (\ref{system_characteristic}), see Figure \ref{fig:kink_initial}. In addition, we impose Dirichlet boundary conditions when numerically solving the problem (\ref{system_characteristic}). 

\begin{figure}[htb!]
	\centering
	\includegraphics[width=0.5\linewidth,trim={.0cm .0cm .0cm .0cm},clip]{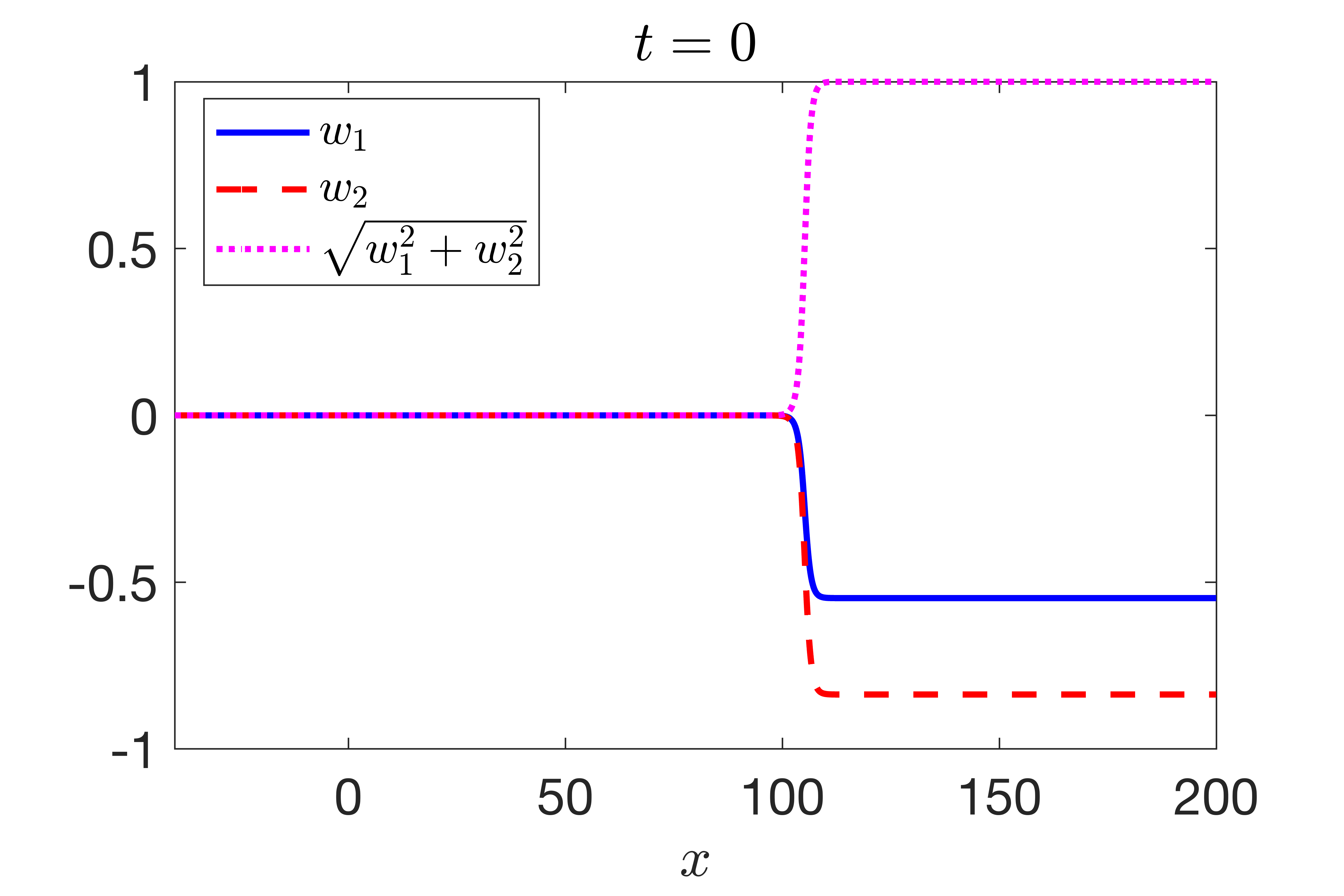}
\caption{Initial profiles, $(w_1(x,0), w_2(x, 0))$, of the kink solutions of the system (\ref{system_characteristic}).}\label{fig:kink_initial}
\end{figure}

Figure \ref{fig:solutions_w1_eg3} (resp. Figure \ref{fig:solutions_w2_eg3}) presents the space-time plots of $w_1^h$ (resp. $w_2^h$), while Figure \ref{fig:solutions_w12_eg3} displays the comparison between numerical solutions with different approximation degree $q$ at final time $t = 100$: on the left is the plot for $w_1^h$ and on the right is for $w_2^h$. In addition, Figure \ref{fig:solutions_w1_eg3} and Figure \ref{fig:solutions_w2_eg3} have the same layout: from left to right are the kink solitons with the degree of approximation space $q = 0, q = 1, q = 2$ and $q = 3$, respectively. From these figures, we see that when the degree of the approximation space is very small, say $q = 0$, non-physical oscillations are generated near the fast transition region of the kink solitons and then move to the left with the oscillations being weaker and weaker, but make the kink solitons a small phase shift to the right. When the degree of the approximation is relatively large, for this example $q \geq 1$, we see that the kink soliton, $w_1^h$, decreases monotonically from $0$ to $\sim -0.5477$ in Figure \ref{fig:solutions_w1_eg3}, while decreases monotonically from $0$ to $\sim -0.8366$ for the kink soliton, $w_2^h$, in Figure \ref{fig:solutions_w2_eg3}. And they move from the left to the right with their original shape.


\begin{figure}[htb!]
	\centering
	\includegraphics[width=0.45\linewidth,trim={.0cm .0cm .0cm .0cm},clip]{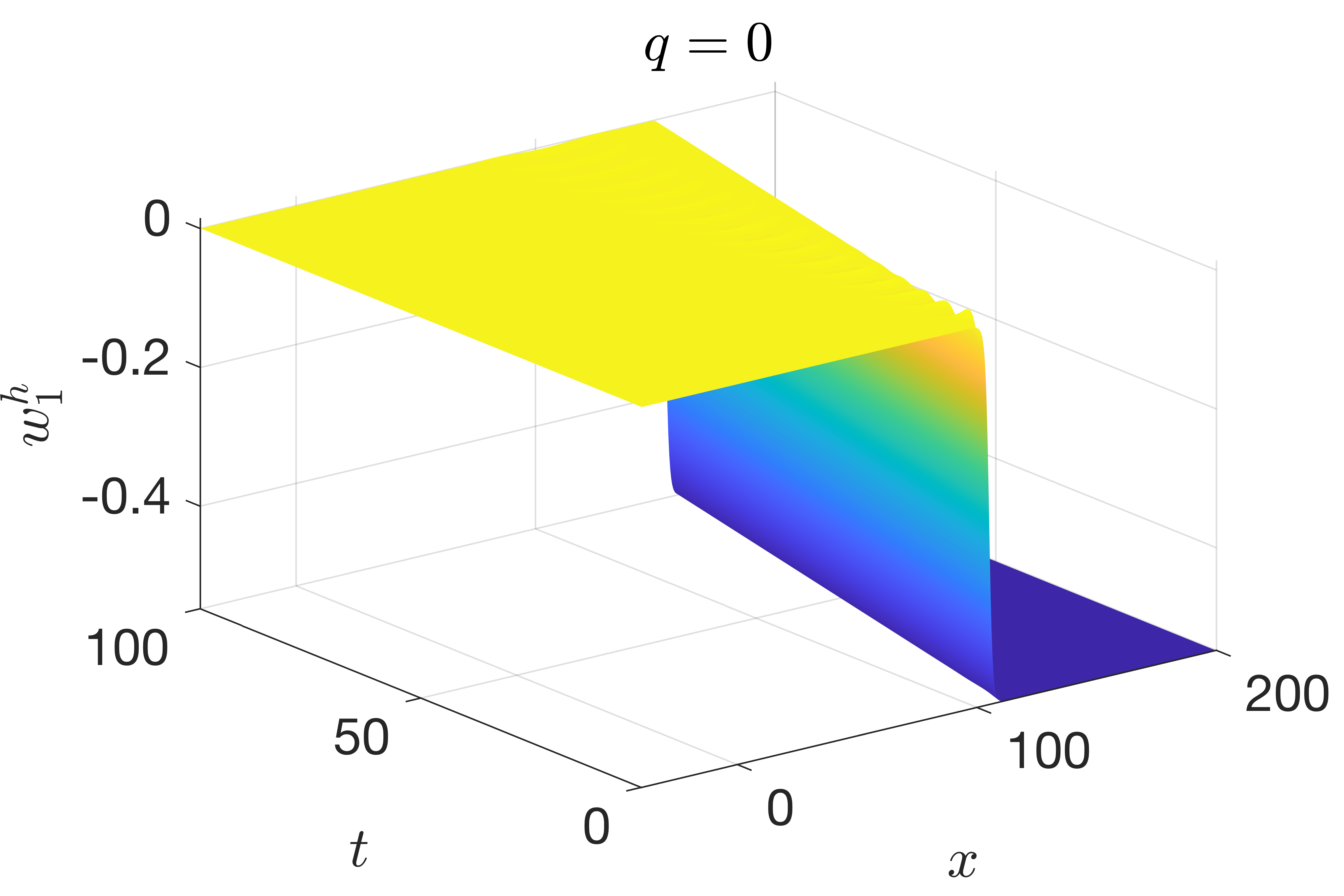}
	\includegraphics[width=0.45\linewidth,trim={.0cm .0cm .0cm .0cm},clip]{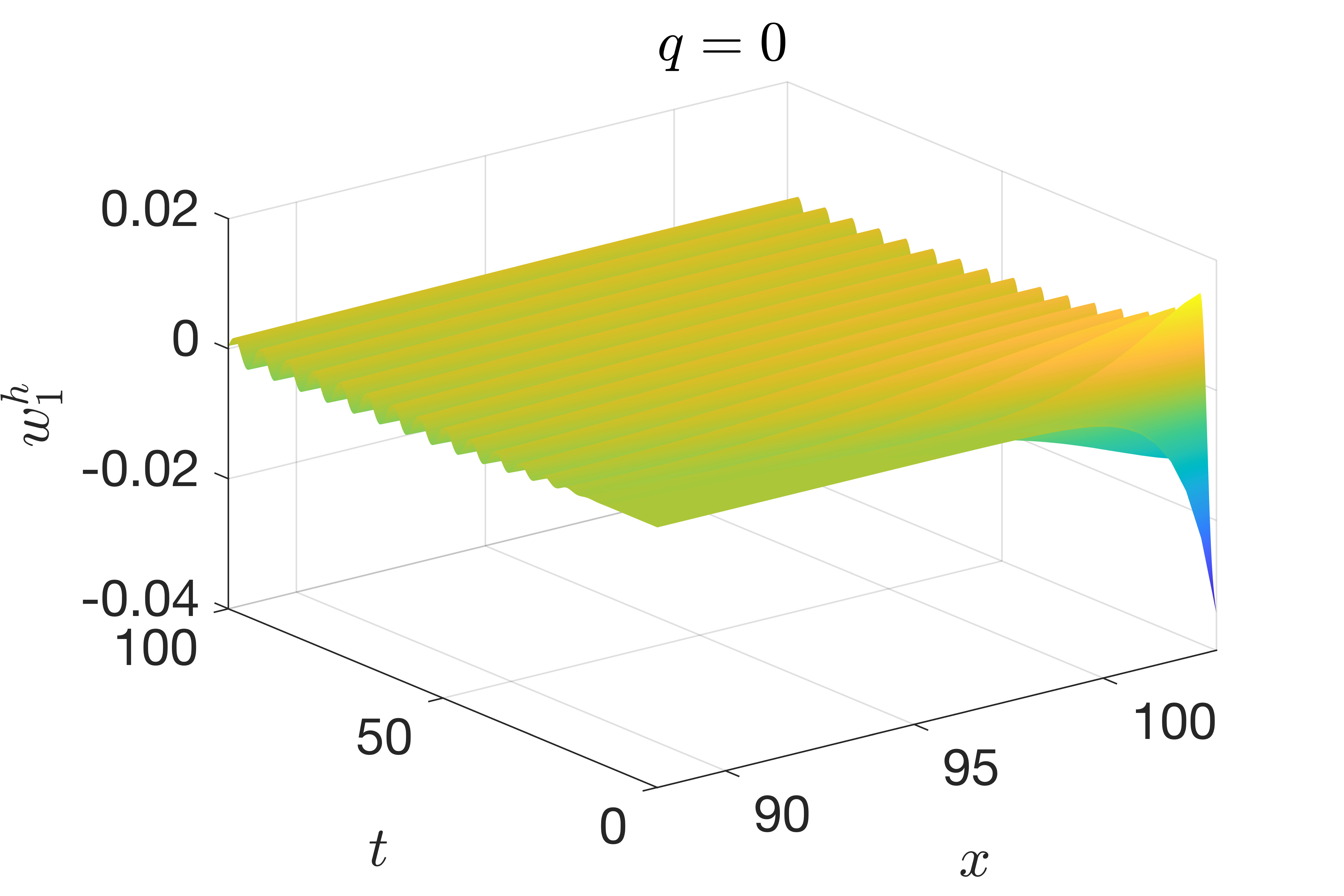}\\
	\includegraphics[width=0.45\linewidth,trim={.0cm .0cm .0cm .0cm},clip]{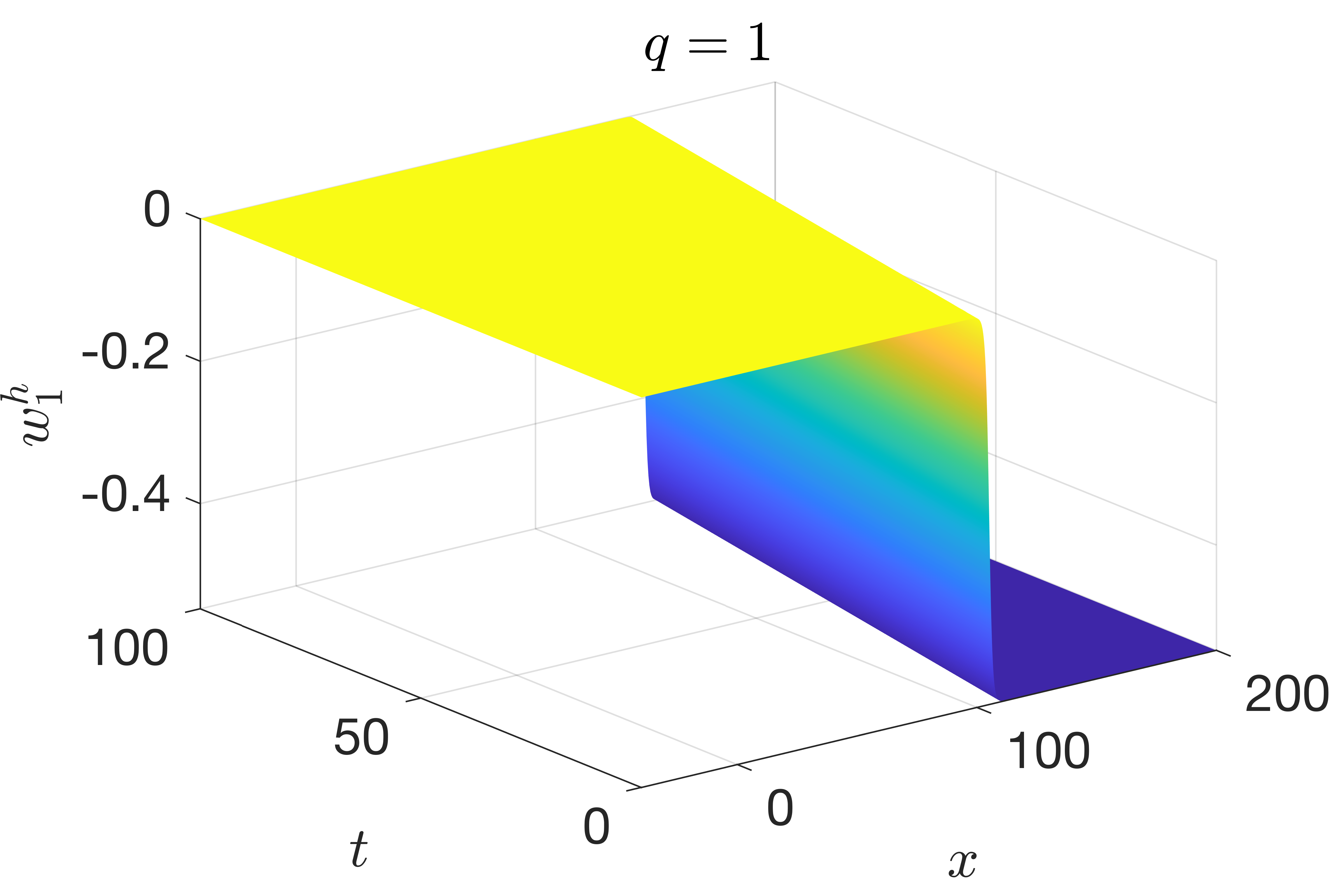}
	\includegraphics[width=0.45\linewidth,trim={.0cm .0cm .0cm .0cm},clip]{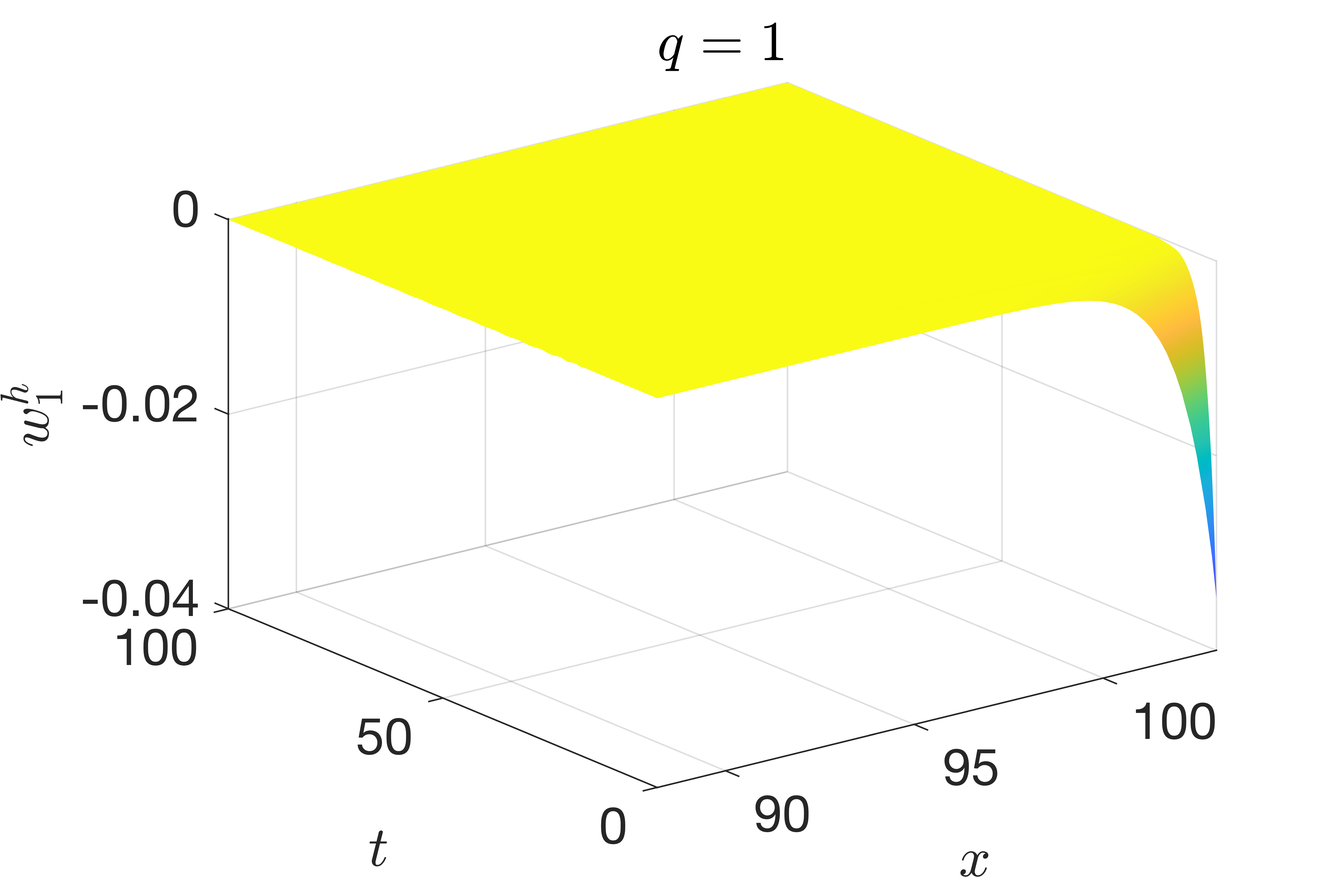}\\
 \includegraphics[width=0.45\linewidth,trim={.0cm .0cm .0cm .0cm},clip]{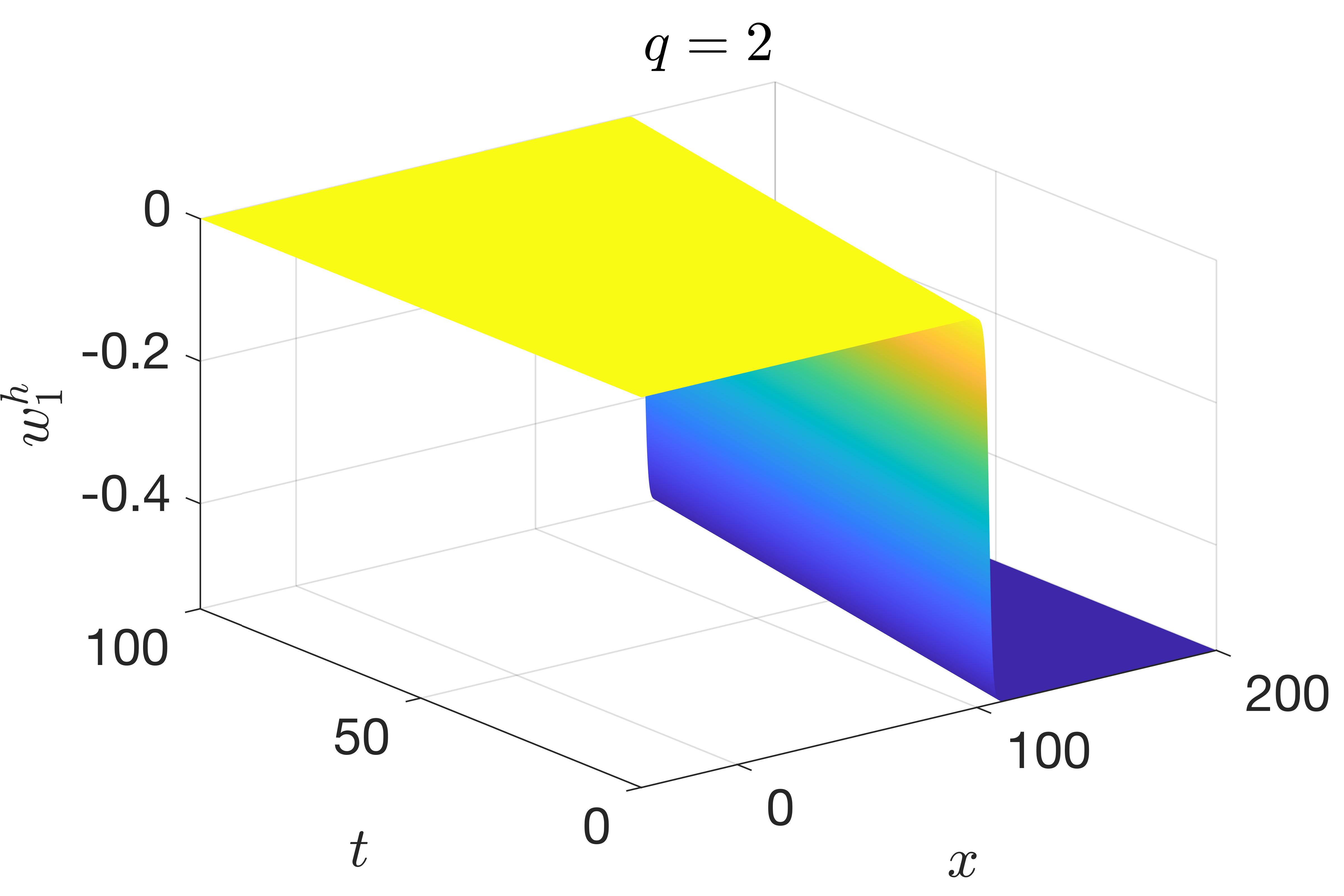}
	\includegraphics[width=0.45\linewidth,trim={.0cm .0cm .0cm .0cm},clip]{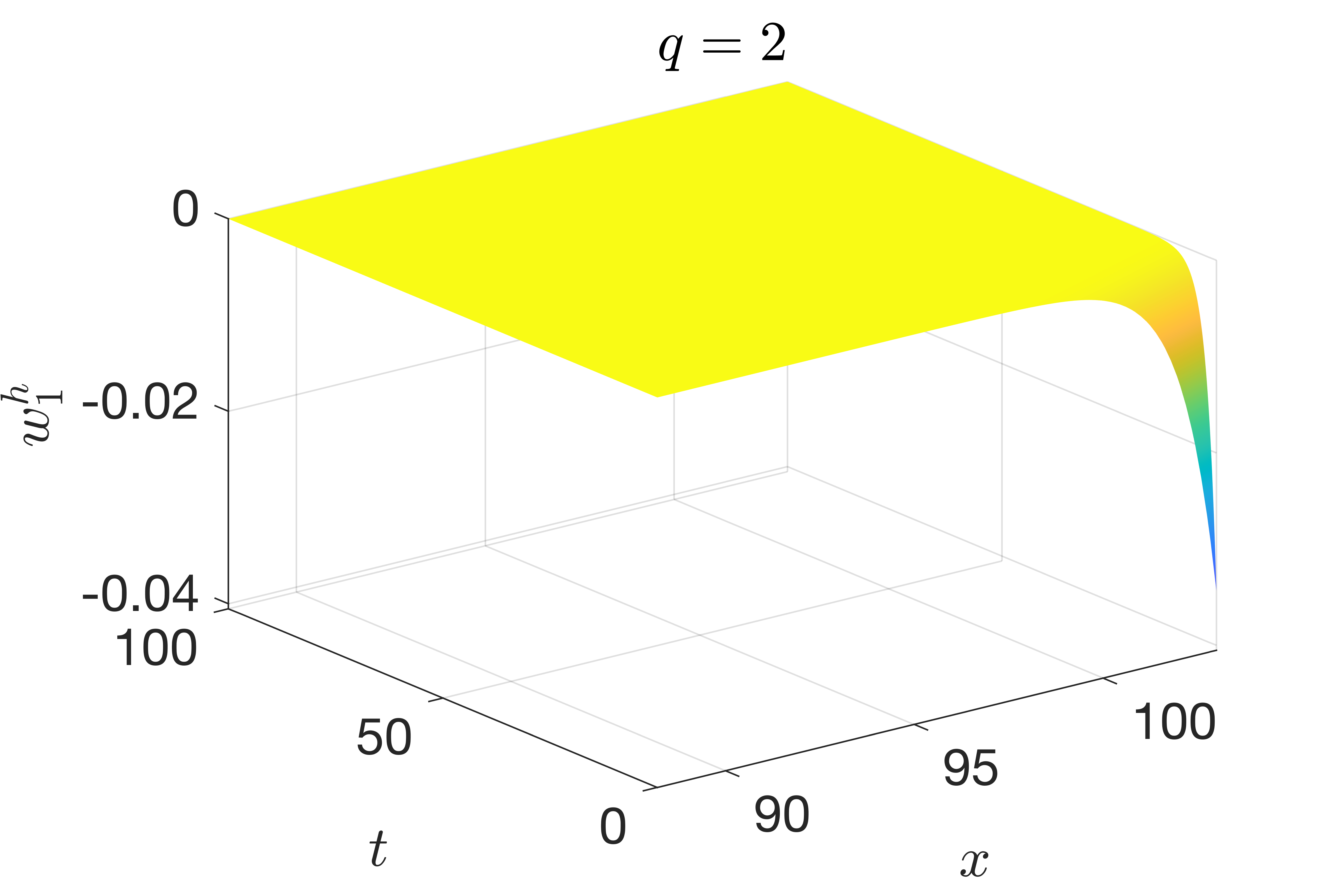}\\
 \includegraphics[width=0.45\linewidth,trim={.0cm .0cm .0cm .0cm},clip]{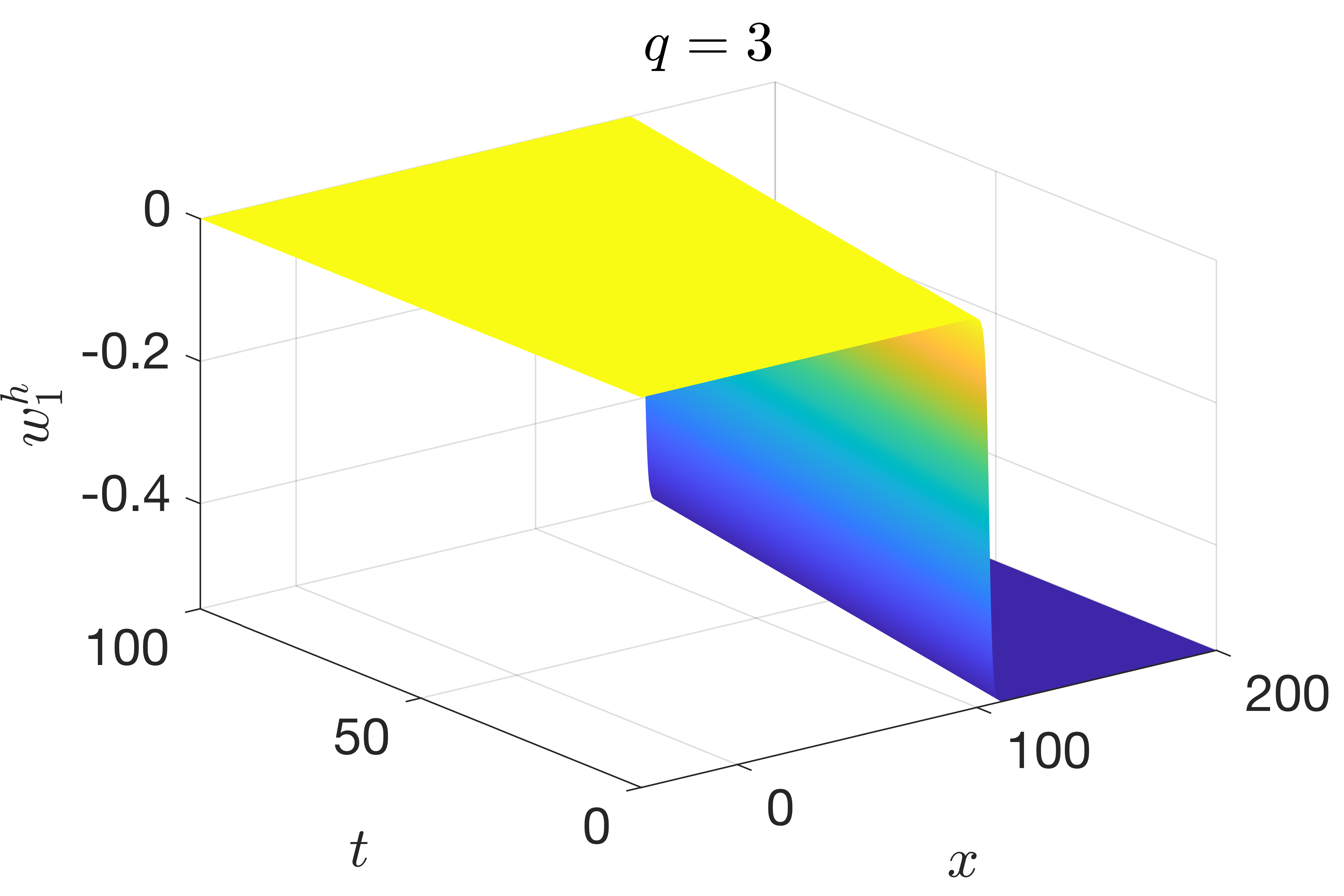}
	\includegraphics[width=0.45\linewidth,trim={.0cm .0cm .0cm .0cm},clip]{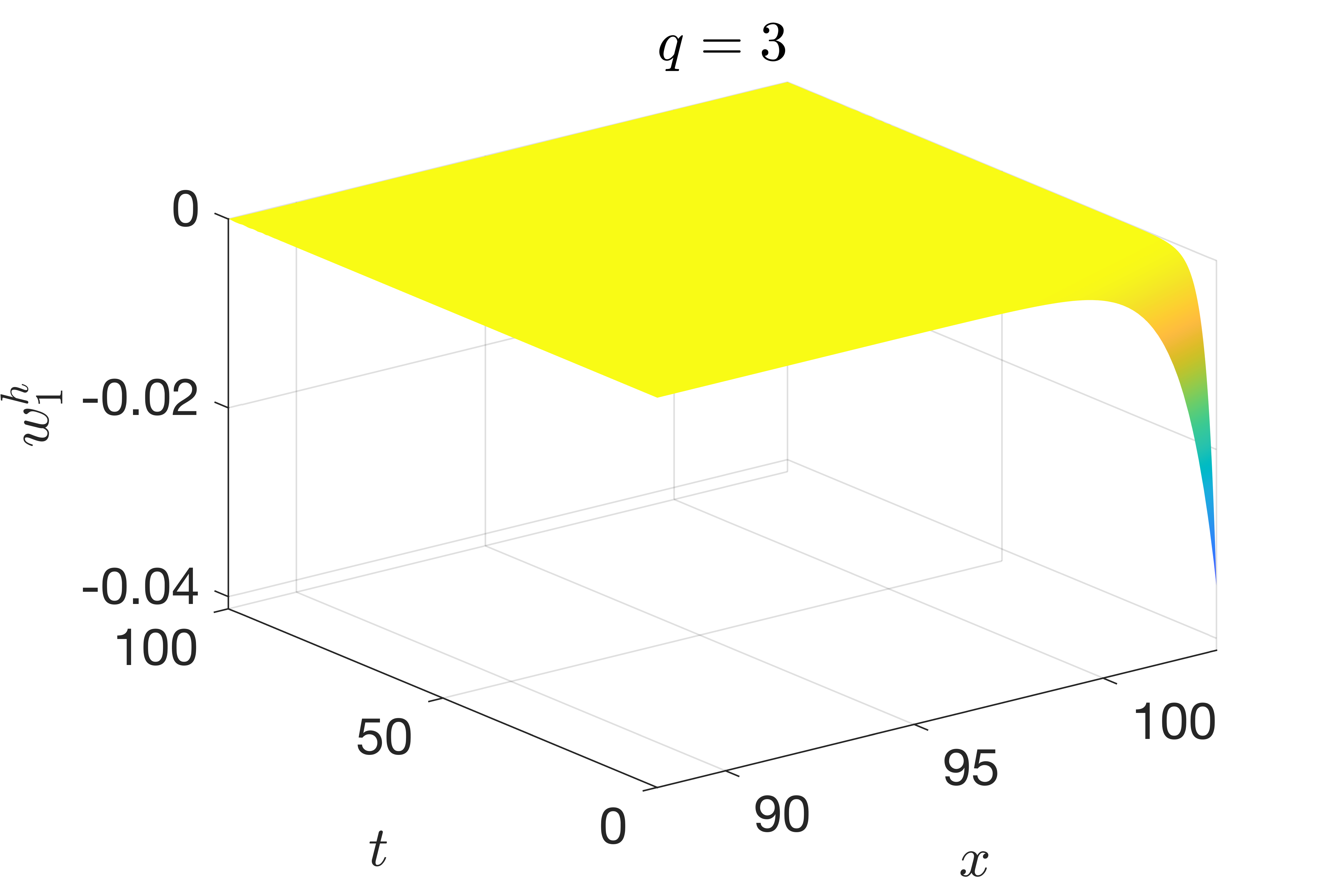}
\caption{Space-time plots of $w_1^h(x,t)$. On the left panel, from the top to the bottom are the plots for the kink soliton with $q = 0$, $q = 1$, $q = 2$, and $q = 3$ respectively; On the right panel are the zoom in plots of the transition region, $x\in (88, 103)$, of the kink soliton.}\label{fig:solutions_w1_eg3}
\end{figure}



\begin{figure}[htb!]
	\centering
	\includegraphics[width=0.45\linewidth,trim={.0cm .0cm .0cm .0cm},clip]{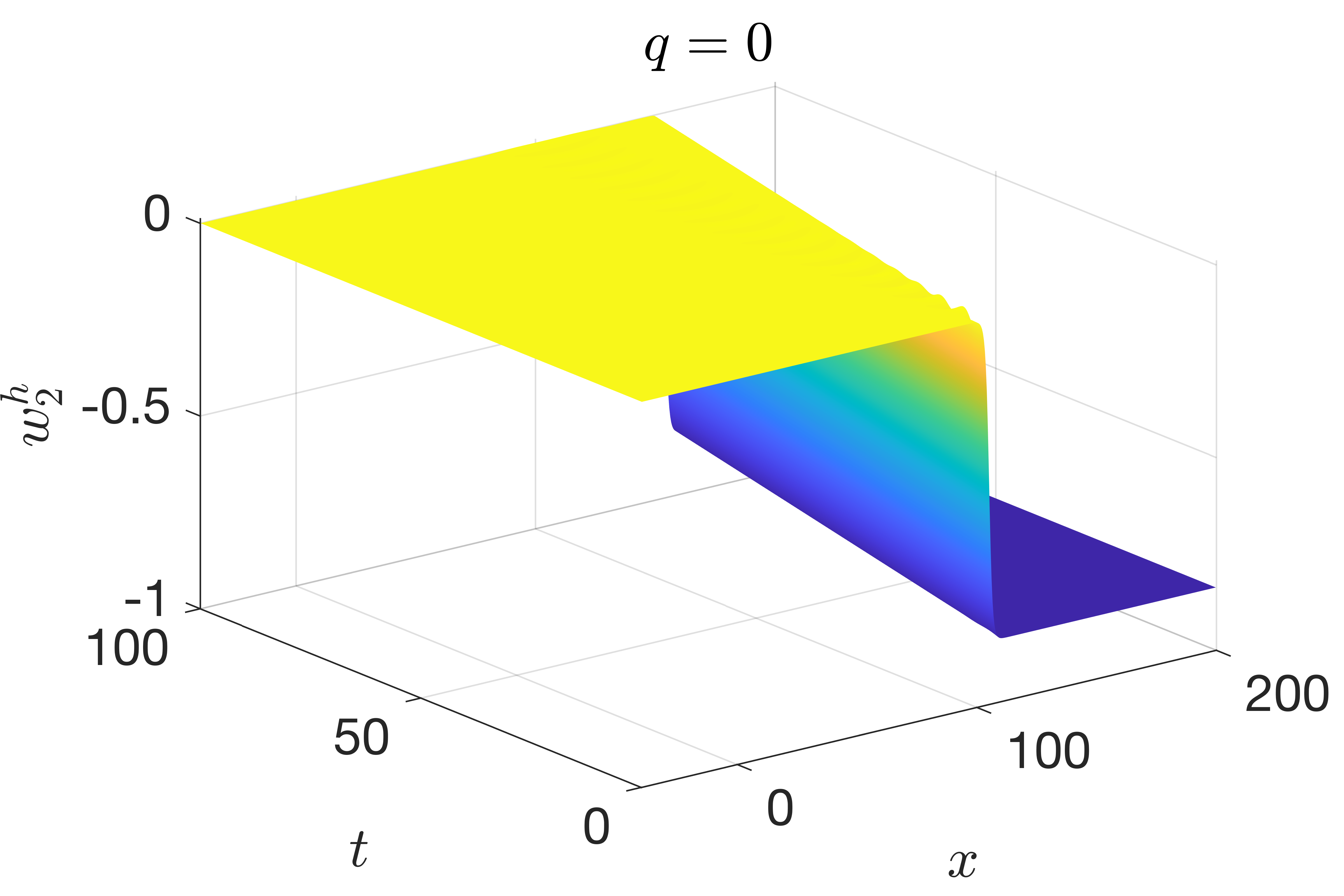}
	\includegraphics[width=0.45\linewidth,trim={.0cm .0cm .0cm .0cm},clip]{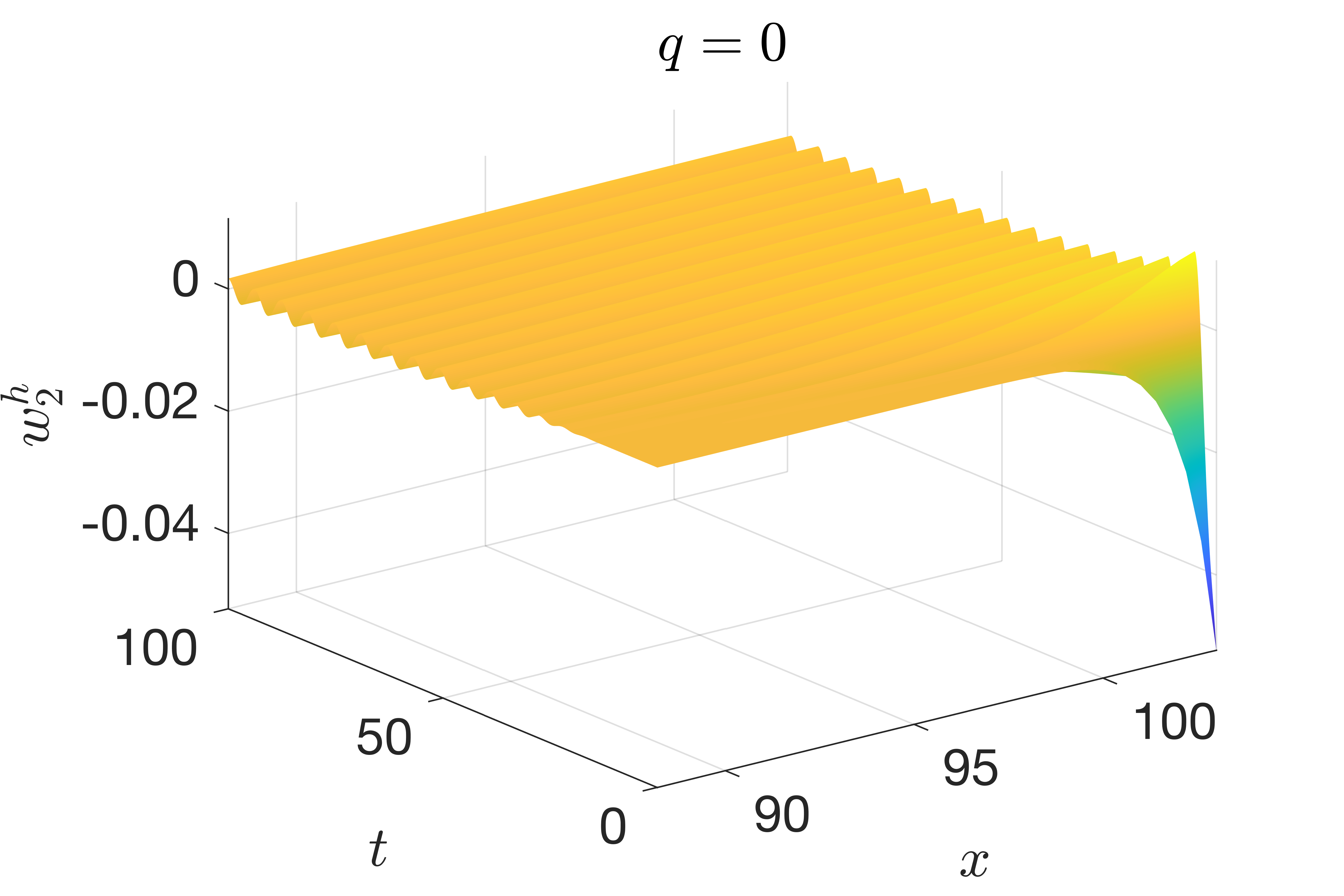}\\
	\includegraphics[width=0.45\linewidth,trim={.0cm .0cm .0cm .0cm},clip]{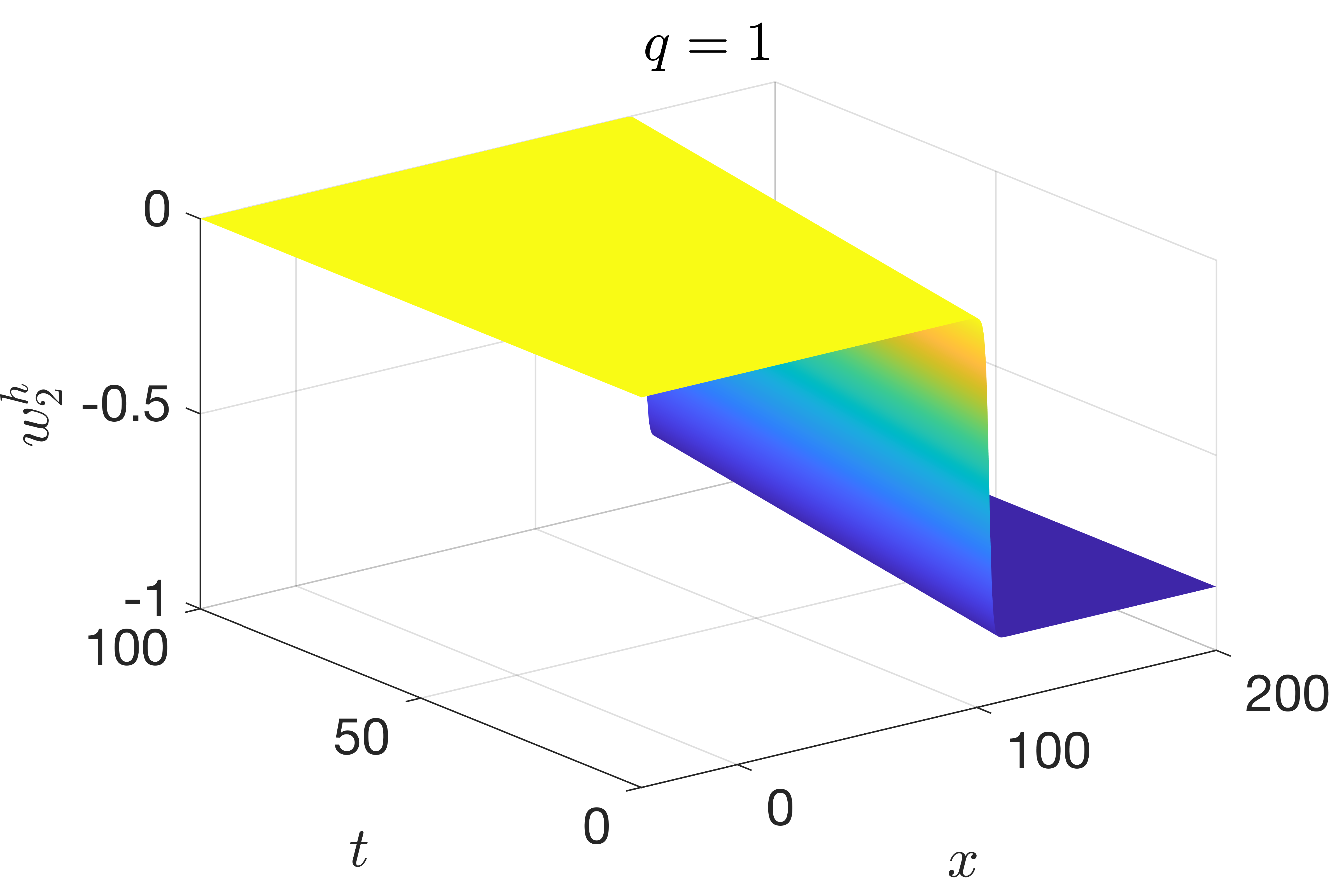}
	\includegraphics[width=0.45\linewidth,trim={.0cm .0cm .0cm .0cm},clip]{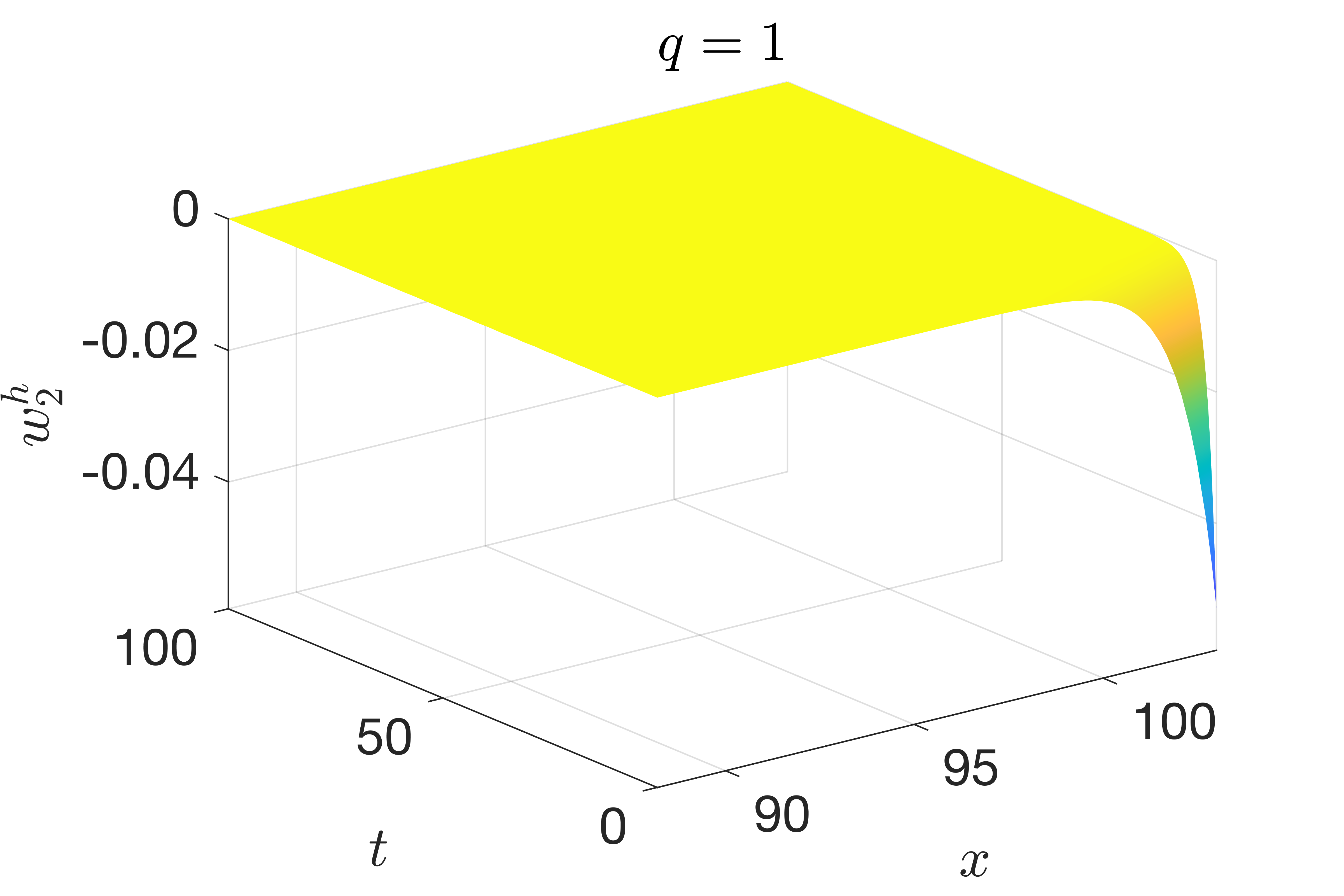}\\
 \includegraphics[width=0.45\linewidth,trim={.0cm .0cm .0cm .0cm},clip]{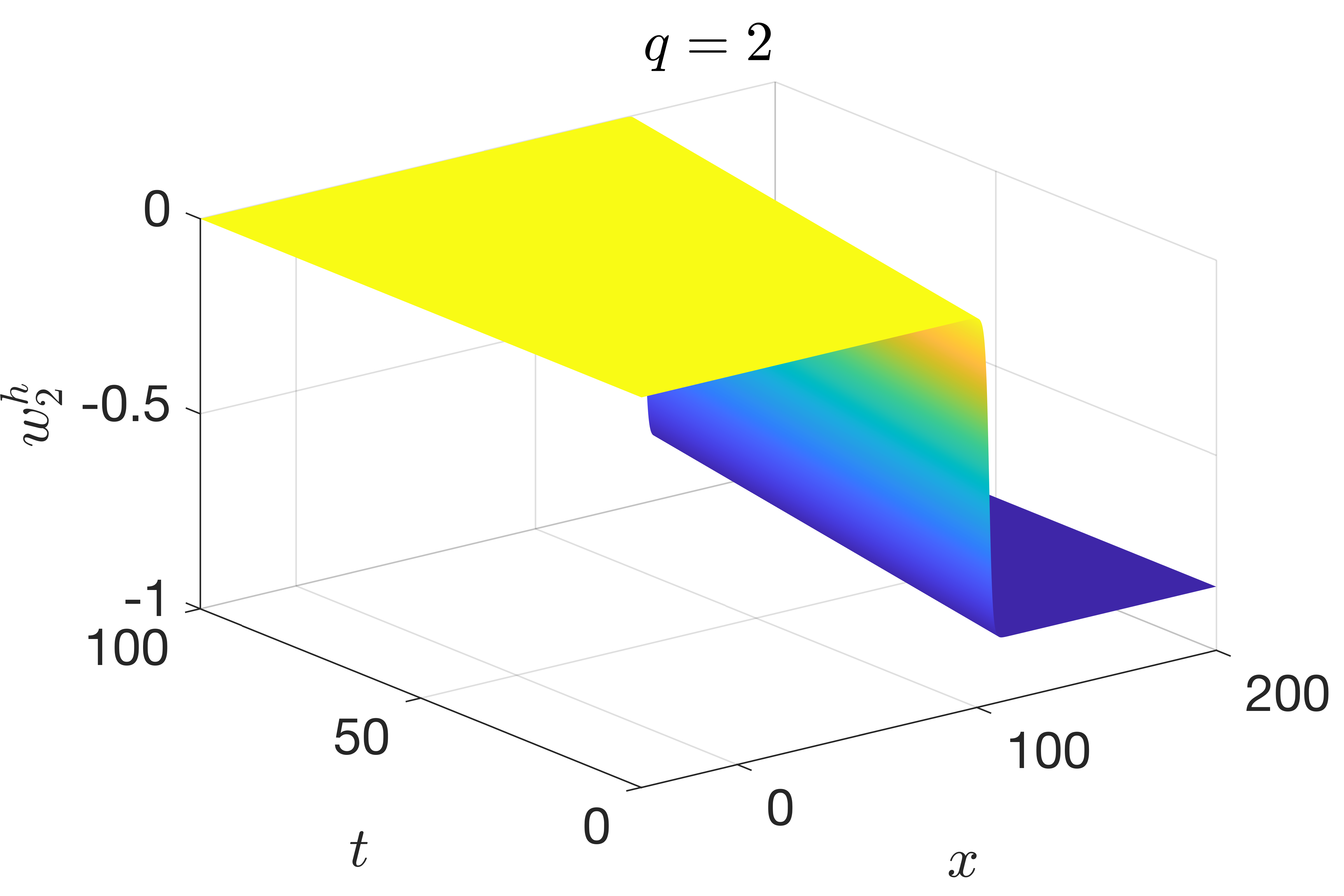}
	\includegraphics[width=0.45\linewidth,trim={.0cm .0cm .0cm .0cm},clip]{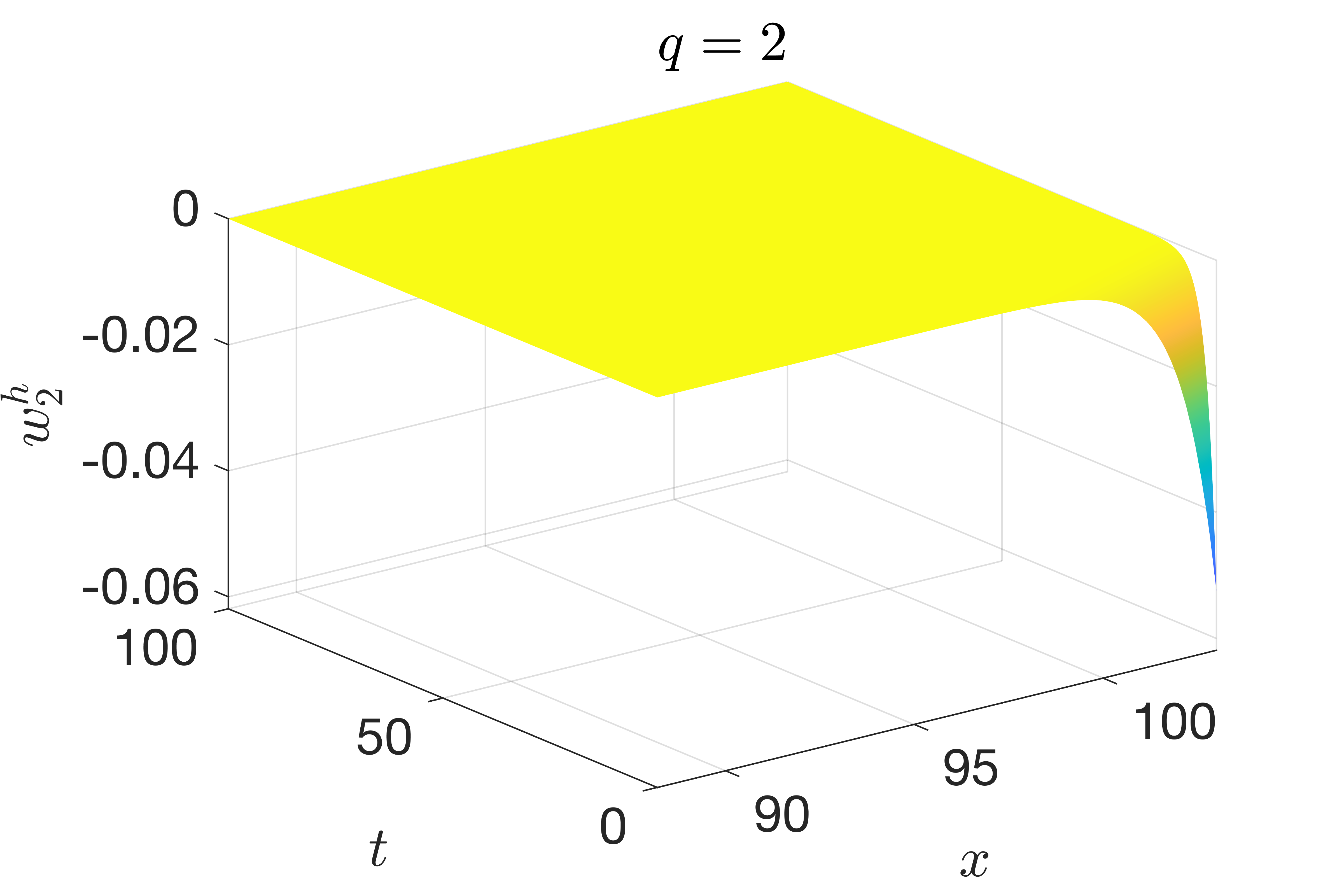}\\
 \includegraphics[width=0.45\linewidth,trim={.0cm .0cm .0cm .0cm},clip]{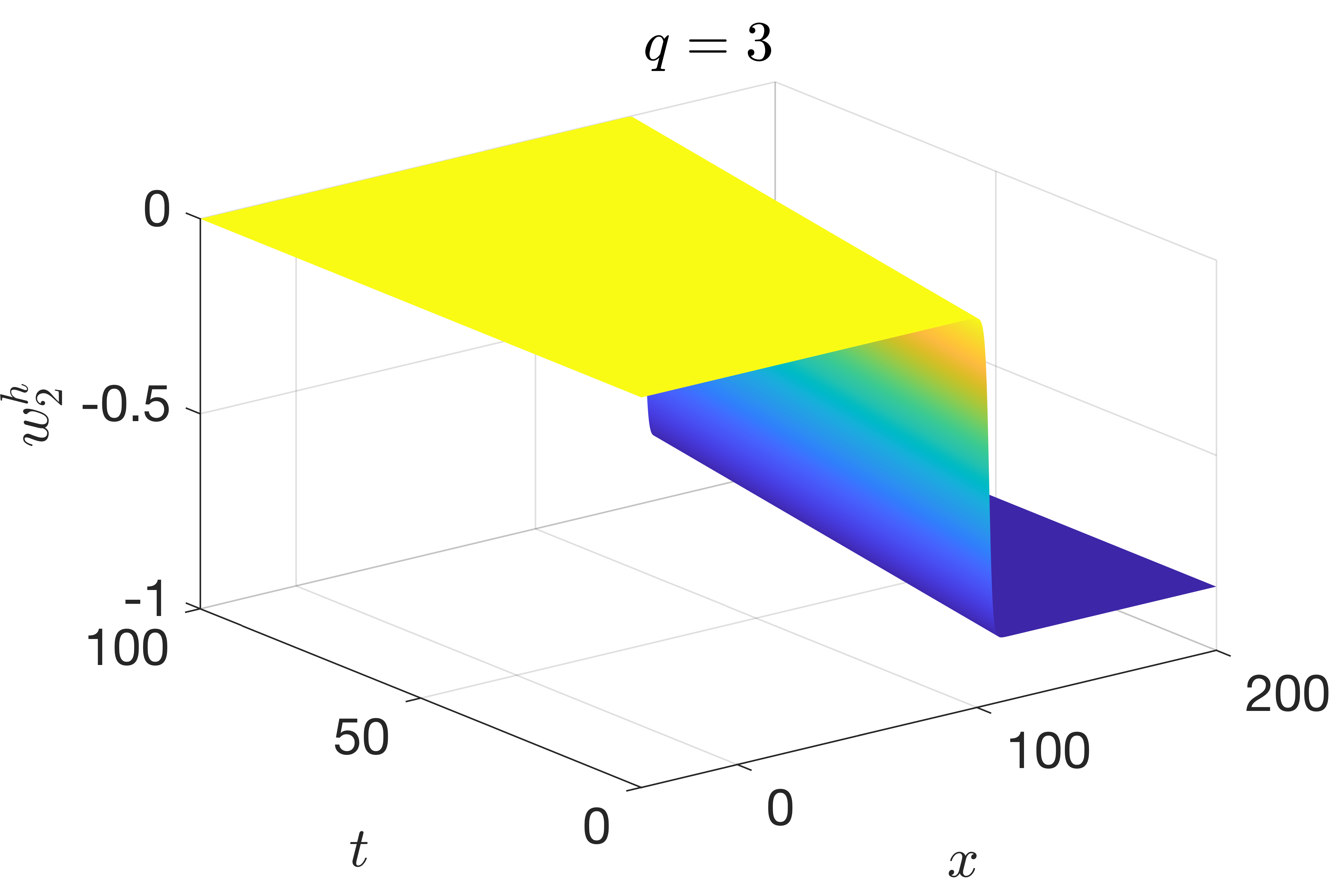}
	\includegraphics[width=0.45\linewidth,trim={.0cm .0cm .0cm .0cm},clip]{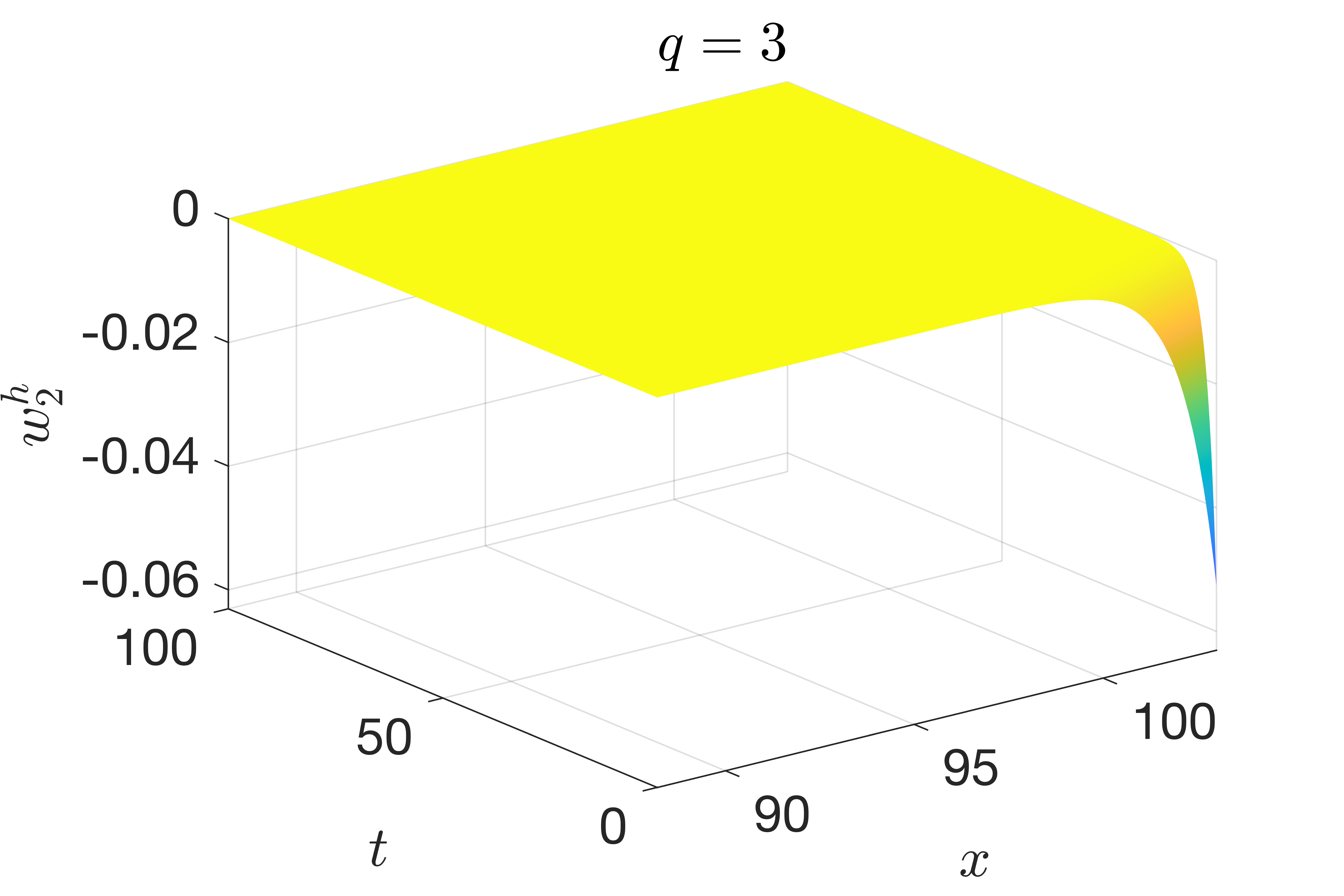}
\caption{Space-time plots of $w_2^h(x,t)$. On the left panel, from top to bottom, are the plots for the kink soliton with $q = 0$, $q = 1$, $q = 2$, and $q = 3$, respectively; On the right panel are the zoom-in plots of the transition region, $x\in (88, 103)$, of the kink soliton.}\label{fig:solutions_w2_eg3}
\end{figure}

\begin{figure}[htb!]
	\centering
	\includegraphics[width=0.45\linewidth,trim={.0cm .0cm .0cm .0cm},clip]{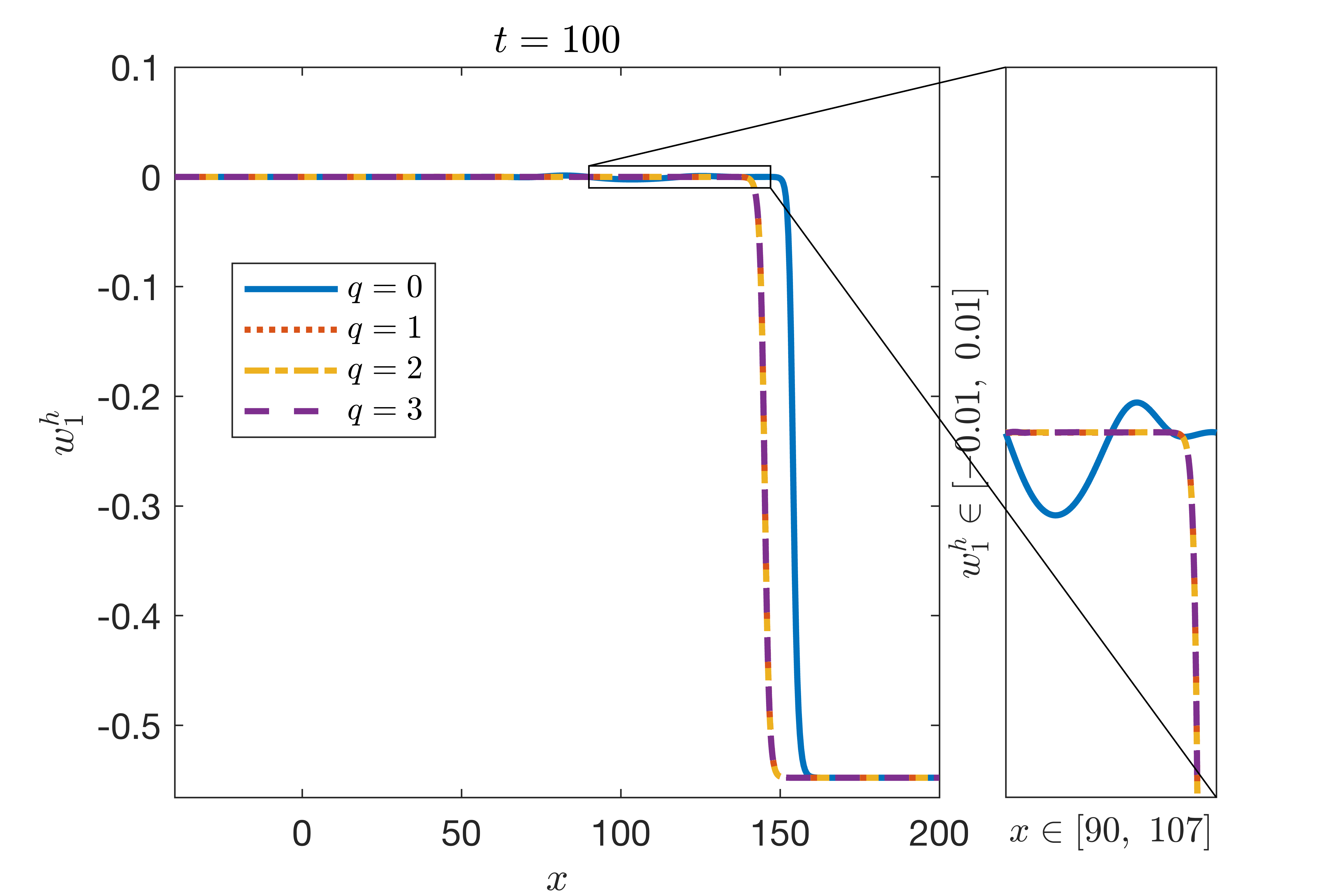}
	\includegraphics[width=0.45\linewidth,trim={.0cm .0cm .0cm .0cm},clip]{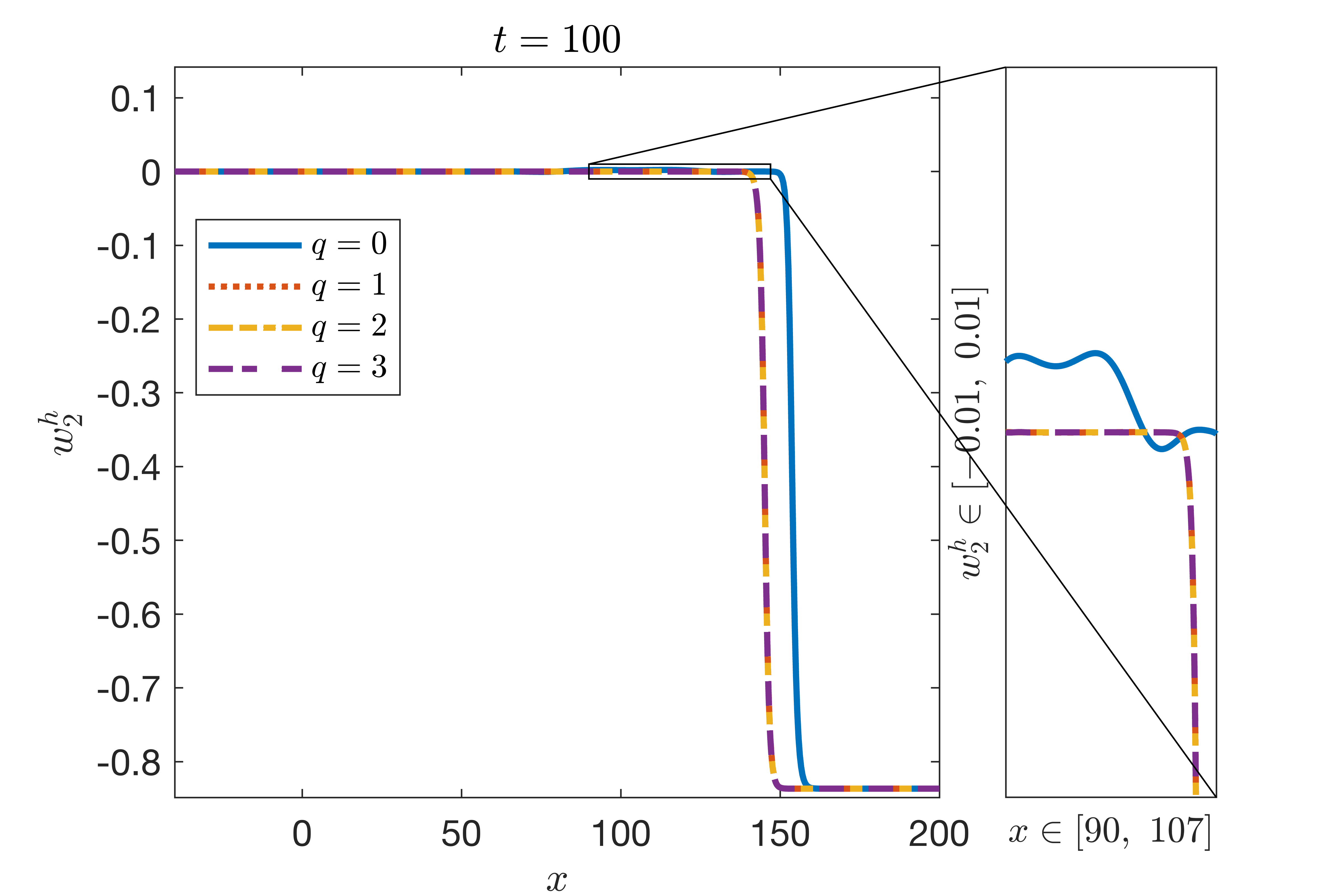}
\caption{Comparison between the kink solitons with different degrees of approximation space $q$ at the final time $t = 100$. On the left presents the results for $w_1^h$, and the right shows the results for $w_2^h$.}\label{fig:solutions_w12_eg3}
\end{figure}

Last, the numerical discrete energy $E^h(t)$, defined in (\ref{eq:discrete_energy}), trajectories of the proposed DG scheme for the problem (\ref{system_characteristic}) until the final time $t = 100$ with the initial profile shown in Figure \ref{fig:kink_initial} are presented in Figure \ref{fig:solutions_energy_eg31} with different degrees, $q$, of the approximation space. Particularly, to measure the performance of the schemes with different approximation degrees $q$, we compute the numerical discrete energy $E^h(t)$ in a moving box that has the same speed $c = 0.4$ with the kink solitons. Initially, the box is located from $x = 60$ to $x = 140$. Overall, we note that the discrete energy is dissipating as predicted since we use an energy dissipating scheme, the upwind fluxes (\ref{u_flux}), for the simulations here.  In addition, we observe that when $q = 0$, the energy dissipates a lot since there are non-physical oscillations generated as shown in Figure \ref{fig:solutions_w1_eg3} and Figure \ref{fig:solutions_w2_eg3}. When $q = 1,2,3$, though the energy is dissipating as predicted, the total dissipation is small, and the changes happen only around the third digit for $q = 1$, around $5$ digits for $q = 2$, and around $8$ digits for $q = 3$ until $T = 100$. One quick takeaway from this is that high-order schemes have better performance, and lower dispersion and dissipation errors, than their lower-order counterparts.

\begin{figure}[htb!]
	\centering
	\includegraphics[width=\linewidth,trim={.0cm .0cm .0cm .0cm},clip]{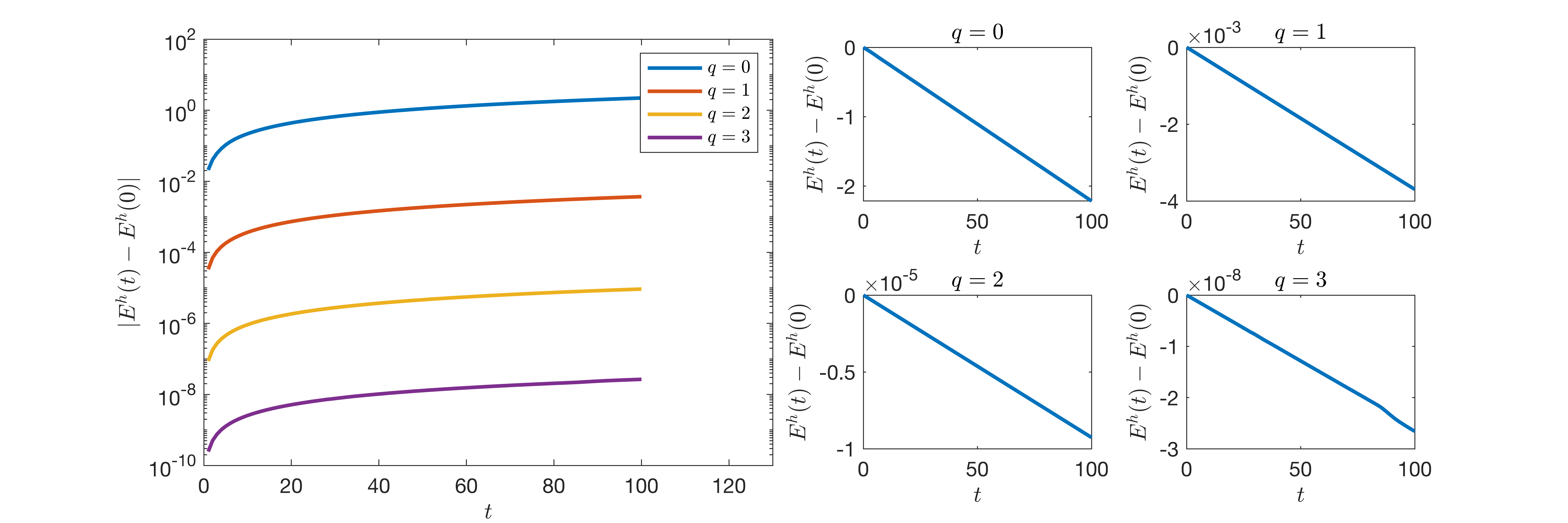}
\caption{we present the energy difference, $|E^h(t) - E^h(0)|$ on the left and $E^h(t) - E^h(0)$ on the right, of kink solitons with different degrees $q$ of approximation space in a moving box until the final time $T = 100$. The box is initially put at the region of $x\in[60, 140]$ and then moves to the right with the same speed as the kink solitons.}\label{fig:solutions_energy_eg31}
\end{figure}

\section{Conclusion}\label{sec:conclusion}
In this work, a DG method is developed and analyzed for a first-order semi-linear hyperbolic system modeling a dimer lattice. The stability of the semi-discrete DG formulation is proved and demonstrated for general mesh-independent numerical fluxes; moreover, we also show optimal $L^2$-error estimates for the upwind fluxes with periodic boundary conditions. Our numerical experiments verify the theoretical findings. In addition, when using mixed fluxes, we also numerically observed the optimal convergence in the $L^2$ norm; for the central fluxes, we observe a sub-optimal convergence for polynomials of odd degrees, but a super-convergence for polynomials of even degrees. This phenomenon is common in literature when central fluxes are used; indeed, Bassi et al. \cite{bassi1997high} reported an optimal convergence order $q+1$ for even values of $q$ and a suboptimal convergence order $q$ for the odd value of $q$ for a purely elliptic problem with uniform grids. Cockburn and Shu \cite{cockburn1998local} also presented the same results for a purely parabolic problem; see their Table 5. In addition, Delfour et al. \cite{delfour1981discontinuous} reported this phenomenon in the framework of ODEs, see their Table A. Applications to kink solitons have demonstrated the effectiveness of high-order numerical schemes. Our main target for future work will be extensions to problems in multi-dimensional space. This will enable applications to a wider variety of problems of physical interest. We will also devote our effort to the development of an energy-conserving fully-discrete DG scheme for the proposed problem, which will be more favorable when simulating waves with long time integration since they can maintain the phase and shape of the waves accurately.

\section*{Funding}
This work is partially supported by US National Science Foundation grant DMS-2012562 (QD), DMS-1937254 (QD and MW), DMS-1620418 (MW and HL), DMS-1908657 (MW and HL), and Simons Foundation Math + X Investigator Award \#376319 (MW).

\bibliography{ref}
\bibliographystyle{plain}

\end{document}